\documentclass[oneside,a4paper,12pt,centertags]{amsart}
\usepackage{vmargin}
\usepackage{amssymb, amsthm, amscd, latexsym, amsmath, verbatim}
\usepackage[all]{xy}
\usepackage[mathscr]{eucal}
\usepackage{mathrsfs}
\usepackage[dvips]{color}
\usepackage[bookmarks=true]{hyperref}
\usepackage{pstricks,pst-plot}

\theoremstyle{plain}
\newtheorem{Thm}{Theorem}[section]

\theoremstyle{definition}
\newtheorem{Prop}[Thm]{Proposition}
\newtheorem{Lem}[Thm]{Lemma}
\newtheorem{Cor}[Thm]{Corollary}
\newtheorem{Def}[Thm]{Definition}

\newtheorem{Eg}[Thm]{Example}

\newtheorem*{Not1}{Notation}

\theoremstyle{remark}
\newtheorem{Rem}[Thm]{Remark}

\newcommand{\R}{\mathbb{R}}
\newcommand{\C}{\mathbb{C}}
\newcommand{\N}{\mathbb{N}}
\newcommand{\Z}{\mathbb{Z}}
\newcommand{\re}{\mathbb{\text{Re}}}
\newcommand{\supp}{\mathop{\text{sppt}}}
\newcommand{\esss}{\mathop{\text{ess sup}\,}}
\newcommand{\ep}{\epsilon}
\newcommand{\al}{\alpha}
\newcommand{\Dom}{\textsf{D}}
\newcommand{\Ran}{\textsf{R}}
\newcommand{\Nul}{\textsf{N}}

\newcommand{\sgn}{\mathop{\text{sgn}}}
\renewcommand{\d}{\mathrm{d}}
\newcommand{\comp}[1]{{}^{c}{#1}}
\newcommand{\Pf}{\mathcal{P}_{\!\!f}}
\newcommand{\svt}[1]{{\substack{\vspace{.025cm} \\ #1}}} \newcommand{\svtt}[1]{{\substack{\vspace{.008cm} \\ #1}}}
\newcommand{\lint}{\mathop{\lrcorner}}
\newcommand{\ext}{\mathop{\scriptstyle{\wedge}}}
\newcommand\ca{{\bf 1}}
\newcommand\cA{\mathcal{A}}
\newcommand\cC{\mathcal{C}}
\newcommand\cD{\mathcal{D}}
\newcommand\cM{\mathcal{M}}
\newcommand\cP{\mathcal{P}}
\newcommand\cQ{\mathcal{Q}}
\newcommand\cS{\mathcal{S}}

\numberwithin{equation}{section}

\begin{document}
\title[Local Hardy Spaces]{Local Hardy Spaces of Differential Forms on Riemannian Manifolds}
\author[Carbonaro]{Andrea Carbonaro}
\author[McIntosh]{Alan McIntosh}
\author[Morris]{Andrew J. Morris}

\address{Andrea Carbonaro \\University of Birmingham\\ School of Mathematics\\ Watson Building\\  Edgbaston\\ Birmingham B15 2TT\\ United Kingdom}
\email{A.Carbonaro@bham.ac.uk}

\address{Alan McIntosh \\ Centre for Mathematics and its Applications\\ Mathematical Sciences Institute\\ Australian National University\\ Canberra\\ ACT 0200\\ Australia}
\email{Alan.McIntosh@anu.edu.au}

\address{Andrew J. Morris \\ Centre for Mathematics and its Applications\\ Mathematical Sciences Institute\\ Australian National University\\ Canberra\\ ACT 0200\\ Australia}
\email{Andrew.Morris@anu.edu.au}


\date{18 April 2011}

\begin{abstract}
We define local Hardy spaces of differential forms $h^p_{\mathcal D}(\wedge T^*M)$ for all $p\in[1,\infty]$ that are adapted to a class of first order differential operators $\mathcal D$ on a complete Riemannian manifold $M$ with at most exponential volume growth. In particular, if $D$ is the Hodge--Dirac operator on $M$ and $\Delta=D^2$ is the Hodge--Laplacian, then the local geometric Riesz transform ${D(\Delta+aI)^{-{1}/{2}}}$ has a bounded extension to $h^p_D$ for all $p\in[1,\infty]$, provided that $a>0$ is large enough compared to the exponential growth of $M$. A characterisation of $h^1_{\mathcal D}$ in terms of local molecules is also obtained. These results can be viewed as the localisation of those for the Hardy spaces of differential forms $H^p_D(\wedge T^*M)$ introduced by Auscher, McIntosh and Russ.
\end{abstract}

\maketitle

\tableofcontents

\section{Introduction and Main Results}\label{SectionIntroduction}
The local Hardy space $h^1(\R^n)$ introduced by Goldberg in \cite{Goldberg01} is an intermediate space $H^1(\R^n)\subset h^1(\R^n) \subset L^1(\R^n)$. The Hardy space $H^1(\R^n)$ is suited to quasi-homogenous multipliers, and indeed the boundedness of the Riesz transforms $(R_ju)\, \hat{\ }(\xi)=i{\xi_j}{|\xi|^{-1}}\hat{u}(\xi)$ is built into its definition. The local Hardy space $h^1(\R^n)$, however, is suited to \textit{smooth} quasi-homogenous multipliers, and the boundedness of local Riesz transforms, such as $(r_ju)\, \hat{}(\xi) =i\xi_j(|\xi|^2+a)^{-1/2}\hat{u}(\xi)$ for $a>0$, is built into its definition.

Let $M$ denote a complete Riemannian manifold with geodesic distance $\rho$ and Riemannian measure $\mu$. We adopt the convention that such a manifold $M$ is smooth and connected. Let $L^2(\wedge T^*M)$ denote the Hilbert space of square-integrable differential forms on $M$. Let $d$ and $d^*$ denote the exterior derivative and its adjoint on $L^2(\wedge T^*M)$. The Hodge--Dirac operator is $D=d+d^*$ and the Hodge--Laplacian is $\Delta=D^2$. The geometric Riesz transform $D\Delta^{-{1}/{2}}$ is bounded on $L^2(\wedge T^*M)$, which led Auscher, McIntosh and Russ in \cite{AMcR} to construct Hardy spaces of differential forms $H^p_D(\wedge T^*M)$, or simply $H^p_D$, for all $p\in[1,\infty]$. Amongst other things, they show that the geometric Riesz transform is bounded on $H^p_D$ and that $H^1_D$ has a molecular characterisation.

The atomic characterisation of $H^1(\R^n)$, due to Coifman \cite{Coifman01} and Latter \cite{Latter01}, was used by Coifman and Weiss in \cite{CW2} to define a Hardy space of functions on a space of homogeneous type. A requirement in the definition of Hardy space atoms $a$ is that they satisfy the moment condition $\int a=0$. The approach taken in \cite{AMcR} is instead based on the connection between the tent spaces $T^{p}(\R^{n+1}_+)$ and $H^p(\R^n)$. This connection was first recognised by Coifman, Meyer and Stein who showed in Section 9B of \cite{CMS} that $H^p(\R^n)$ is isomorphic to a complemented subspace of $T^{p}(\R^{n+1}_+)$ for all $p\in[1,\infty]$. More precisely, there exist two bounded operators $P:H^p(\R^n)\rightarrow T^{p}(\R^{n+1}_+)$ and $\pi:T^{p}(\R^{n+1}_+)\rightarrow H^p(\R^n)$ such that $P\pi$ is a projection and $H^p(\R^n)$ is isomorphic to $P\pi(T^{p}(\R^{n+1}_+))$.

The definition of the tent space $T^1(\R^{n+1}_+)$ and its atoms, which are not required to satisfy a moment condition, admit natural generalisations to differential forms. Also, both $P$ and $\pi$ are convolution-type operators, which can be interpreted in terms of the functional calculus of $-i \d/\d x$. The idea in \cite{AMcR} was to define $H^p_D$ in terms of the tent space of differential forms $T^p(\wedge T^*M\times(0,\infty))$ and operators $\cQ$ and $\cS$, which are adapted to $D$ in the same way that $P$ and $\pi$ are adapted to~$-i \d/\d x$. The main requirement for the construction was that operators such as the projection $\cQ\cS$ be bounded on $T^p(\wedge T^*M\times(0,\infty))$. The authors of \cite{AMcR} prove this by using off-diagonal estimates for the resolvents of $D$ to establish uniform bounds on tent space atoms.

In this paper, we adapt the definition of $H^p_D$ to define local Hardy spaces of differential forms $h^p_D(\wedge T^*M)$, or simply $h^p_D$, for all $p\in[1,\infty]$. We first consider a general locally doubling metric measure space $X$, and define a local tent space $t^p(X\times(0,1])$ and a new function space $L^p_{\mathscr{Q}}(X)$, both of which have an atomic characterisation for $p=1$ and admit a natural generalisation to differential forms. One can show classically that $h^p(\R^n)$ is isomorphic to a complemented subspace of $t^p(\R^{n}\times(0,1])\oplus L^p_{\mathscr{Q}}(\R^n)$. Whilst square function characterisations for $h^p(\R^n)$ are certainly known, this characterisation appears to be new.

The atomic characterisation of $h^1(\R^n)$, due to Goldberg \cite{Goldberg01}, consists of two types of atoms. The first kind are supported on balls of radius less than one and satisfy a moment condition, whilst the second kind are supported on balls with radius larger than one and are not required to satisfy a moment condition. In our new characterisation, we can associate the first kind of atoms with elements of ${t^1(\R^{n}\times(0,1])}$ and the second kind with elements of $L^1_{\mathscr{Q}}(\R^n)$.

The definition of $H^p_D$ in \cite{AMcR} is limited to Riemannian manifolds that are doubling, which we define below using the following notation. Given $x\in M$ and $r>0$, let $B(x,r)$ denote the open geodesic ball in $M$ with centre $x$ and radius $r$, and let $V(x,r)$ denote the Riemannian measure $\mu(B(x,r))$.
\begin{Def}\label{DefD}
A complete Riemannian manifold $M$ is \textit{doubling} if there exists $A\geq1$ such that
\begin{equation}\tag{D}\label{D}
0<V(x,2r)\leq A V(x,r)<\infty
\end{equation}
for all $x\in X$ and $r>0$.
\end{Def}

The doubling condition is equivalent to the requirement that there exist $A\geq1$ and $\kappa\geq0$ such that
\begin{equation*}
0<V(x,\alpha r)\leq A\alpha^{\kappa} V(x,r)<\infty
\end{equation*}
for all $x\in X$, $r>0$ and $\alpha\geq1$. This condition is imposed to define $H^p_D$ because the Hardy space norm incorporates global geometry. The nature of the local Hardy space, however, allows us to define $h^p_D$ on manifolds that are only locally doubling. Specifically, we define $h^p_D$ on the following class of manifolds.
\begin{Def}\label{DefExpLD}
A complete Riemannian manifold $M$ is \textit{exponentially locally doubling} if there exist $A\geq1$ and $\kappa,\lambda\geq0$ such that
\begin{equation}\label{ELD}\tag{E$_{\kappa,\lambda}$}
0<V(x,\alpha r)\leq A\alpha^\kappa e^{\lambda(\alpha-1) r} V(x,r)<\infty
\end{equation}
for all $\alpha\geq1$, $r>0$ and $x\in M$. The constants $\kappa$ and $\lambda$ are referred to as the \textit{polynomial} and \textit{exponential growth parameters}, respectively.
\end{Def}

The class of doubling Riemannian manifolds includes $\R^n$ with the Euclidean distance and the standard Lebesgue measure, as well as Lie groups with polynomial volume growth; other examples are listed in \cite{AMcR}. The class of exponentially locally doubling Riemannian manifolds is larger and includes hyperbolic space (see Section 3.H.3 of \cite{GHL}), Lie groups with exponential volume growth (see Section II.4 of \cite{DtER}) and thus all Lie groups. More generally, by Gromov's variant of the Bishop comparison theorem (see \cite{BishopCrittenden64,Gromov81}), all Riemannian manifolds with Ricci curvature bounded from below are exponentially locally doubling. This includes Riemannian manifolds with bounded geometry, noncompact symmetric spaces and Damek--Ricci spaces.

Taylor recently defined local Hardy spaces of functions on Riemannian manifolds with bounded geometry in \cite{Taylor}. Hardy spaces of functions have also been defined on some nondoubling metric measure spaces in \cite{CaMaMe09,CaMaMe10}; extensions of that work are also in \cite{CarbonaroMauceriMeda2010,MauceriMedaVallarino2010(1), MauceriMedaVallarino2010(2)}. The theory developed in those papers applies to $\R^n$ with the Euclidean distance and the Gaussian measure, as well as to Riemannian manifolds on which the Ricci curvature is bounded from below and the Laplace--Beltrami operator has a spectral gap.

The Hardy spaces $H^p_D$ in \cite{AMcR} are defined using the holomorphic functional calculus of $D$ in $L^2(\wedge T^* M)$. In particular, the authors consider the class $H^\infty(S^o_\theta)$ of functions that are bounded and holomorphic on the open bisector $S^o_\theta$ of angle ${\theta\in(0,\pi/2)}$ centered at the origin in the complex plane. This is because the function $\sgn(\re(z))=z/\sqrt{z^2}$ maps $D$ to the geometric Riesz transform $D\Delta^{-{1}/{2}}$ under the $H^\infty(S^o_\theta)$ functional calculus. The local Hardy spaces, however, are suited to the local geometric Riesz transforms $D(\Delta+aI)^{-{1}/{2}}$ for $a>0$, so we consider the smaller class $H^\infty(S^o_{\theta,r})$ of functions that are bounded and holomorphic on $S^o_\theta\cup D^o_r$, where $D^o_r$ is the open disc of radius $r>0$ centered at the origin in the complex plane.

The space $h^1_D$ has a characterisation in terms of local molecules, which are defined in Section \ref{SubSection2ndmainresult}. This is the first main result of the paper.
\begin{Thm}\label{Thm: Intro.Molecules}
Let $\kappa,\lambda\geq0$ and suppose that $M$ is a complete Riemannian manifold satisfying \eqref{ELD}. If $N\in\N$, $N>\kappa/2$ and $q\geq\lambda$, then
$h^1_D=h^1_{D,\text{mol}(N,q)}$.
\end{Thm}

The following is the principal result of the paper.

\begin{Thm}\label{Thm: Intro.FuncCalc}
Let $\kappa,\lambda\geq0$ and suppose that $M$ is a complete Riemannian manifold satisfying \eqref{ELD}. Let $\theta\in(0,\pi/2)$ and $r>0$ such that $r\sin\theta>\lambda/2$. Then for all $f\in H^\infty(S_{\theta,r}^o)$, the operator $f(D)$ on $L^2$ has a bounded extension to $h^p_D$ such that
\[
\|f(D)u\|_{h^p_D}\lesssim \|f\|_\infty \|u\|_{h^p_D}
\]
for all $u\in h^p_D$ and $p\in[1,\infty]$.
\end{Thm}

There is then the following corollary for the local geometric Riesz transforms.
\begin{Cor}\label{Cor: Intro.Riesz}
Let $\kappa,\lambda\geq0$ and suppose that $M$ is a complete Riemannian manifold satisfying \eqref{ELD}. If $a>\lambda^2/4$, then the local geometric Riesz transform ${D(\Delta+aI)^{-1/2}}$ has a bounded extension to $h^p_D$ for all $p\in[1,\infty]$.
\end{Cor}

The theory in this paper actually applies to a large class of first order differential operators, which we introduce in Section \ref{SectionODE}. These operators are denoted by $\cD$. Theorems \ref{Thm: Intro.Molecules} and~\ref{Thm: Intro.FuncCalc} follow from the more general results in Theorems \ref{mainmoleculeresult} and~\ref{mainresult} by setting $\cD=D$, where $D$ will always denote the Hodge--Dirac operator. The case of the Hodge--Dirac operator is considered in Example~\ref{Eg: HodgeDirac}.

Taylor proved in \cite{Taylor89} that on a Riemannian manifold with bounded geometry, where $\Delta_0$ denotes the Hodge--Laplacian on functions, a sufficient condition for the operator $f(\sqrt{\Delta_0})$ to be bounded on $L^p$ for all $p\in(1,\infty)$ is that $f$ be holomorphic and satisfy inhomogeneous Mihlin boundary conditions on an open strip of width $W\geq \lambda/2$ in the complex plane, where $\lambda\geq0$ is such that $\eqref{ELD}$ holds. This result was improved by Mauceri, Meda and Vallarino in \cite{MauceriMedaVallarino2009}, and then by Taylor in \cite{Taylor2009}. The need for $f$ to be holomorphic on a strip was originally noted by Clerc and Stein in the setting of noncompact symmetric spaces in \cite{ClercStein74}, and that work was extended by others in \cite{AnkerLohoue86, CheegerGromovTaylor82, StantonTomas78}.

In this paper, we do not assume bounded geometry. Theorem~\ref{Thm: Intro.FuncCalc} represents the beginning of the development of an approach to the theory discussed above based on first order operators. The connection between $h^p_\cD$ and $L^p$ is investigated in the sequel, although note that $h^2_\cD$ is defined so that it can be identified with $L^2$. Also, Taylor's result suggests that the bound $r\sin\theta> \lambda/2$ in Theorem~\ref{Thm: Intro.FuncCalc} may be the best possible, since $r\sin\theta$ is the width of the largest open strip contained in $S^o_{\theta,r}$.

For all $x,y\in\R$, we adopt the convention whereby $x\lesssim y$ means that there exists a constant $c\geq1$, which may only depend on constants specified in the relevant preceding hypotheses, such that $x\leq cy$. To emphasize that the constant $c$ depends on a specific parameter $p$, we write $x \lesssim_p y$. Also, we write $x\eqsim y$ to mean that $x\lesssim y \lesssim x$. For all normed spaces $X$ and $Y$, we write $X \subseteq Y$ to mean both the set theoretical inclusion and the topological inclusion, whereby $\|x\|_Y \lesssim \|x\|_X$ for all $x\in X$. Also, we write $X = Y$ to mean that $X$ and $Y$ are equal as sets and that they have equivalent norms.

The structure of the paper is as follows: In Section 2 we develop local analogues of some basic tools from harmonic analysis in the context of a locally doubling metric measure space $X$. The local tent spaces $t^p(X\times(0,1])$ and the new spaces $L^p_{\mathscr{Q}}(X)$ are introduced and shown to have atomic characterisations for $p=1$ in Sections \ref{SectionLocalTent} and \ref{SectionNewFunctionSpaces}, respectively. We also obtain duality and interpolation results for these spaces. Next, we introduce a general class of first order differential operators, which includes the Hodge--Dirac operator. We denote these operators by $\cD$ and prove exponential off-diagonal estimates for their resolvents in Section \ref{SectionODE}. These are used to prove the main technical estimate in Section \ref{SectionTheMainEstimate}, which allows us to define the local Hardy spaces of differential forms $h^p_{\cD}(\wedge T^*M)$ in Section \ref{Sectionh^1defintion}. We also obtain duality and interpolation results for the local Hardy spaces. Finally, Theorems \ref{Thm: Intro.Molecules} and \ref{Thm: Intro.FuncCalc} follow from Theorems \ref{mainmoleculeresult} and \ref{mainresult} in the case of the Hodge--Dirac operator.

\section{Localisation}\label{SectionLocalisation}
The subsequent two sections do not require a differentiable structure. To distinguish these results, it is convenient to let $X$ denote a metric measure space with metric $\rho$ and Borel measure $\mu$.

\begin{Not1}
A ball in $X$ will always refer to an open metric ball. Given $x\in X$ and $r>0$, let $B(x,r)$ denote the ball in $X$ with centre $x$ and radius $r$, and let $V(x,r)$ denote the measure $\mu(B(x,r))$. Given $\alpha,r>0$ and a ball $B$ of radius $r$, let $\alpha B$ denote the ball with the same centre as $B$ and radius $\alpha r$.
\end{Not1}

The results here and in the next section hold if we assume the following local doubling condition.

\begin{Def}\label{DefLD} A metric measure space $X$ is \textit{locally doubling} if for each $r>0$, the function $x\mapsto V(x,r)$ is continuous on $X$, and if for each $b>0$, there exists $A_b\geq1$ such that
\begin{equation}\tag{D$_\text{loc}$}\label{LD}
0<V(x,2r)\leq A_b V(x,r)<\infty
\end{equation}
for all $x\in X$ and $r\in(0,b]$.
\end{Def}

\begin{Rem} The continuity of $x\mapsto V(x,r)$ is assured on a complete Riemannian manifold and in most applications, but in general only lower semicontinuity is guaranteed. We require this condition because it implies that the volumes of open balls and closed balls are identical (see also Remark~\ref{Alocmsblty}). \end{Rem}

If $\sup_b A_b<\infty$, then \eqref{LD} is equivalent to \eqref{D} in Definition~\ref{DefD}. In fact, the doubling condition was introduced by Coifman and Weiss in \cite{CW} to define a space of homogeneous type. The results here are a localised version of that work. We begin by proving the following useful consequence of local doubling.

\begin{Prop}\label{PropkappaLD}
If $X$ is locally doubling, then for each $b>0$ there exists $\kappa_b\geq 0$ such that
\[V(x,\alpha r)\leq A_b\alpha^{\kappa_b} V(x,r)\]
for all $x\in X$, $r\in(0,b]$ and $\alpha\in[1,2b/r]$.
\end{Prop}
\begin{proof} Let $N=\lceil\log_2 \alpha\rceil$, which is the smallest integer not less than $\log_2 \alpha$, so that $2^{N-1}<\alpha\leq2^N$ and $B(x,\tfrac{\alpha}{2^N}r) \subseteq B(x,r)$. Application of the \eqref{LD} inequality $N$~times reveals that
\[V(x,\alpha r) \leq A_b^N V(x,\tfrac{\alpha}{2^N}r) \leq A_b\alpha^{\kappa_b} V(x,r),\]
where $\kappa_b=\log_2 A_b$.
\end{proof}

We introduce the local property of homogeneity, which is the local analog of the property of homogeneity from \cite{CW}, and show that it holds on a locally doubling space. This property allows us to apply harmonic analysis locally on $X$.
\begin{Def}
A metric space $(X,\rho)$ has the \textit{local property of homogeneity} if for each $b>0$ there exists $N_b\in\N$ such that for all $x\in X$ and $r\in(0,b]$, the ball $B(x,r)$ contains at most $N_b$ points $(x_j)_{j=1,...,N_b}$ satisfying $\rho(x_j,x_k)\geq{r}/{2}$ for all $j\neq k$.
\end{Def}

\begin{Rem} \label{CW1.1loc}
The local property of homogeneity is equivalent to the requirement that if $b>0$, then for all $x\in X$, $r\in(0,b]$ and $n\in\N$, the ball $B(x,r)$ contains at most $N_b^n$ points $(x_j)_{j=1,...,N_b^n}$ satisfying $\rho(x_j,x_k)\geq{r}/{2^n}$ for $j\neq k$. The proof of this is similar to that of Lemma~1.1 in Chapter III of \cite{CW}. This property is more suited to applications. It can be used, for instance, to prove the next proposition.
\end{Rem}

\begin{Prop} \label{DoublingHomogloc}
If $X$ is a locally doubling metric measure space, then it has the local property of homogeneity.
\end{Prop}
\begin{proof}
This follows the proof of the Remark on page~67 in Chapter III of~ \cite{CW}.
\end{proof}

Following the scheme of \cite{CW}, we use the local property of homogeneity to prove local covering lemmas. The next two proofs are adapted from \cite{Aimar}, which treats the global case.

\begin{Prop}(Vitali-Wiener type covering lemma.) \label{Aimar2.11}
Let $X$ be a metric space with the local property of homogeneity. Let $\mathscr{B}$ be a collection of balls in $X$. If there is a finite upper bound on the radii of the balls in $\mathscr{B}$, then there exists a sequence $(B_j)_j$ of pairwise disjoint balls in $\mathscr{B}$ with the property that each $B\in\mathscr{B}$ is contained in some $4B_j$.
\end{Prop}
\begin{proof}
Fix $R>0$ such that the radii $r(B)\leq R$ for all $B\in\mathscr{B}$. Let $\delta\in(0,1)$ to be fixed later, and for each $k\in\N$ define
\[\mathscr{B}_k = \{B\in\mathscr{B}\ |\ \delta^kR < r(B)\leq \delta^{k-1}R\}.\]
Proceeding recursively for $k=1,2,\dots$, choose a maximal subset $\tilde{\mathscr{B}}_k$ of pairwise disjoint balls in $\mathscr{B}_k$ according to the following requirements:
\begin{enumerate}
\item $\tilde{\mathscr{B}}_k \subseteq \mathscr{B}_k$;
\item If $B,B'\in \bigcup_{j=1}^k \tilde{\mathscr{B}}_j$ and $B\neq B'$, then $B\cap B' = \emptyset$;
\item If $B\in\mathscr{B}_k\setminus\tilde{\mathscr{B}}_k$, then there exists $B'\in\bigcup_{j=1}^k\tilde{\mathscr{B}}_j$ such that $B\cap B'\neq\emptyset$.
\end{enumerate}
To show that each $\tilde{\mathscr{B}}_k$ is countable, choose $B_0\in\tilde{\mathscr{B}}_k$ and write
\[\tilde{\mathscr{B}}_k = \textstyle{\bigcup_{n\in\N}} \{B\in\tilde{\mathscr{B}}_k\ |\ B\subseteq nB_0\}.\]
For each $n\in\N$, the centres of all of the balls in $\{B\in\tilde{\mathscr{B}}_k\ |\ B\subseteq~ nB_0\}$ are separated by at least a distance of $\delta^k R$ and contained in a ball of radius $nR$, so countability follows by the local property of homogeneity.
Therefore, the collection $\tilde{\mathscr{B}}=\bigcup_{k} \tilde{\mathscr{B}_k}$ is a sequence $(B_j)_j$ of pairwise disjoint balls in $\mathscr{B}$.

To complete the proof, let $B\in\mathscr{B}\setminus\tilde{\mathscr{B}}$. For some $k\in\N$, we have $B\in \mathscr{B}_k\setminus\tilde{\mathscr{B}}_k$ and there exists $B'\in\bigcup_{j=1}^k\tilde{\mathscr{B}}_j$ such that $B\cap B'\neq\emptyset$. In particular, we have $B'\in\tilde{\mathscr{B}}_{k'}$ for some $k'\leq k$, so if $x'$ denotes the centre of $B'$, then
\[
\rho(y,x') \leq 2r(B)+r(B') \leq 2\delta^{k'-1}R + r(B') \leq (2/\delta+1)r(B')
\]
for all $y\in B$. If we set $\delta=2/3$, then $B\subseteq 4B'$ and the proof is complete.
\end{proof}

\begin{Prop}(Whitney type covering lemma.) \label{Aimar2.13 + MS2.16}
Let $X$ be a metric space with the local property of homogeneity. Let $O$ be a nonempty proper open subset of $X$ and let $\comp{O} = X \setminus O$. For each $h>0$, there exists a sequence of pairwise disjoint balls $(B_j)_j$ with centre $x_j\in X$ and radius
\[r_j = \tfrac{1}{8}\min(\rho(x_j,\comp{O}),h)\]
such that, if $\tilde B_j = 4B_j$, then $O=\bigcup_j \tilde{B}_j$ and the following bounded intersection property is satisfied:
\[\sup_{j}\sharp\big(\{k\ |\ \tilde B_j\cap \tilde B_k \neq \emptyset\}\big) < \infty.\]
Furthermore, there exists a sequence $(\phi_j)_j$ of nonnegative functions supported in $\tilde B_j$ such that $\inf_{x\in B_j} \phi_j(x)>0$ and $\sum_j \phi_j = \ca_O$, where $\ca_O$ denotes the characteristic function of $O$.
\end{Prop}
\begin{proof}
Let $\mathscr{B}$ denote the collection of all balls with centre $x\in O$ and radius $r=\frac{1}{8} \min(\rho(x,\comp{O}),h)$. Proposition~\ref{Aimar2.11} gives a sequence $(B_j)_j=(B(x_j,r_j))_j$ of pairwise disjoint balls from $\mathscr{B}$ such that $O\subseteq\bigcup_j \tilde{B}_j$, and since $4r_j<\rho(x_j,\comp{O})$, we actually have $O=\bigcup_j \tilde{B}_j$.

We note some facts to help prove that $(\tilde{B}_j)_j$ has the bounded intersection property. First, if $x\in \tilde{B}_j$, then
\begin{equation}\label{eq: bddint1}
\rho(x,\comp{O}) \geq \rho(x_j,\comp{O})-\rho(x_j,x)
\geq 8r_j- 4r_j
= 4 r_j.
\end{equation}
Second, given $c>0$, if $x\in \tilde B_j$ and $\rho(x_j,\comp{O}) \leq c r_j$, then
\begin{equation}\label{eq: bddint2}
\rho(x,\comp{O}) \leq \rho(x,x_j)+\rho(x_j,\comp{O})
\leq (4+c)r_j.
\end{equation}

Now suppose that $\tilde B_j\cap \tilde B_k \neq \emptyset$. This implies that
\begin{equation}\label{G}
\rho(x_j,x_k)
\leq 4(r_j+r_k)
\leq h.
\end{equation}
Consider two cases: (1) If $\rho(x_j,\comp{O}) > 2h$, then by \eqref{G} we have \[\rho(x_k,\comp{O}) \geq \rho(x_j,\comp{O})-\rho(x_j,x_k)>h,\] so
$r_k=h/8=r_j$ and $B_k \subseteq 9B_j$; (2)
If $\rho(x_j,\comp{O}) \leq 2h$, then by \eqref{G} we have
\[\rho(x_k,\comp{O}) \leq \rho(x_k,x_j)+\rho(x_j,\comp{O}) \leq 3h,\]
which implies that $\rho(x_k,\comp{O}) \leq 24r_k$, since either $\rho(x_k,\comp{O})=8r_k$ or $h=8r_k$. In this case, if $x \in \tilde B_j\cap \tilde B_k$, then by \eqref{eq: bddint1} and \eqref{eq: bddint2} with $c=24$ we obtain
\[
4 r_j \leq \rho(x,\comp{O}) \leq 28 r_j
\quad\text{and}\quad
4 r_k \leq \rho(x,\comp{O}) \leq 28 r_k,
\]
so $(1/7)r_j \leq r_k \leq 7r_j$ and $B_k \subseteq 39B_j$.

The above shows that for each $j\in\N$, the centres of all balls $\tilde B_k$ satisfying $\tilde B_j\cap \tilde B_k \neq \emptyset$ are separated by at least a distance of $(1/7)r_j$ and contained in a ball of radius $39r_j\leq5h$. The bounded intersection property then follows from the local property of homogeneity.

To construct the sequence of functions $(\phi_j)_j$, let $\eta$ be the function equal to 1 on [0,1) and 0 on $[1,\infty)$. For each $j\in\N$, define
\[
\psi_j(x)= \eta\left(\frac{\rho(x,x_j)}{4r_j}\right)
\]
for all $x\in X$. These are nonnegative functions supported in $\tilde B_j$. We also have $1\leq \sum_j \psi_j(x) < \infty$ for all $x\in O$, since $O=\bigcup_j \tilde{B}_j$ and the bounded intersection property is satisfied. The required functions are then defined for each $j\in\N$ by
\[\phi_j(x)= \begin{cases}
\psi_j(x)/\sum_k\psi_ k(x), &\text{if}\quad x\in O; \\
0, &\text{if}\quad x\in \comp{O}.
\end{cases}\]
\end{proof}

We now prove a general version of the fundamental theorem for the (centered) \textit{local maximal operator} $\mathcal{M}_{\text{loc}}$ defined for all measurable functions $f$ on $X$ by
\[
\mathcal{M}_\text{loc}f(x)=\sup_{r\in(0,1]} \frac{1}{V(x,r)}\int_{B(x,r)} |f(y)|\ \d\mu(y)
\]
for all $x\in X$.

\begin{Prop}\label{CW2.1loc}
Let $X$ be a locally doubling metric measure space. If $f$ is a measurable function on $X$, then $\mathcal{M}_{\text{loc}}f$ is lower semicontinuous, and thus measurable, and the following hold:
\begin{enumerate}\setlength{\itemsep}{3pt}
\item If $\alpha>0$, then
$\mu\left(\{x\in X\ |\ \mathcal{M}_\text{loc}f(x)>\alpha\}\right) \lesssim {\|f\|_1}/{\alpha}$
for all $f\in L^1(X)$;
\item If $1<p\leq\infty$, then
$\|\mathcal{M}_\text{loc}f\|_p \lesssim_p \|f\|_p$
for all $f\in L^p(X)$.
\end{enumerate}

\begin{proof}
The lower semicontinuity of $\cM_{\text{loc}}f$ is guaranteed by Fatou's Lemma and the continuity of the mapping $x\mapsto V(x,r)$ from Definition~\ref{DefLD}.

To prove (1), let $f\in L^1(X)$ and set $E_\alpha=\{x\in X\ |\ {\mathcal{M}}_\text{loc} f(x)>\alpha\}$ for each $\alpha>0$. If $x\in E_\alpha$, then there exists $r_x\in(0,1]$ such that
\[\frac{1}{V(x,r_x)}\int_{B(x,r_x)} |f(y)|\ \d\mu(y)>\alpha.\]
By Proposition~\ref{Aimar2.11}, the collection $\mathscr{B}=(B(x,r_x))_{x\in E_\alpha}$ contains a subsequence $(B_j)_j$ of pairwise disjoint balls such that, if $\tilde B_j=4B_j$, then $(\tilde{B}_j)_j$ cover $E_\alpha$. Therefore, by \eqref{LD} we have
\begin{align*}
\int_X |f(y)|\ \d\mu(y) \geq \sum_j \int_{B_j}|f(y)|\ \d\mu(y)
> \alpha \sum_j \mu(B_j)
\gtrsim \alpha \mu(E_\alpha).
\end{align*}

The proof of (2) is then standard (see, for instance, Section I.1.5 of \cite{St2}).
\end{proof}
\end{Prop}

We conclude this section by proving that a locally doubling space is exponentially locally doubling, as in Definition~\ref{DefExpLD}, if and only if it satisfies a certain additional condition on volume growth. Whilst we do not make explicit use of this equivalence, it shows why \eqref{ELD} is often a more useful assumption than \eqref{LD}. In particular, it allows us to obtain the atomic characterisation of the space $L^1_{\mathscr{Q}}(X)$ in Section~\ref{SectionNewFunctionSpaces}.

\begin{Prop}\label{Prop: ELDequivDgloDloc}
Let $X$ be a locally doubling metric measure space. Then $X$ is exponentially locally doubling if and only if there exist $A_0\geq1$ and $b_0,\delta>0$ such that
\begin{equation}\tag{D$_{\text{glo}}$}\label{U}
V(x,r+\delta)\leq A_0 V(x,r)
\end{equation}
for all $r\geq b_0$ and $x\in X$.
\end{Prop}
\begin{proof}
If $X$ satisfies \eqref{ELD}, then for any $b_0>0$ and $\delta>0$, we have
\[V(x,r+\delta) = V(x,(1+\delta/r)r) \leq A(1+\delta/b_0)^{\kappa} e^{\lambda \delta} V(x,r)\]
for all $r\geq b_0$ and $x\in X$.

To prove the converse, suppose $X$ satisfies \eqref{U} and let $\alpha>1$. Consider three cases:

If $r>b_0$, choose $N\in\N$ so that $\alpha r-N\delta\in(r,r+\delta]$. Application of the \eqref{U} inequality $N+1$ times reveals that
\begin{equation}\label{eq: case1}
V(x,\alpha r) \leq A_0^{N+1} V(x,r) \leq A_0e^{\lambda(\alpha-1)r}V(x,r),
\end{equation}
where $\lambda=(\log A_0)/\delta$;

If $r\in(0,b_0]$ and $\alpha\in(1,2b_0/r]$, then Proposition~\ref{PropkappaLD} implies that
\begin{equation}\label{eq: case2}
V(x,\alpha r)\leq A_{b_0}\alpha^{\kappa_{b_0}} V(x,r);
\end{equation}

If $r\in(0,b_0]$ and $\alpha>2b_0/r$, then we obtain
\begin{align*}
V(x,\alpha r) &= V(x,(\alpha r/2b_0)2b_0) \\
&\leq A_0 e^{\lambda(\alpha r/2b_0-1)2b_0}V(x,2b_0) \\
&\leq A_0 e^{\lambda(\alpha-1)r}V(x,(2b_0/r) r) \\
&\leq A_0A_{b_0}\alpha^{\kappa_{b_0}}e^{\lambda(\alpha -1)r}V(x,r),
\end{align*}
where we used \eqref{eq: case1} to obtain the first inequality and \eqref{eq: case2} to obtain the final inequality.

\noindent These show that $X$ satisfies \eqref{ELD} with $\kappa=\kappa_{b_0}$ and $\lambda=(\log A_0)/\delta$.
\end{proof}

\section{Local Tent Spaces $t^{p}(X\times(0,1])$}\label{SectionLocalTent}
We introduce the local tent spaces $t^{p}(X\times(0,1])$, or simply $t^p$, for all ${p\in[1,\infty]}$ in the context of a locally doubling metric measure space $X$. Note that functions on $X\times(0,1]$ are assumed to be complex-valued. There is also the following notation.

\begin{Not1}
The cone of aperture $\alpha>0$ and height 1 with vertex at $x\in X$ is
\[
\Gamma_\alpha^1(x) = \{(y,t)\in X\times (0,1]\ |\ \rho(x,y)< \alpha t \}.
\]
Let $\Gamma^1(x)=\Gamma_1^1(x)$. For any closed set $F\subseteq X$ and any open set $O\subseteq X$, define
\[
{R}_\alpha^1(F) = \bigcup_{x\in F} \Gamma_\alpha^1(x)
\quad{\text{ and }}\quad
T^1_\alpha(O) = (X\times(0,1])\setminus R^1_\alpha(\comp{O}),
\]
where $\comp{O} = X \setminus O$. Let $T^1(O)=T^1_1(O)$ and call it the \textit{truncated tent over} $O$. Note that
\[
T^1_\alpha(O) = \{(y,t)\in X\times (0,1]\ |\ \rho(y,\comp{O})\geq \alpha t\}.
\]
For any ball $B$ in $X$ of radius $r(B)>0$, the \textit{truncated Carleson box over} $B$ is \[
C^1(B)=B\times(0,\min\{r(B),1\}].
\]
Finally, if $E$ is a measurable subset of $X\times(0,1]$, then $\ca_E$ denotes the characteristic function of $E$ on $X\times(0,1]$.
\end{Not1}

The local Lusin operator $\mathcal{A}_{\text{loc}}$ and its dual $\mathcal{C}_{\text{loc}}$ are defined for any measurable function $f$ on $X\times~(0,1]$ as follows:
\begin{align*}
\mathcal{A}_{\text{loc}} f(x) &= \bigg(\iint_{\Gamma^1(x)} |f(y,t)|^2\ \frac{\d\mu(y)}{V(x,t)} \frac{\d t}{t}\bigg)^{\frac{1}{2}};\\
\mathcal{C}_{\text{loc}} f(x) &= \sup_{B \in \mathscr{B}_2(x)}\bigg(\frac{1}{\mu(B)}\iint_{T^1(B)}|f(y,t)|^2\ \d\mu(y)\frac{\d t}{t}\bigg)^{\frac{1}{2}}
\end{align*}
for all $x\in X$, where $\mathscr{B}_2(x)$ denotes the set of all balls $B$ in $X$ of radius $r(B)\leq2$ such that $x\in B$.
We now define the local tent spaces.

\begin{Def} Let $X$ be a locally doubling metric measure space. For each $p\in[1,\infty)$, the \textit{local tent space} $t^{p}(X\times(0,1])$ consists of all measurable functions $f$ on $X\times~(0,1]$ with
\[\|f\|_{t^{p}}=\|\mathcal{A}_{\text{loc}}f\|_p < \infty.\]
The \textit{local tent space} $t^{\infty}(X\times(0,1])$ consists of all measurable functions $f$ on $X\times~(0,1]$ with
\[\|f\|_{t^{\infty}}=\|\mathcal{C}_{\text{loc}}f\|_\infty < \infty.\]
\end{Def}

\begin{Rem}\label{Alocmsblty}
Recall that in Definition~\ref{DefLD} we required the continuity of the mapping $x\mapsto V(x,r)$ for each $r>0$. This implies that the volumes of open balls and closed balls are identical, which ensures that $\mathcal{A}_{\text{loc}}f$ and $\mathcal{C}_{\text{loc}}f$ are lower semicontinuous and thus measurable.
\end{Rem}

The local tent spaces are Banach spaces under the usual identification of functions that are equal almost everywhere. This follows as in the global case in \cite{CMS}. In particular, completeness holds by dominated convergence upon noting that for each compact set $K\subseteq X\times(0,1]$ and each $p\in[1,\infty]$, we have
\begin{equation}\label{(1.3)}
\|\ca_Kf\|_{t^p} \lesssim_{K,p} \left(\iint_K |f(y,t)|^2\ \d\mu(y)\d t\right)^{\frac{1}{2}} \lesssim_{K,p} \|f\|_{t^p}
\end{equation}
for all measurable functions $f$ on $X\times(0,1]$.

Let $L^2_\bullet(X\times(0,1])$, or simply $L^2_\bullet$, denote the Hilbert space of all measurable functions $f$ on $X\times(0,1]$ with
\[\|f\|_{L^2_\bullet} = \left(\iint |f(y,t)|^2\ \d\mu(y)\frac{\d t}{t}\right)^{\frac{1}{2}} <\infty.\]
We have $t^2 = L^2_\bullet$, since if $(y,t)\in\Gamma^1(x)$, then $t\leq1$,
$B(y,t)\subseteq B(x,2t)$ and $B(x,t)\subseteq B(y,2t)$, so by \eqref{LD} we obtain
\[
\|f\|_{t^{2}}^2 \eqsim \iiint \ca_{\Gamma^1(x)}(y,t)\ |f(y,t)|^2\ \frac{\d\mu(y)}{V(y,t)}\frac{\d t}{t}\d\mu(x)
=\|f\|_{L_{\bullet}^2}^2.
\]

These observations lead us to the following density result, which is crucial to the extension procedure in Section \ref{SectionTheMainEstimate}.
\begin{Prop}\label{Prop: tpt2denseintp}
Let $X$ be a locally doubling metric measure space. For all $p\in[1,\infty)$ and $q\in[1,\infty]$, the set $t^p\cap t^q$ is dense in $t^p$.
\end{Prop}
\begin{proof}
Let $f\in t^p$ and $p\in[1,\infty)$. Fix a ball $B$ in $X$ and define
\[
f_k= \ca_{kB\times[{1}/{k},1]}f
\]
for each $k\in\N$. The functions $f_k$ belong to $t^p\cap t^q$ for all $q\in[1,\infty]$ by \eqref{(1.3)}, and $\lim_{k\rightarrow\infty} \|f-f_k\|_{t^p} = 0$ by dominated convergence.
\end{proof}

We characterise $t^{1}$ in terms of the following atoms.

\begin{Def}\label{DefTentAtom} Let $X$ be a locally doubling metric measure space. A $t^{1}$-\textit{atom} is a measurable function $a$ on $X\times(0,1]$ supported in the truncated tent $T^{1}(B)$ over a ball $B$ in $X$ of radius $r(B)\leq 2$ with
\[\|a\|_{L^2_\bullet}=\left(\iint_{T^1(B)} |a(y,t)|^2\ \d\mu(y)\frac{\d t}{t}\right)^{\frac{1}{2}}\leq\mu(B)^{-\frac{1}{2}}.\]
\end{Def}

If $a$ is a $t^{1}$-atom corresponding to a ball $B$ as above, then the Cauchy--Schwarz inequality implies that $a\in t^{1}\cap t^2$ with
$\|a\|_{t^{2}} \lesssim \|a\|_{L^2_\bullet} \leq \mu(B)^{-1/2}$
and
\begin{equation}\label{eq: uniformtentatoms}
\|a\|_{t^{1}} \leq \mu(B)^{\frac{1}{2}}\|a\|_{t^{2}} \lesssim 1.
\end{equation}
\begin{Rem}\label{RemLargeTentAtoms}
If $(\lambda_j)_j$ is a sequence in $\ell^1$ and $(a_j)_j$ is a sequence of $t^{1}$-atoms, then \eqref{eq: uniformtentatoms} implies that $\sum_j\lambda_ja_j$ converges in $t^{1}$ with $\|\sum_j\lambda_ja_j\|_{t^{1}} \lesssim \|(\lambda_j)_j\|_{\ell^1}$. Note that this did not require the condition $r(B)\leq 2$ in Definition~\ref{DefTentAtom}.
\end{Rem}

The atomic characterisation of $t^{1}$ asserts the converse of the above remark. This is the content of the following theorem.

\begin{Thm}\label{maintentatomic}
Let $X$ be a locally doubling metric measure space. If $f\in t^1$, then there exist a sequence $(\lambda_j)_j$ in $\ell^1$ and a sequence $(a_j)_j$ of $t^1$-atoms such that $\sum_j\lambda_ja_j$ converges to $f$ in $t^1$ and almost everywhere in $X\times(0,1]$. Moreover, we have
\[
\|f\|_{t^{1}} \eqsim \inf \{\|(\lambda_j)_j\|_{\ell^1} : f=\textstyle{\sum_j}\lambda_ja_j\}.
\]
Also, if $p\in(1,\infty)$ and $f\in t^1\cap t^p$, then $\sum_j\lambda_ja_j$ converges to $f$ in $t^p$ as well.
\end{Thm}

The proof of Theorem~\ref{maintentatomic} is deferred to Appendix~\ref{AppendixA} so that it does not disrupt the main flow of ideas. It is also possible to characterise $t^1$ in terms of atoms supported in truncated Carleson boxes.

\begin{Def}\label{DefCarlesonTentAtom} Let $X$ be a locally doubling metric measure space. A $t^1\!\textit{-Carleson}$ \textit{atom} is a measurable function $a$ on $X\times(0,1]$ supported in the truncated Carleson box $C^1(B)$ over a ball $B$ in $X$ of radius $r(B)>0$ with $\|a\|_{L^2_\bullet}\leq\mu(B)^{-1/2}$.
\end{Def}

It is immediate that Theorem~\ref{maintentatomic} holds with $t^1$-Carleson atoms in place of $t^1$-atoms. As explained in Remark~\ref{RemLargeTentAtoms}, the converse of Theorem~\ref{maintentatomic} does not require the upper bound $r(B)\leq 2$ on the radii of the supports of $t^1$-atoms. This may not be the case for $t^1$-Carleson atoms on a locally doubling metric measure space. In the following proposition, however, we show that this is the case on an exponentially locally doubling metric measure space. We will need this to prove the molecular characterisation of $h^1_\cD$ in Lemma~\ref{localmolecularlemma2}. This is the first indication that \eqref{ELD} is more suited to our purposes than \eqref{LD}.

\begin{Prop} \label{Prop: CarlesonBoxTentAtoms}
Let $X$ be an exponentially locally doubling metric measure space. If $(\lambda_j)_j$ is a sequence in $\ell^1$ and $(a_j)_j$ is a sequence of $t^{1}$-Carleson atoms, then $\sum_j\lambda_ja_j$ converges in $t^{1}$ with
$
\textstyle{\|\sum_j\lambda_ja_j\|_{t^{1}} \lesssim \|(\lambda_j)_j\|_{\ell^1}}.
$
\end{Prop}
\begin{proof}
It is enough to show that $\sup \|a\|_{t^1} \lesssim 1$, where the supremum is taken over all $a$ that are $t^1$-Carleson atoms.

Let $a$ be a $t^1$-Carleson atom supported on a ball $B$ in $X$ of radius $r(B)>0$ with $\|a\|_{L^2_\bullet}\leq\mu(B)^{-1/2}$. First suppose that $r(B)\leq1$. It follows by \eqref{LD} that $\mu(2B)\leq c\mu(B)$ for some $c>0$ that does not depend on $B$. Also, we have $C^1(B)\subset T^1(2B)$ and the radius $r(2B)\leq2$. This implies that $a/\sqrt{c}$ is a $t^1$-atom and the result follows by \eqref{eq: uniformtentatoms}.

Now suppose that $r(B)>1$. Let $\mathscr{B}$ be the collection of all balls centered in $B$ with radius equal to $1/4$. Proposition~\ref{Aimar2.11} gives a sequence $(B_j)_j$ of pairwise disjoint balls from $\mathscr{B}$ such that $B\subseteq \bigcup_j \tilde B_j$, where $\tilde B_j=4B_j$. We also have the following bounded intersection property:
\[\sup_{j}\sharp\big(\{k\ |\ \tilde B_j\cap \tilde B_k \neq \emptyset\}\big) < \infty.\]
This follows from the local property of homogeneity, and in particular Remark~\ref{CW1.1loc}, since for each $j\in\N$, the centres of all balls $\tilde B_k$ satisfying $\tilde B_j\cap \tilde B_k \neq \emptyset$ are separated by at least a distance of $1/4$ and contained in $2\tilde B_j$. Therefore, the following are well defined for each $j\in\N$:
\[\tilde a_j = \frac{a\ca_{C^1(\tilde B_j)}}{\sum_k\ca_{C^1(\tilde B_k)}};
\quad a_j=\frac{\tilde{a}_j}{\mu(\tilde B_j)^{\frac{1}{2}}\|\tilde{a}_j\|_{L^2_\bullet}};
\quad \lambda_j= \mu(\tilde B_j)^{\frac{1}{2}} \|\tilde{a}_j\|_{L^2_\bullet}.\]
Also, we have $C^1(B)=B\times(0,1]\subseteq\bigcup_j C^1(\tilde B_j)$, since the radius $r(\tilde B_j)=1$. We can then write $a=\sum_j \lambda_j a_j$, where each $a_j$ is a $t^1$-atom by the previous paragraph. Therefore, we have
\[
\|a\|_{t^1}^2
\lesssim \Big(\sum_j |\lambda_j|\Big)^2
\leq \Big(\sum_j \mu(\tilde B_j)\Big)
\Big(\sum_j \|\tilde a_j\|_2^2\Big)  \\
\lesssim \mu\Big(\bigcup_j B_j\Big) \|a\|_2^2,
\]
where we used \eqref{LD} in the final inequality to obtain $\mu(\tilde B_j)\lesssim \mu(B_j)$. Each $B_j$ is contained in
$(1+\frac{1}{4r(B)})B$, so by \eqref{ELD} we obtain
\[
\|a\|_{t^1}^2
\lesssim \mu((1+\tfrac{1}{4r(B)})B)\mu(B)^{-1}
\lesssim 1,
\]
which completes the proof.
\end{proof}

The following duality and interpolation results for the local tent spaces follow as in the global case.

\begin{Thm}\label{Thm: tp2propertiesDuality}
Let $X$ be a locally doubling metric measure space. If $p\in[1,\infty)$ and $1/p+1/p'=1$, then the mapping
\[
g\mapsto \langle f,g\rangle_{L^2_\bullet}=\iint f(x,t)\overline{{g}(x,t)}\ \d\mu(x)\frac{\d t}{t}
\]
for all $f\in t^{p}$ and $g\in t^{p'}$, is an isomorphism from $t^{p'}$ onto the dual space $(t^p)^*$.
\end{Thm}
\begin{proof}
For $p=1$ and $p'=\infty$, the proof is closely related to the atomic characterisation in Theorem~\ref{maintentatomic} and follows the proof of Theorem~1 in \cite{CMS}. The remaining cases follow the proof of Theorem~2 in \cite{CMS}.
\end{proof}

\begin{Thm}\label{Thm: tp2propertiesInterp}
Let $X$ be a locally doubling metric measure space. If $\theta\in(0,1)$ and $1\leq p_0 < p_1\leq \infty$, then
\[
[t^{p_0},t^{p_1}]_\theta= t^{p_\theta},
\]
where $1/p_\theta=(1-\theta)/p_0+\theta/p_1$ and $[\cdot,\cdot]_\theta$ denotes complex interpolation.
\end{Thm}
\begin{proof}
The interpolation space $[t^{p_0},t^{p_1}]_\theta$ is well-defined because
\[
t^p(X\times(0,1]) \subseteq L_{\text{loc}}^2(X\times(0,1])
\]
for all $p\in[1,\infty]$ by \eqref{(1.3)}. This allows us to construct the Banach space ${t^{p_0}+t^{p_1}}$, which is the smallest ambient space in which $t^{p_0}$ and $t^{p_1}$ are continuously embedded. The proof then follows that of Theorem~3 and Proposition~1 in \cite{Bernal}.
\end{proof}

We conclude this section by dealing with a technicality involving the space $t^\infty$. In contrast with Proposition~\ref{Prop: tpt2denseintp}, the set $t^\infty\cap t^2$ may not be dense in $t^\infty$ when $X$ is not compact. Therefore, define $\tilde t^\infty$ to be the closure of $t^1\cap t^\infty$ in $t^\infty$, so we have both the density of $\tilde{t}^\infty\cap t^2$ in $\tilde{t}^\infty$ and the interpolation result in the following corollary.

\begin{Cor}\label{Cor: tp2Interp}
Let $X$ be a locally doubling metric measure space. If $\theta\in(0,1)$ and $1\leq p < \infty$, then
\[
[t^p,\tilde t^\infty]_\theta= t^{p_\theta},
\]
where $1/p_\theta=(1-\theta)/p$ and $[\cdot,\cdot]_\theta$ denotes complex interpolation. Also, the set $\tilde t^\infty\cap t^q$ is dense in $\tilde t^\infty$ for all $q\in[1,\infty]$, and $t^2$ is dense in $t^1+\tilde t^\infty$.
\end{Cor}
\begin{proof}
If $\theta\in(0,1)$, then by a standard property of complex interpolation, as in Theorem~1.9.3(g) of \cite{T}, and Theorem~\ref{Thm: tp2propertiesInterp}, we have
\[
[t^1,\tilde t^\infty]_\theta = [t^1,t^\infty]_{\theta} = t^{1/(1-\theta)}.
\]
If $p\in(1,\infty)$, then by the standard reiteration theorem for complex interpolation, as in Theorem~1.7 in Chapter IV of \cite{KPS}, we have
\[
[t^p,\tilde t^\infty]_\theta = [t^1,\tilde t^\infty]_{(1-\theta)(1-1/p)+\theta} = t^{p_\theta},
\]
where the density properties required to apply the reiteration theorem are guaranteed by Proposition~\ref{Prop: tpt2denseintp}.

Finally, the interpolation in Theorem~\ref{Thm: tp2propertiesInterp} implies that $t^1\cap t^\infty \subseteq t^q$ for all $q\in[1,\infty]$. Therefore, the density of $t^1\cap t^\infty$ in $\tilde t^\infty$ implies that $\tilde t^\infty\cap t^q$ is dense in $\tilde t^\infty$ for all $q\in[1,\infty]$. The density of $t^1\cap t^2$ in $t^1$ from Proposition~\ref{Prop: tpt2denseintp} then implies that $t^2$ is dense in $t^1+\tilde t^\infty$.
\end{proof}

\section{Some New Function Spaces $L^p_{\mathscr{Q}}(X)$}\label{SectionNewFunctionSpaces}
We introduce some new function spaces $L^p_{\mathscr{Q}}(X)$, or simply $L^p_{\mathscr{Q}}$, for all $p\in[1,\infty]$ in the context of a locally doubling metric measure space $X$. Note that functions on $X$ are assumed to be complex-valued. We begin with the following abstraction of the unit cube structure in $\R^n$.
\begin{Def}\label{Def: UnitCubeStructure}
Let $X$ be a metric measure space. A \textit{unit cube structure} on $X$ is a countable collection $\mathscr{Q}=(Q_j)_j$ of pairwise disjoint measurable sets that cover $X$, for which there exists $\delta\in(0,1]$ and a sequence of balls $(B_j)_j$ in $X$ of radius equal to 1 such that
\[\delta B_j \subseteq Q_j\subseteq B_j.\]
The sets in $\mathscr{Q}$ are called \textit{unit cubes}.
\end{Def}

A unit cube structure exists on a locally doubling space.
\begin{Lem}
If $X$ is a locally doubling metric measure space, then it has a unit cube structure.
\end{Lem}
\begin{proof}
The cubes are constructed in the same way that general dyadic cubes are constructed in Section I.3.2 of \cite{St3}. Let $\mathscr{B}$ be the collection of all balls in $X$ with radius equal to $1/4$. Proposition~\ref{Aimar2.11} gives a sequence $(B_j)_j$ of pairwise disjoint balls from $\mathscr{B}$ such that $X=\bigcup_j 4B_j$. The unit cubes $Q_j$ are then defined recursively for each $j\in\N$ by
\[Q_j = 4B_j \cap \comp{\Big(\bigcup_{k<j}Q_k\Big)} \cap \comp{\Big(\bigcup_{k>j} B_k\Big)}.\]
We have $\delta=1/4$ in this unit cube structure.
\end{proof}

In the proof above we could instead use the dyadic cubes constructed by Christ in \cite{Christ}. In any case, this brings us to the definition of $L^1_{\mathscr{Q}}(X)$.

\begin{Def} \label{LpQdef}
Let $X$ be a locally doubling metric measure space. Let ${\mathscr{Q}=(Q_j)_j}$ be a unit cube structure on $X$. For each $p\in[1,\infty)$, the space $L^p_{\mathscr{Q}}(X)$ consists of all measurable functions $f$ on $X$ with
\[
\|f\|_{L^p_{\mathscr{Q}}}= \Big(\sum_{Q_j\in\mathscr{Q}}\big(\mu(Q_j)^{\frac{1}{p}-\frac{1}{2}}\|\ca_{Q_j}f\|_2\big)^p\Big)^{\frac{1}{p}} <\infty.
\]
The space $L^\infty_\mathscr{Q}(X)$ consists of all measurable functions $f$ on $X$ with
\[
\|f\|_{L^\infty_\mathscr{Q}}
=\sup_{Q_j\in \mathscr Q} \mu(Q_j)^{-\frac{1}{2}}\|\ca_{Q_j}f\|_2<\infty.
\]
\end{Def}

These are Banach spaces under the usual identification of functions that are equal almost everywhere. The space $L^2_\mathscr{Q}(X)$ is exactly the Hilbert space $L^2(X)$. More generally, completeness holds because for each compact set $K\subseteq X$ and each ${p\in[1,\infty]}$, we have
\begin{equation}\label{(1.4)}
\|\ca_Kf\|_{L^p_{\mathscr Q}}\lesssim_{K,p} \|\ca_K f\|_2 \lesssim_{K,p} \|f\|_{L^p_{\mathscr Q}}
\end{equation}
for all measurable functions $f$ on $X$.

We will see that the $L^p_\mathscr{Q}$ spaces are independent of the unit cube structure $\mathscr Q$ used in their definition. First, however, we consider their relationship with the $L^p$~spaces.

\begin{Prop}\label{Prop: LpQL2denseinLp}
Let $X$ be a locally doubling metric measure space. The following hold:
\begin{enumerate}
\item $L^p_{\mathscr{Q}}\cap L^q_{\mathscr{Q}}$ is dense in $L^p_{\mathscr{Q}}$ for all $p\in[1,\infty)$ and $q\in[1,\infty]$;
\item $L^p_{\mathscr Q} \subseteq L^p$ for all $p\in[1,2]$;
\item $L^p \subseteq L^p_{\mathscr Q}$ for all $p\in[2,\infty]$.
\end{enumerate}
\end{Prop}
\begin{proof}
Let $p\in[1,\infty)$ and $f\in L^p_{\mathscr Q}$. Fix a ball $B$ in $X$ of radius $r(B)\geq1$ and define $f_k=\ca_{kB}f$ for each $k\in\N$. The functions $f_k$ belong to $L^p_{\mathscr{Q}}\cap L^q_{\mathscr{Q}}$ for all $q\in[1,\infty]$ by \eqref{(1.4)},
and
\[
\lim_{k\rightarrow\infty}\|f-f_k\|_{L^p_{\mathscr{Q}}}^p = \lim_{k\rightarrow\infty} \sum_{Q_j\cap \comp{(kB)} \neq \emptyset} \big(\mu(Q_j)^{\frac{1}{p}-\frac{1}{2}}\|\ca_{Q_j}f\|_2\big)^p = 0
\]
because $f\in L^p_{\mathscr{Q}}$, which proves (1).

We use H\"{o}lder's inequality to prove (2) and (3). If $p\in[1,2]$, then
\begin{align*}
\|f\|_p^p
= \sum_{Q_j\in\mathscr{Q}}\|\ca_{Q_j}f^p\|_1
\leq \sum_{Q_j\in\mathscr{Q}} \big(\mu(Q_j)^{\frac{1}{p}-\frac{1}{2}}\|\ca_{Q_j}f\|_2\big)^p
=\|f\|_{L^p_{\mathscr{Q}}}^p
\end{align*}
for all $f\in L^p_{\mathscr Q}$, which proves (2). If $p\in[2,\infty)$, then
\[
\|f\|_{L^p_{\mathscr{Q}}}^p
= \sum_{Q_j\in\mathscr{Q}} \big( \mu(Q_j)^{\frac{1}{p}-\frac{1}{2}}\|\ca_{Q_j}f^2\|_1^{\frac{1}{2}}\big)^p
\leq \sum_{Q_j\in\mathscr Q} \|\ca_{Q_j}f\|_p^p
=\|f\|_p^p
\]
for all $f\in L^p$, whilst
\begin{align*}
\|f\|_{L^\infty_{\mathscr{Q}}}
= \sup_{Q_j\in \mathscr Q} \mu(Q_j)^{-\frac{1}{2}}\|\ca_{Q_j}f^2\|_1^{\frac{1}{2}}
\leq \sup_{Q_j\in \mathscr Q} \|\ca_{Q_j}f^2\|_{\infty}^{\frac{1}{2}}
=\|f\|_\infty
\end{align*}
for all $f\in L^\infty$, which proves (3).
\end{proof}

Now we turn to the atomic characterisation of $L^1_{\mathscr Q}$.
\begin{Def}\label{L1QatomDef}
Let $X$ be a locally doubling metric measure space. An $L^1_{\mathscr{Q}}$-\textit{atom} is a measurable function $a$ on $X$ supported on a ball $B$ in $X$ of radius $r(B)\geq1$ with $\|a\|_2\leq\mu(B)^{-{1}/{2}}$.
\end{Def}

If $a$ is an $L^1_\mathscr{Q}$-atom, then $a$ belongs to $L^1_{\mathscr Q}\cap L^2$ with $\|a\|_1\lesssim 1$. If $X$ is exponentially locally doubling, then it is shown in the following theorem that $\|a\|_{L^1_\mathscr{Q}} \lesssim 1$. This allows us to prove that $L^1_\mathscr{Q}$ is precisely the subspace of $L^1$ in which functions have an atomic characterisation consisting purely of atoms supported on balls with large radii. The effectiveness of \eqref{ELD} in the proof of the first part of the following theorem can be understood in terms of its equivalence with the condition  \eqref{U} from Proposition~\ref{Prop: ELDequivDgloDloc}.

\begin{Thm}\label{mainLQatomic}
Let $X$ be an exponentially locally doubling metric measure space. The following hold:

\begin{enumerate}\item If $(\lambda_j)_j$ is a sequence in $\ell^1$ and $(a_j)_j$ is a sequence of $L^1_{\mathscr{Q}}$-atoms, then $\sum_j\lambda_ja_j$ converges in $L^1_{\mathscr Q}$ with $\|\sum_j\lambda_ja_j\|_{L^1_{\mathscr Q}} \lesssim \|(\lambda_j)_j\|_{\ell^1}$;

\item If $f\in L^1_{\mathscr{Q}}$, then there exist a sequence $(\lambda_j)_j$ in $\ell^1$ and a sequence $(a_j)_j$ of $L^1_{\mathscr{Q}}$-atoms such that $\sum_j\lambda_ja_j$ converges to $f$ in $L^1_{\mathscr Q}$ and almost everywhere in $X$. Moreover, we have
{\begin{center}
$\|f\|_{L^1_{\mathscr{Q}}} \eqsim \inf \{\|(\lambda_j)_j\|_{\ell^1} : f=\textstyle{\sum_j}\lambda_ja_j\}.$
\end{center}}
\noindent Also, if $p\in(1,\infty)$ and $f\in L^1_\mathscr{Q} \cap L^p_\mathscr{Q}$, then $\sum_j\lambda_ja_j$ converges to $f$ in $L^p_{\mathscr Q}$ as well.
\end{enumerate}
\end{Thm}
\begin{proof}
To prove (1), it is enough to show that $\sup\{\|a\|_{L^1_{\mathscr{Q}}} : \text{$a$ is an $L^1_\mathscr{Q}$-atom}\} \lesssim 1$. Let $a$ be an $L^1_{\mathscr{Q}}\text{-atom}$ supported on a ball $B$ of radius $r(B)~\geq~1$. Let
\[\mathscr{Q}_B=\{Q_j\in\mathscr{Q} :Q_j\cap B \neq \emptyset\}.\]
For each $Q_j\in\mathscr{Q}_B$, there exists a ball $B_j$ in $X$ of radius equal to 1 such that
\[\delta B_j \subseteq {Q}_j\subseteq B_j,\]
where $\delta$ is the constant associated with $\mathscr{Q}$ in Definition~\ref{Def: UnitCubeStructure}. The Cauchy--Schwarz inequality and the properties of the unit cube structure imply that
\begin{align*}
\|a\|_{L^1_{\mathscr{Q}}}^2
 \leq \|a\|_2^2 \sum_{Q_j\in\mathscr{Q}_B} \mu(Q_j) = \|a\|_2^2\ \mu\Big(\bigcup_{Q_j\in\mathscr{Q}_B} Q_j\Big)
\leq \mu(B)^{-1} \mu((1+\tfrac{2}{r(B)})B).
\end{align*}
The lower bound on $r(B)$ and \eqref{ELD} then imply that $\|a\|_{L^1_{\mathscr{Q}}}\lesssim 1$, where the constant in $\lesssim$ does not depend on $a$.

To prove (2), let $f\in L^1_{\mathscr{Q}}$. We can write $f(x)=\sum_{Q_j\in{\mathscr{Q}}}\lambda_ja_j(x)$ for almost every $x\in X$, where
\[
a_j(x)= \frac{\ca_{Q_j} f(x)}{\mu(Q_j)^{\frac{1}{2}}\|\ca_{Q_j}f\|_2}
\quad\text{and}\quad
\lambda_j=\mu(Q_j)^{\frac{1}{2}}\|\ca_{Q_j}f\|_2.
\]
Given that $f\in L^1_\mathscr Q$, this series also converges to $f$ in $L^1_\mathscr{Q}$. The same reasoning shows that if ${f\in L^1_\mathscr{Q}\cap L^p_\mathscr{Q}}$ for some $p\in(1,\infty)$, then the series also converges to $f$ in $L^p_\mathscr{Q}$.
Also, each $a_j$ is supported in $Q_j \subseteq B_j$, so by \eqref{LD} we obtain
\[\|a_j\|_2 \leq \mu(Q_j)^{-\frac{1}{2}} \leq  \mu(\delta B_j)^{-\frac{1}{2}} \lesssim \mu(B_j)^{-\frac{1}{2}}.\]
Therefore, each $a_j$ is a constant multiple of an $L^1_{\mathscr{Q}}$-atom, and this constant does not depend on $f$ or $Q_j$. The result then follows since
$\|(\lambda_j)_j\|_{\ell^1}=\|f\|_{L^1_\mathscr{Q}}$.
\end{proof}

\begin{Rem}\label{RemSmallL1QAtoms}
The proof of the second part of Theorem~\ref{mainLQatomic} actually shows that a function in $L^1_{\mathscr{Q}}$ has a characterisation in terms of $L^1_{\mathscr{Q}}$-atoms supported on balls of radius \textit{equal} to 1.
\end{Rem}

The definition of $L^1_\mathscr{Q}$-atoms does not require a unit cube structure. Therefore, the atomic characterisation of $L^1_{\mathscr Q}$ shows that, up to an equivalence of norms, $L^1_{\mathscr Q}$ is independent of the unit cube structure $\mathscr Q$ used in its definition. The atomic characterisation of $L^1_\mathscr{Q}$ is also related to the following duality.

\begin{Thm}\label{Thm: LpQpropertiesDuality}
Let $X$ be an exponentially locally doubling metric measure space. If $p\in[1,\infty)$ and $1/p+1/p'=1$, then the mapping
\[
g\mapsto\langle f,g\rangle_{L^2}=\int f(x)\overline{g(x)}\ \d\mu(x)
\]
for all $f\in L^{p}_\mathscr{Q}$ and $g\in L^{p'}_{\mathscr{Q}}$, is an isometric isomorphism from $L^{p'}_\mathscr{Q}$ onto the dual space $(L^p_{\mathscr{Q}})^*$.
\end{Thm}
\begin{proof}
Let $p\in[1,\infty)$. If $f\in L^p_{\mathscr Q}$ and $g\in L^{p'}_{\mathscr Q}$, then H\"{o}lder's inequality gives
\begin{align*}
|\langle f, g\rangle_{L^2}|
&\leq \sum_{Q_j\in\mathscr{Q}} |\langle \ca_{Q_j}f, \ca_{Q_j} g\rangle_{L^2}| \\
&\leq \sum_{Q_j\in\mathscr{Q}} \|\ca_{Q_j}f\|_2  \|\ca_{Q_j}g\|_2 \,\mu(Q_j)^{\frac{1}{p}-\frac{1}{2}}\mu(Q_j)^{\frac{1}{2}-\frac{1}{p}}\\
&\leq \|f\|_{L^p_{\mathscr Q}} \|g\|_{L^{p'}_{\mathscr Q}}.
\end{align*}

To prove the converse, given $p$ and $q\in[1,\infty)$, let
$
w_q(Q_j)=\mu(Q_j)^{1-q/2}
$
for all $Q_j\in\mathscr Q$, and define $\ell^p(w_q)$ to be the space of all sequences $\xi=\{\xi_{Q_j}\}_{Q_j\in\mathscr Q}$ with $\xi_{Q_j}\in L^2(Q_j)$ and
\[
\|\xi\|_{\ell^p(w_q)}=\Big(\sum_{Q_j\in\mathscr{Q}}\|\ca_{Q_j}\xi_{Q_j}\|_2^p\, w_q(Q_j)\Big)^{\frac{1}{p}} <\infty.
\]
Let $T\in(L^p_{\mathscr{Q}})^*$ and define $\tilde{T}\in(\ell^p(w_p))^*$ by
\[
\tilde{T}(\xi)=T\Big(\sum_{Q_j\in\mathscr Q} \ca_{Q_j}\xi_{Q_j}\Big)
\]
for all $\xi\in\ell^p(w_p)$. It is immediate that $\|\tilde T\|\leq\|T\|$, and by the standard duality there exists $\eta \in \ell^{p'}(w_p)$ such that $\|\eta\|_{\ell^{p'}(w_p)}\leq\|\tilde T\|$ and
\[
\tilde{T}(\xi) = \sum_{Q_j\in\mathscr Q} \langle\ca_{Q_j}\xi_{Q_j}, \ca_{Q_j}  \eta_{Q_j}\rangle_{L^2}\, w_p(Q_j)
\]
for all $\xi \in \ell^p(w_p)$. Therefore, we have
\[
T(f) = \tilde T(\{\ca_{Q_j}f\}_{Q_j\in\mathscr Q})
= \sum_{Q_j\in\mathscr Q} \langle f, \ca_{Q_j}\eta_{Q_j}\rangle_{L^2}\, w_p(Q_j)
= \langle f, g\rangle_{L^2}
\]
for all $f\in L^p_{\mathscr Q}$, where $g=\sum_{Q_j\in\mathscr Q} \ca_{Q_j}\eta_{Q_j} w_p(Q_j)$. Now consider two cases: (1) If $p\in(1,\infty)$, then
\begin{align*}
\|g\|_{L^{p'}_{\mathscr Q}}
= \Big(\sum_{Q_j\in\mathscr Q} \mu(Q_j)^{1-\frac{p'}{2}}\|\eta_{Q_j} \mu(Q_j)^{1-\frac{p}{2}}\|_2^{p'}\Big)^{\frac{1}{p'}}
=\|\eta\|_{\ell^{p'}(w_p)}
\leq \|T\|;
\end{align*}
(2) If $p=1$, then
\begin{align*}
\|g\|_{L^\infty_{\mathscr Q}}
&= \sup_{Q_j\in \mathscr Q} \mu(Q_j)^{-\frac{1}{2}}\|\ca_{Q_j}g\|_2 \\
&= \sup_{Q_j\in \mathscr Q} \mu(Q_j)^{-\frac{1}{2}}\sup_{\substack{\|f\|_2=1,\\ \supp f\subseteq Q_j}}|\langle f, g\rangle_{L^2}| \\
&= \sup_{Q_j\in \mathscr Q}\sup_{\substack{\|f\|_2=1,\\ \supp f\subseteq Q_j}} \mu(Q_j)^{-\frac{1}{2}}|T(f)|\\
&\leq \sup_{Q_j\in \mathscr Q}\sup_{\substack{\|f\|_2=1,\\ \supp f\subseteq Q_j}} \mu(Q_j)^{-\frac{1}{2}} \|T\|\|f\|_{L^1_{\mathscr Q}}\\
&= \|T\|,
\end{align*}
which completes the proof.
\end{proof}

The duality between $L^1_\mathscr{Q}$ and $L^\infty_\mathscr{Q}$ shows that, up to an equivalence of norms, $L^\infty_{\mathscr Q}$ is independent of the unit cube structure $\mathscr Q$ used in its definition. This is made explicit by the following corollary.

\begin{Cor}
Let $X$ be an exponentially locally doubling metric measure space. Let $\mathscr{B}^1$ denote the set of all balls $B$ in $X$ of radius $r(B)\geq1$. Then
\[
\|f\|_{L^\infty_\mathscr{Q}}
\eqsim \sup_{B\in \mathscr{B}^1} \mu(B)^{-\frac{1}{2}}\|\ca_Bf\|_2
\]
for all $f\in L^\infty_{\mathscr Q}$.
\end{Cor}
\begin{proof}
Let $f\in L^\infty_{\mathscr Q}$. Given $Q\in \mathscr Q$, let $B$ be a ball in $X$ of radius $r(B)=1$ such that $\delta B\subseteq Q\subseteq B$, where $\delta$ is the constant associated with $\mathscr{Q}$ in Definition~\ref{Def: UnitCubeStructure}. It follows by $\eqref{LD}$ that $\mu(B)\lesssim \mu(\delta B)$, where the constant in $\lesssim$ does not depend on $Q$. Therefore, we have
\[
\mu(Q)^{-\frac{1}{2}}\|\ca_Q f\|_2  \lesssim \mu(B)^{-\frac{1}{2}}\|\ca_Bf\|_2
\]
for all $Q\in\mathscr Q$, which implies that
\[
\|f\|_{L^\infty_\mathscr{Q}}
\lesssim \sup_{B\in \mathscr{B}^1} \mu(B)^{-\frac{1}{2}}\|\ca_Bf\|_2.
\]
To show the converse, suppose that $g\in L^2$ is supported in a ball $B\in \mathscr{B}^1$ with radius $r(B)\geq 1$. As in the first part of the proof of Proposition~\ref{mainLQatomic}, we find that
\begin{align*}
\|g\|_{L^1_{\mathscr{Q}}}^2 \leq \|g\|_2^2\, \mu((1+\tfrac{2}{r(B)})B)\lesssim \|g\|_2^2\, \mu(B),
\end{align*}
where the second inequality, which follows from \eqref{ELD} since $r(B)\geq 1$, does not depend on $g$ or $B$. Using this and Theorem~\ref{Thm: LpQpropertiesDuality}, we obtain
\begin{align*}
\sup_{B\in \mathscr{B}^1} \mu(B)^{-\frac{1}{2}}\|\ca_Bf\|_2
&= \sup_{B\in \mathscr{B}^1} \mu(B)^{-\frac{1}{2}}\sup_{\substack{\|g\|_2=1,\\ \supp f\subseteq B}}|\langle g, f\rangle_{L^2}| \\
&\leq \sup_{B\in \mathscr{B}^1} \sup_{\substack{\|g\|_2=1,\\ \supp f\subseteq B}} \mu(B)^{-\frac{1}{2}}\|g\|_{L^1_\mathscr{Q}}\|f\|_{L^\infty_\mathscr{Q}} \\
&\lesssim \|f\|_{L^\infty_\mathscr{Q}},
\end{align*}
which completes the proof.
\end{proof}

Given that $L^1_{\mathscr Q}$ and $L^\infty_{\mathscr Q}$ are independent of the choice of $\mathscr Q$, the following interpolation result shows that, up to an equivalence of norms, the $L^p_\mathscr{Q}$ spaces for all $p\in(1,\infty)$ are independent of the unit cube structure $\mathscr Q$ used in their definition.

\begin{Thm}\label{Thm: LpQpropertiesInterp}
Let $X$ be an exponentially locally doubling metric measure space. If $\theta\in(0,1)$ and $1\leq p_0 < p_1\leq \infty$, then
\[
[L^{p_0}_{\mathscr{Q}},L^{p_1}_{\mathscr{Q}}]_\theta= L^{p_\theta}_{\mathscr{Q}}
\]
isometrically, where $1/p_\theta=(1-\theta)/p_0+\theta/p_1$ and $[\cdot,\cdot]_\theta$ denotes complex interpolation.
\end{Thm}
\begin{proof}
The interpolation space $[L^{p_0}_{\mathscr Q},L^{p_1}_{\mathscr Q}]_\theta$ is well-defined because
\[
L^p_{\mathscr Q}(X) \subseteq L_{\text{loc}}^2(X)
\]
for all $p\in[1,\infty]$ by \eqref{(1.4)}. This allows us to construct the Banach space ${L^{p_0}_{\mathscr Q}+L^{p_1}_{\mathscr Q}}$, which is the smallest ambient space in which $L^{p_0}_{\mathscr Q}$ and $L^{p_1}_{\mathscr Q}$ are continuously embedded.

The space $\ell^p(w_p)$ was defined for all $p\in[1,\infty)$ in the proof of Theorem~\ref{Thm: LpQpropertiesDuality}. Likewise, let
\[
w_\infty(Q_j)= \mu(Q_j)^{-\frac{1}{2}}
\]
for all $Q_j\in\mathscr Q$, and define $\ell^\infty(w_\infty)$ to be the space of all sequences $\xi=\{\xi_{Q_j}\}_{Q_j\in\mathscr Q}$ with $\xi_{Q_j}\in L^2(Q_j)$ and
\[
\|\xi\|_{\ell^\infty(w_\infty)}=\sup_{Q_j\in\mathscr{Q}}\|\ca_{Q_j}\xi_{Q_j}\|_2\, w_\infty(Q_j) <\infty.
\]

If $1\leq p_0 < p_1 <\infty$, then
$
w_{p_0}^{(1-\theta)/{p_0}}w_{p_1}^{\theta/{p_1}} = w_{p_\theta}^{1/p_\theta},
$
whilst if $p_1=\infty$, then $w_{p_0}^{(1-\theta)/{p_0}}w_{\infty}^{\theta} = w_{p_\theta}^{1/p_\theta}$. Therefore, by the interpolation of vector-valued $\ell^p$ spaces, as in Theorem~1.18.1 of \cite{T}, and the interpolation of weighted $L^2$ spaces, as  in Theorem~5.5.3 of \cite{BL}, we obtain
\[
[\ell^{p_0}(w_{p_0}),\ell^{p_1}(w_{p_1})]_\theta
= \ell^{p_\theta}(w_{p_\theta})
\]
isometrically. Note that the isometric equivalence is proved in Remark~1 of Section 1.18.1 of \cite{T}, and the proof for $p_1=\infty$ is given in Remark~2 of the same reference.

Define the operators $R$ and $S$ by
\[
R\xi= \sum_{Q_j\in\mathscr{Q}} \ca_{Q_j} \xi_{Q_j}
\quad\text{and}\quad
S f= \{\ca_{Q_j}f\}_{Q_j\in \mathscr Q}
\]
for all sequences $\xi=\{\xi_{Q_j}\}_{Q_j\in\mathscr Q}$ with $\xi_{Q_j}\in L^2(Q_j)$, and all measurable functions $f$ on $X$. If $p\in[1,\infty]$, then the restricted operators
\[
R:\ell^p(w_p)\rightarrow~ L^p_{\mathscr{Q}}
\quad\text{and}\quad
S:L^p_{\mathscr{Q}}\rightarrow\ell^p(w_p)
\]
are bounded with operator norms equal to 1. Moreover, we have $RS=I$ on $L^p_{\mathscr{Q}}$ and $R(\ell^p(w_p))=L^p_{\mathscr{Q}}$. The operator $R$ is a retraction and $S$ is its coretraction. It follows by Theorem~1.2.4 of \cite{T}, which concerns the interpolation of spaces related by a retraction, that $S$ is an isometric isomorphism from
$
[L^{p_0}_\mathscr{Q},L^{p_1}_\mathscr{Q}]_\theta
$
onto
\[
SR([\ell^{p_0}(w_{p_0}),\ell^{p_1}(w_{p_1})]_\theta)
=SR(\ell^{p_\theta}(w_{p_\theta}))
=S(L^{p_\theta}_{\mathscr{Q}})
\]
in $\ell^{p_\theta}(w_{p_\theta})$. Therefore, we have $[L^{p_0}_\mathscr{Q},L^{p_1}_\mathscr{Q}]_\theta = L^{p_\theta}_\mathscr{Q}$ isometrically.
\end{proof}

We conclude this section by defining $\tilde L^\infty_{\mathscr Q}$ to be the closure of $L^1_{\mathscr Q}\cap L^\infty_{\mathscr Q}$ in $L^\infty_{\mathscr Q}$, and noting the following corollary.

\begin{Cor}\label{Cor: LpQInterp}
Let $X$ be an exponentially locally doubling metric measure space. If $\theta\in(0,1)$ and $1\leq p < \infty$, then
\[
[L^p_{\mathscr Q},\tilde L^\infty_{\mathscr Q}]_\theta= L^{p_\theta}_{\mathscr Q}
\]
isometrically, where $1/p_\theta=(1-\theta)/p$ and $[\cdot,\cdot]_\theta$ denotes complex interpolation. Also, the set $\tilde L^\infty_{\mathscr Q}\cap L^q_{\mathscr Q}$ is dense in $\tilde L^\infty_{\mathscr Q}$ for all $q\in[1,\infty]$, and $L^2$ is dense in $L^1_{\mathscr Q}+\tilde L^\infty_{\mathscr Q}$.
\end{Cor}
\begin{proof}
The proof follows that of Corollary \ref{Cor: tp2Interp} by using Proposition~\ref{Prop: LpQL2denseinLp}(1) and Theorem~\ref{Thm: LpQpropertiesInterp}.
\end{proof}

\section{Exponential Off-Diagonal Estimates}\label{SectionODE}
We return to the setting of a complete Riemannian manifold $M$ and derive the off-diagonal estimates required to define and characterise our local Hardy spaces. To consider differential forms on $M$, let us first dispense with some technicalities.

For each $k=0,...,\dim M$ and $x\in M$, let $\wedge^k T^*_xM$ denote the $k\text{th}$ exterior power of the cotangent space $T^*_xM$. Let $\wedge^k T^*M$ denote the bundle over $M$ whose fibre at $x$ is $\wedge^k T^*_xM$, and let $\wedge T^*M=\oplus_{k=0}^{\dim M}\wedge^k T^*M$. A \textit{differential form} is a section of $\wedge T^*M$. For each $p\in [1,\infty]$, let $L^p(\wedge T^*M)$ denote the Banach space of all measurable differential forms $u$ with
\[
\|u\|_{L^p(\wedge T^*M)}=
\begin{cases}
    \big(\int_M |u(x)|_{\wedge T^*_xM}^p\ \d\mu(x)\big)^{\frac{1}{p}}, &{\rm if}\quad p\in[1,\infty); \\
    \esss_{x\in M} |u(x)|_{\wedge T^*_xM}, &{\rm if}\quad p=\infty,
  \end{cases}
\]
where $|\cdot|_{\wedge T^*_xM}$ is the norm associated with the inner-product $\langle\cdot,\cdot\rangle_{\wedge T^*_xM}$ given by the bundle metric on $\wedge T^*M$ at $x$.

Our main technical tool will be holomorphic functional calculus, which requires the following definition. A comprehensive introduction to this topic can be found in Chapter VII of~\cite{DunfordSchwartz(volI)1958}.

\begin{Def}\label{Def: sectors}
Given $0\leq\mu<\theta<\pi/2$, define the closed and open bisectors in the complex plane as follows:
\begin{align*}
S_\mu &= \{z\in\C : |\arg z|\leq\mu \text{ or } |\pi-\arg z|\leq\mu\};\\
S_\theta^o &= \{z\in\C\setminus\{0\} : |\arg z|<\theta \text{ or } |\pi-\arg z|<\theta\}.
\end{align*}
Given $r\geq0$, define the closed and open discs as follows:
\begin{align*}
D_r &= \{z\in\C : |z|\leq r\}; \\
D_r^o &= \{z\in\C : |z|< r\}.
\end{align*}
These are denoted together by $S_{\mu,r}=S_\mu \cup D_r$ and $S_{\theta,r}^o=S_\theta^o \cup D_r^o$. Note that $D_0=\{0\}$, $S_{\mu,0}=S_\mu$, $D_0^o=\emptyset$ and $S_{\theta,0}^o=S_\theta^o$. A holomorphic function on $S_{\theta,r}^o$ is called \textit{nondegenerate} if it is not identically zero on $S_{\theta,r}^o$ and, when $r=0$, is not identically zero on either component of $S_\theta^o$.

Let $H^\infty(S_{\theta,r}^o)$ denote the algebra of bounded holomorphic functions on $S_{\theta,r}^o$. Given $f\in H^\infty(S_{\theta,r}^o)$ and $t\in(0,1]$, define $f^*\in H^\infty(S_{\theta,r}^o)$ and $f_t\in H^\infty(S_{\theta,r/t}^o)$ as follows:
\begin{align*}
f^*(z)&=\overline{f(\bar z)}\ \ \text{ for all }\ z\in S_{\theta,r}^o;
\\
f_t(z)&=f(tz)\ \text{ for all }\ z\in S_{\theta,r/t}^o.
\end{align*}
Given $\alpha,\beta>0$, define the following sets:
\begin{align*}
\Psi_\alpha^\beta(S_{\theta,r}^o) &= \{\psi \in H^\infty(S_{\theta,r}^o) :
|\psi(z)| \lesssim \min\{|z|^{\alpha},|z|^{-\beta}\}\};\\
\Theta^\beta(S_{\theta,r}^o) &= \{\phi \in H^\infty(S_{\theta,r}^o) :
|\phi(z)| \lesssim |z|^{-\beta}\}.
\end{align*}
Let $\Psi_\alpha(S_{\theta,r}^o)=\bigcup_{\beta>0}\Psi_\alpha^\beta(S_{\theta,r}^o)$, $\Psi^\beta(S_{\theta,r}^o)=\bigcup_{\alpha>0}\Psi_\alpha^\beta(S_{\theta,r}^o)$, $\Psi(S_{\theta,r}^o)=\bigcup_{\beta>0}\Psi^\beta(S_{\theta,r}^o)$ and $\Theta(S_{\theta,r}^o)=\bigcup_{\beta>0}\Theta^\beta(S_{\theta,r}^o)$.
\end{Def}

There is also the following notation.

\begin{Not1}
Given a linear operator $T$ on $L^2(\wedge T^*M)$, let $\Dom(T)$, $\Ran(T)$, $\Nul(T)$ and $\|T\|$ denote its domain, range, null space and operator norm, respectively. The \textit{resolvent set} $\rho(T)$ is the set of all $z\in\C$ for which $z I-T$ has a bounded inverse with domain equal to $L^2(\wedge T^*M)$. The \textit{resolvent} $R_T(z)$ is defined by
\[
R_T(z)=(z I - T)^{-1}
\]
for all $z\in\rho(T)$. The \textit{spectrum} $\sigma(T)$ is the complement of the resolvent set in the extended complex plane.

Given a bounded measurable scalar-valued function $\eta$ on $M$, let $\eta I$ denote the operator on $L^2(\wedge T^*M)$ of pointwise multiplication by $\eta$. Square brackets $[\cdot,\cdot]$ denote the commutator operator.
\end{Not1}

In the remainder of the paper, we consider a closed and densely-defined operator $\mathcal{D}:\Dom(\cD)\subseteq L^2(\wedge T^*M) \rightarrow L^2(\wedge T^*M)$ satisfying the following hypotheses:

\newcounter{Hyp}
\begin{list}
{(H\arabic{Hyp})}
{\setlength{\leftmargin}{.9cm}
\setlength{\rightmargin}{0cm}
\setlength{\topsep}{10pt}
\setlength{\itemsep}{3pt}
\usecounter{Hyp}}
\item\label{H1} There exists $\omega\in[0,\pi/2)$ and $R\geq0$ such that $\mathcal{D}$ is of $\textit{type }S_{\omega,R}$. This is defined to mean that $\sigma(\mathcal{D}) \subseteq S_{\omega,R}$ and that for each $\theta\in(\omega,\pi/2)$ and $r>R$, the constant \[C_{\theta,r} := \sup \{|z| \|R_\cD(z)\| : z\in \C\setminus S_{\theta,r}\}\] satisfies $0< C_{\theta,r} < \infty.$ Given $\phi\in\Theta(S^o_{\theta, r})$, this property allows us to define the bounded operator $\phi(\cD)$ on $L^2(\wedge T^*M)$ by
\[
\phi(\cD)u = \frac{1}{2\pi i} \int_{+\partial S^o_{\tilde\theta,\tilde r}} \phi(z) R_\cD(z)u\ \d z
\]
for all $u\in L^2(\wedge T^*M)$, where $\tilde\theta\in(\omega,\theta)$, $\tilde r\in(R,r)$ and $+\partial S^o_{\tilde\theta,\tilde r}$ denotes the boundary of $S^o_{\tilde\theta,\tilde r}$ oriented anticlockwise.

\item\label{H2} For all $\theta\in(\omega,\pi/2)$ and $r>R$, the operator $\mathcal{D}$ has a \textit{bounded $H^\infty(S_{\theta,r}^o)$ functional calculus in $L^2(\wedge T^*M)$}. This is defined to mean that for each ${\theta \in (\omega,\pi/2)}$ and $r>R$, there exists $c>0$ such that
\[
\|\phi(\cD)\| \leq c \|\phi\|_\infty
\]
for all $\phi\in \Theta(S^o_{\theta, r})$. Given $f\in H^\infty(S^o_{\theta, r})$, this property allows us to define the bounded operator $f(\cD)$ on $L^2(\wedge T^*M)$ by $f(\cD)u = \lim_{n\rightarrow\infty} (f\phi_{1/n})(\cD)u$ for all $u\in L^2(\wedge T^*M)$, where $\phi\in\Theta(S^o_{\theta, r})$ such that $\phi_{1/n}(z):=\phi(z/n)$ converges to 1 uniformly on compact subsets of $S^o_{\theta,r}$. The mapping $f~\mapsto~f(\mathcal{D})$ is an algebra homomorphism from $H^\infty(S_{\theta,r}^o)$ into the algebra of bounded linear operators on $L^2(\wedge T^*M)$ such that
    \[\|f(\mathcal{D})\|\leq c \|f\|_\infty\]
for all $f\in H^\infty(S_{\theta,r}^o)$.

\item\label{H3} The operator $\mathcal{D}$ is a first-order differential operator in the following sense. There exists $C_\cD >0$ such that for all smooth compactly-supported scalar-valued functions $\eta\in C_c^\infty(M)$, the domain $\Dom(\cD)\subseteq \Dom(\cD\circ \eta I)$ and the commutator $[\cD,\eta I]$ is a pointwise multiplication operator such that
\[
|[\mathcal{D},\eta I]u(x)|_{\wedge T^*_xM}
\leq C_\cD  |d\eta(x)|_{T^*_xM}|u(x)|_{\wedge T^*_xM}
\]
for all $u\in \Dom(\cD)$ and almost all $x\in M$, where $d$ is the exterior derivative.
\end{list}

The hypothesis (H1) in the case $R=0$ is precisely the condition that $\cD$ is of $\text{type }S_{\omega}$ (or $\omega$-sectorial). The theory of type $S_{\omega}$ operators is well-understood and can be found in, for instance, \cite{Haase2006,Kato}. In particular, given $0\leq\omega<\theta<\pi/2$, $\psi\in\Psi(S^o_\theta)$ and an operator $T$ of type $S_{\omega}$, it is proved in \cite{ADMc,McIntosh1986} that $T$ has a bounded $H^\infty(S^o_\theta)$ functional calculus if and only if the quadratic estimate
\[
\int_0^\infty \|\psi_t(T)u\|_2^2\ \frac{\d t}{t} \eqsim \|u\|^2
\]
holds for all $u\in\overline{\Ran(T)}$. The theory of type $S_{\omega,R}$ operators in the case $R\geq0$ is contained in \cite{Morris2010(1)}. In particular, given an operator $\cD$ satisfying (H1), the main result of that paper shows that (H2) is equivalent to the requirement that $\cD$ satisfies local quadratic estimates, which we will introduce after Proposition~\ref{sq=I}.

Note that, as a means of generalizing this theory to other contexts, one could replace the space $C_c^\infty(M)$ in (H3) with the space of bounded scalar-valued Lipschitz functions $\text{Lip}(M)$. This stronger condition is still satisfied by the Hodge--Dirac operator, as in Example~\ref{Eg: HodgeDirac} below, and it obviates the need to construct smooth approximations in the proof of Lemma~\ref{ODexp0}. Moreover, all of the results in this paper hold under this condition.

\begin{Eg}\label{Eg: HodgeDirac}
The Hodge--Dirac operator $D=d+d^*$ defined on the space of smooth compactly supported differential forms $C^\infty_c(\wedge T^*M)$ is essentially self-adjoint (see Theorem~1.17 and Example~1.7 in \cite{GromovLawson1983}). Therefore, its unique self-adjoint extension, also denoted by $D$, immediately satisfies (H1-2) with $\omega=0$, $R=0$ and $C_{\theta,r}=1/\sin\theta$ for all $\theta\in(0,\pi/2)$ and $r>0$. It also satisfies (H3), since it is a first-order differential operator, and $C_D =1$, since for all $\eta\in C_c^\infty(M)$ we have
\[
|[D,\eta I]u(x)|_{\wedge T^*_xM}
= |d\eta(x)\ext u(x) - d\eta(x)\lint u(x)|_{\wedge T^*_xM}
= |d\eta(x)|_{T^*_xM}|u(x)|_{\wedge T^*_xM}
\]
for all $u\in C^\infty_c(\wedge T^*M)$ and almost all $x\in M$, where $\ext$ and $\lint$ denote the exterior and (left) interior products on $\wedge T_x^*M$, respectively. Note that the second equality above holds because $d\eta(x)\lint$ is an antiderivation on $\wedge T^*_xM$, which implies that
\[
d\eta\lint (d\eta\ext u) = |d\eta|_{T^*M}^2 u - d\eta \ext(d\eta\lint u)
\]
pointwise almost everywhere on $M$.
\end{Eg}

Off-diagonal estimates, otherwise known as Davies--Gaffney estimates, provide a measure of the decay associated with the action of an operator. Their use as a substitute for pointwise kernel bounds is becoming abundant in the literature. In particular, they are an essential tool used to prove the Kato Conjecture in \cite{AHLMT} and the related results in \cite{AKMc}.  The theory of off-diagonal estimates has also been developed in its own right in \cite{AuscherMartell2007}. The following notation is suited to these estimates.
\begin{Not1}
For all $x\geq0$, let $\langle x\rangle=\min\{1,x\}$. For all closed subsets $E,F\subseteq M$, let $\rho(E,F)=\inf_{x\in E,\,y\in F}\rho(x,y).$
\end{Not1}

We prove off-diagonal estimates for the resolvents $R_\cD(z)$ and then deduce estimates for more general functions of $\cD$ by using holomorphic functional calculus. The following proof utilizes the higher-commutator technique from Section 2 of \cite{McIntoshNahmod2000}. Note that we could instead apply the technique for establishing off-diagonal estimates from \cite{AuscherAxelssonMcIntosh2010,AHLMT}.

\begin{Lem}\label{ODexp0} Let $0\leq\omega<\theta<\pi/2$ and $0\leq R<r$ and suppose that $\cD$ is a closed operator on $L^2(\wedge T^*M)$ of type $S_{\omega,R}$ satisfying (H1) and (H3) with constants $C_{\theta,r}>0$ and $C_\cD>0$. For each $a\in(0,1)$ and $b\geq0$, there exists $c>0$ such that
\[\|\ca_ER_\cD(z)\ca_F\| \leq c\,\frac{C_{\theta,r}}{|z|}\left\langle\frac{1}{\rho(E,F)|z|}\right\rangle^b \exp\left(-a\frac{\rho(E,F)|z|}{C_\cD C_{\theta,r}} \right)\]
for all $z\in \C\setminus S_{\theta,r}$ and closed subsets $E$ and $F$ of $M$.
\end{Lem}
\begin{proof}
Let $E$ and $F$ be closed subsets of $M$ with $\rho(E,F)>0$. For each $\ep>0$, there exists $\eta:M\rightarrow[0,1]$ in $C_c^\infty(M)$ such that
\[
\eta(x)=\begin{cases}
    1, &{\rm if}\quad x\in E; \\
    0, &{\rm if}\quad \rho(x,E)\geq\rho(E,F)
  \end{cases}
\]
and $\|d\eta\|_{\infty} = \sup_{x\in M} |d\eta(x)|_{T^*_xM} \leq (1+\ep)/{\rho(E,F)}$. The function $\eta$ can be constructed from smooth approximations of the Lipschitz function $f$ defined by
\[
f(x)=\begin{cases}
    1-\rho(x,E)/\rho(E,F), &{\rm if}\quad \rho(x,E)< \rho(E,F); \\
    0, &{\rm if}\quad \rho(x,E)\geq\rho(E,F)
  \end{cases}
\]
for all $x\in M$. Note that $f$ is Lipschitz because the geodesic distance $\rho$ is Lipschitz on a Riemannian manifold. For further details see, for instance, \cite{AzagraFerreraLopezRangel2007}.

Fix $a\in(0,1)$, $\delta\in(a,1)$ and $\ep=\frac{\delta-a}{a+1}$. It suffices to show that
\begin{equation}\label{eq: factorial.to.exp}
\|\ca_ER_\cD(z)\ca_F\| \leq \inf_{n\in\N_0}  n! \frac{C_{\theta,r}}{|z|} \left(\frac{(1+\ep)C_\cD  C_{\theta,r}}{\rho(E,F)|z|}\right)^{n},
\end{equation}
where $\N_0=\N\cup\{0\}$. For $b=0$, the result follows from \eqref{eq: factorial.to.exp} because $\delta/(1+\ep)\geq a$ and $e^{\delta x} = \sum_{n\in\N_0} (\delta x)^n/n! \leq \frac{1}{1-\delta} \sup_{n\in\N_0} x^n/n!$ for all $x>0$. For each $b>0$, the result follows from \eqref{eq: factorial.to.exp} because $(\delta-\ep)/(1+\ep)\geq a$ and $e^{-\delta x} \lesssim x^{-b} e^{-(\delta-\ep)x}$ for all $x>0$.

We make repeated use, without reference, of the following easily verified identities for operators $A,B$ and $C$:
\[[A,BC]=[A,B]C+B[A,C]; \quad [A,B^{-1}]=B^{-1}[B,A]B^{-1}\]
on the largest domains for which both sides are defined.

First, we show by induction that
\begin{equation}\label{indeqn1}
[\eta I, ([\cD,\eta I]R_\cD(z))^n]=-n([\cD,\eta I]R_\cD(z))^{n+1}
\end{equation}
for all $n\in\N$. The commutator $[\cD,\eta I]$ is a pointwise multiplication operator by hypothesis (H3). This implies that $[\eta I,[\cD,\eta I]]=0$, so $\eqref{indeqn1}$ holds for $n=1$. If $\eqref{indeqn1}$ holds for some $k\in\N$, then
\begin{align*}
[\eta I,([\cD,&\eta I]R_\cD(z))^{k+1}]\\
&= [\eta I,[\cD,\eta I]R_\cD(z)]([\cD,\eta I]R_\cD(z))^{k} + [\cD,\eta I]R_\cD(z)[\eta I,([\cD,\eta I]R_\cD(z))^{k}] \\
&= [\cD,\eta I][\eta I,R_\cD(z)]([\cD,\eta I]R_\cD(z))^{k} - k ([\cD,\eta I] R_\cD(z))^{k+2}\\
&= -[\cD,\eta I]R_\cD(z)[\cD,\eta I]R_\cD(z)([\cD,\eta I]R_\cD(z))^{k} -k ([\cD,\eta I] R_\cD(z))^{k+2} \\
&= -(k+1)([\cD,\eta I]R_\cD(z))^{k+2},
\end{align*}
so \eqref{indeqn1} holds for all $n\in\N$. Next, we show by induction that
\begin{equation}\label{indeqn2}
\overbrace{[\eta I,...[\eta I,}^{n} R_\cD(z)]...] = (-1)^nn!R_\cD(z)([\cD,\eta I]R_\cD(z))^n
\end{equation}
for all $n\in\N$. This is immediate for $n=1$. If $\eqref{indeqn2}$ holds for some $k\in\N$, then by \eqref{indeqn1} we have
\begin{align*}
\overbrace{[\eta I,...[\eta I,}^{k+1} R&_\cD(z)]...]\\
&=(-1)^kk![\eta I, R_\cD(z)([\cD,\eta I]R_\cD(z))^k] \\
&=(-1)^kk!\{[\eta I,R_\cD(z)]([\cD,\eta I]R_\cD(z))^k + R_\cD(z)[\eta I,([\cD,\eta I]R_\cD(z))^k]\} \\
&=(-1)^{k}k!\{-R_\cD(z)([\cD,\eta I]R_\cD(z))^{k+1} - kR_\cD(z)([\cD,\eta I]R_\cD(z))^{k+1} \} \\
&=(-1)^{k+1}(k+1)!R_\cD(z)([\cD,\eta I]R_\cD(z))^{k+1},
\end{align*}
so \eqref{indeqn2} holds for all $n\in \N$. Using \eqref{indeqn2} with hypotheses (H1) and (H3), we obtain
\begin{align*}
\|\ca_ER_\cD(z)\ca_F\| &\leq \|(\eta I)^n R_\cD(z)\ca_F\|\\
&= \|(\eta I)^{n-1}[\eta I, R_\cD(z)]\ca_F\| \\
&= \|\overbrace{[\eta I,...[\eta I,}^{n} R_\cD(z)]...]\ca_F\| \\
&\leq n!\|R_\cD(z)([\cD,\eta I]R_\cD(z))^n\| \\
&\leq n! (C_\cD \|d\eta\|_\infty)^n \|R_\cD(z)\|^{n+1} \\
&\leq n!\frac{C_{\theta,r}}{|z|}\left(\frac{(1+\ep)C_\cD  C_{\theta,r}}{\rho(E,F)|z|}\right)^n
\end{align*}
for all $n\in\N_0$, which proves \eqref{eq: factorial.to.exp}.
\end{proof}

The following proof was inspired by the proof of Lemma~7.3 in \cite{HytonenvanNeervenPortal2008}.

\begin{Lem} \label{ODexpHvNP}
Let $0\leq\omega<\theta<\pi/2$ and $0\leq R<r$ and suppose that $\cD$ is an operator satisfying the assumptions of Lemma~\ref{ODexp0}. Let $M\geq0$ and $\delta>0$. For each $\psi\in\Psi_{M+\delta}^{\delta}(S_{\theta,r}^o)$, $\phi\in\Theta^\delta(S_{\theta,r}^o)$ and $a\in(0,1)$, there exists $c>0$ such that the following hold:
\begin{enumerate}
\setlength{\itemsep}{5pt}
\item $\displaystyle \|\ca_E(f\psi_t)(\mathcal{D})\ca_F\| \leq c \|f\|_\infty\left\langle\frac{t}{\rho(E,F)}\right\rangle^{M} \exp\left(-a\frac{r}{C_\cD C_{\theta,r}}\rho(E,F)\right)$;
\item $\displaystyle \|\ca_E(f\phi_{\phantom{t}})(\mathcal{D})\ca_F\| \leq c \|f\|_\infty\exp\left(-a\frac {r}{C_\cD C_{\theta,r}}\rho(E,F)\right),$
\end{enumerate}
for all $t\in(0,1]$, $f\in H^\infty(S_{\theta,r}^o)$ and closed subsets $E$ and $F$ of $M$.
\begin{proof}
For all $\tilde\theta\in(\omega,\theta)$ and $\tilde r\in(R,r)$, let $+\partial S^o_{\tilde\theta,\tilde r}$ denote the boundary of $S^o_{\tilde\theta,\tilde r}$ oriented anticlockwise, and divide this into $\gamma_{\tilde r}= +\partial S^o_{\tilde\theta,\tilde r} \cap D_{\tilde r}$ and $\gamma_{\tilde\theta}=+\partial S^o_{\tilde\theta,\tilde r} \cap S_{\tilde\theta}$. Using the Cauchy integral formula from (H1), we have
\begin{align*}
\ca_E(f\psi_t)(\mathcal{D})\ca_F&=\frac{1}{2\pi i}\left(\int_{\gamma_{\tilde r}}+\int_{\gamma_{\tilde\theta}}\right) f(z)\psi_t(z) \ca_E R_\cD(z)\ca_F\ \d z =I_1+I_2
\end{align*}
for all $\tilde\theta\in(\omega,\theta)$ and $\tilde r\in(R,r)$. It follows by Lemma~\ref{ODexp0} that for each $a\in(0,1)$ and $b\geq0$, we have
\begin{align*}
\|I_1\| &\lesssim C_{\tilde\theta,\tilde r}\|f\|_\infty\int_{\gamma_{\tilde r}}\min\{|tz|^{M+\delta},|tz|^{-\delta}\} \left\langle\frac{1}{\rho(E,F)|z|}\right\rangle^M e^{-a\rho(E,F)|z|/C_\cD C_{\tilde\theta,\tilde r}}\ \frac{|\d z|}{|z|}\\
&\lesssim_{r,R} C_{\tilde\theta,\tilde r}\|f\|_\infty \langle{t}/{\rho(E,F)}\rangle^{M} e^{-a\rho(E,F)\tilde r/C_\cD C_{\tilde\theta,\tilde r}}
\end{align*}
and
\begin{align*}
\|I_2\| \lesssim C_{\tilde\theta,\tilde r}\|f\|_\infty \left(\int^{\tilde r/t}_{\tilde r} \right.&\frac{|tz|^{M+\delta}}{(\rho(E,F) |z|)^{b}} e^{-a\rho(E,F)|z|/C_\cD C_{\tilde\theta,\tilde r}}\ \frac{\d|z|}{|z|} \\
&\left.\hspace{1.6cm} + \int^{\infty}_{\tilde r/t}\frac{|tz|^{-\delta}}{(\rho(E,F) |z|)^b} e^{-a\rho(E,F)|z|/C_\cD C_{\tilde\theta,\tilde r}}\ \frac{\d|z|}{|z|}\right)
\end{align*}
for all $\tilde\theta\in(\omega,\theta)$ and $\tilde r\in(R,r)$. Setting $b=0$ shows that
\[
\|I_2\| \lesssim_{r,R} C_{\tilde\theta,\tilde r} \|f\|_\infty e^{-a\rho(E,F)\tilde r/C_\cD C_{\tilde \theta,\tilde r}},
\]
and setting $b=M$ shows that
\[
\|I_2\| \lesssim_{r,R} C_{\tilde\theta,\tilde r} \|f\|_\infty (t/\rho(E,F))^M e^{-a \rho(E,F)\tilde r/C_\cD C_{\tilde \theta,\tilde r}}.
\]
Altogether, this shows that for each $a\in(0,1)$, there exists $c>0$ such that
\[
\|\ca_E(f\psi_t)(\mathcal{D})\ca_F\| \leq c\, C_{\tilde\theta,\tilde r} \|f\|_\infty\langle{t}/{\rho(E,F)}\rangle^{M} e^{-a(\tilde r/C_\cD C_{\tilde \theta,\tilde r})\rho(E,F)}
\]
for all $\tilde\theta\in(\omega,\theta)$ and $\tilde r\in(R,r)$. The first result follows by noting that
\[
\sup\{{\tilde r}/{C_{\tilde\theta,\tilde r}} : \tilde\theta\in(\omega,\theta),\,\tilde r\in(R,r)\} = {r}/{C_{\theta,r}}.
\]
The proof of the second result is similar.
\end{proof}
\end{Lem}

We conclude this section with a useful application of this result.
\begin{Prop} \label{ODexp5}
Let $0\leq\omega<\theta<\pi/2$ and $0\leq R<r$ and suppose that $\cD$ is an operator satisfying the assumptions of Lemma~\ref{ODexp0}. Let $0<\sigma<\alpha$ and $0<\tau<\beta$. For each $\psi\in\Psi_{\alpha}^{\beta}(S_{\theta,r}^o)$, $\tilde{\psi}\in \Psi_\beta^{\alpha}(S_{\theta,r}^o)$, $\phi\in\Theta^\beta(S_{\theta,r}^o)$, $\tilde{\phi}\in \Theta^\alpha(S_{\theta,r}^o)$ and $a\in(0,1)$, there exists $c>0$ such that the following hold:
\begin{enumerate}
\setlength{\itemindent}{-0.7cm}
\setlength{\labelsep}{0cm}
\item $\|\ca_E(\psi_tf\tilde\psi_s)(\cD)\ca_F\| \leq c \|f\|_\infty
  \begin{cases}
    ({s}/{t})^{\tau}\langle{t}/{\rho(E,F)}\rangle^{\alpha+\tau}
    e^{-a (r/C_\cD C_{\theta,r})\rho(E,F)}& \!\!\!{\rm if} \,\ s\leq t; \\
    ({t}/{s})^{\sigma}\langle{s}/{\rho(E,F)}\rangle^{\beta+\sigma}
    e^{-a (r/C_\cD C_{\theta,r})\rho(E,F)}& \!\!\!{\rm if} \,\ t\leq s;
  \end{cases}$
\item $\|\ca_E(\phi_{\phantom{t}} f\tilde\psi_s)(\cD)\ca_F\| \leq c \|f\|_\infty \ s^{\tau} e^{-a (r/C_\cD C_{\theta,r})\rho(E,F)}$;
\item $\|\ca_E(\psi_t f\tilde\phi_{\phantom{s}})(\cD)\ca_F\| \leq c \|f\|_\infty \ t^{\sigma} e^{-a (r/C_\cD C_{\theta,r})\rho(E,F)}$;
\item $\|\ca_E(\phi_{\phantom{t}} f\tilde\phi_{\phantom{s}})(\cD)\ca_F\|\leq c \|f\|_\infty \ e^{-a (r/C_\cD C_{\theta,r})\rho(E,F)},$
\end{enumerate}
for all $s,t\in(0,1]$, $f\in H^\infty(S_{\theta,r}^o)$ and closed subsets $E$ and $F$ of $M$.
\end{Prop}
\begin{proof}
To prove (1), first suppose that $0<s\leq t\leq 1$ and choose $\delta\in(0,\beta-\tau)$.  Let $g_{(s)}(z)=(s z)^{-(\tau+\delta)}\tilde\psi_s(z)f(z)$ and $\eta(z)= z^{\tau+\delta}{\psi}(z)$ so that \[\psi_tf\tilde\psi_s = (s/t)^{\tau+\delta} g_{(s)} \eta_t.\] The function $\eta$ is in $\Psi^{\beta-\tau-\delta}_{\alpha+\tau+\delta}(S_{\theta,r}^o)$ and the functions $g_{(s)}$ are in $\Psi(S_{\theta,r}^o)$ and satisfy $\sup_{s\in(0,1]}\|g_{(s)}\|_\infty \lesssim \|f\|_\infty$. Therefore, Lemma~\ref{ODexpHvNP}  provides the  off-diagonal estimate
\[
\|\ca_E(g_{(s)} \eta_t)(\cD)\ca_F\| \lesssim \|g_{(s)}\|_\infty\langle{t}/{\rho(E,F)}\rangle^{\alpha+\tau} e^{-a(r/C_\cD C_{\theta,r})\rho(E,F)}
\]
and the required estimate follows. The proof in the case $0<t\leq s\leq 1$ is analogous.

The results in (2) and (3) follow from Lemma~\ref{ODexpHvNP} by writing the following:
\begin{align*}
(\phi f\tilde\psi_s)(z)
&= s^{\tau}z^{\tau}\phi(z)f(z)(sz)^{-\tau}\tilde{\psi}(sz); \\
(\psi_t f\tilde\phi)(z)
&=t^{\sigma}z^{\sigma}\tilde\phi(z)f(z)(tz)^{-\sigma}\psi(tz).
\end{align*}
The result in (4) follows immediately from Lemma~\ref{ODexpHvNP}.
\end{proof}

\section{The Main Estimate}\label{SectionTheMainEstimate}
We consider a complete  Riemannian manifold $M$ that is exponentially locally doubling. The spaces $t^{p}(X\times(0,1])$ and $L^p_{\mathscr{Q}}(X)$ introduced in Sections \ref{SectionLocalTent} and \ref{SectionNewFunctionSpaces} consist of measurable functions. We begin by showing that it is a simple matter to formulate that theory for differential forms.

The local Lusin operator $\mathcal{A}_{\text{loc}}$ is defined for any measurable family of differential forms $U=(U_t)_{t\in(0,1]}$ on $M$, where each $U_t$ is a section of $\Lambda T^*M$, by
\[
\mathcal{A}_{\text{loc}}U(x) = \bigg(\iint_{\Gamma^1(x)} |U_t(y)|_{\wedge T^*_yM}^2\ \frac{\d\mu(y)}{V(x,t)} \frac{\d t}{t}\bigg)^{\frac{1}{2}}
\]
for all $x\in M$. The dual operator $\mathcal{C}_{\text{loc}}$ is defined in the same way. For each $p\in[1,\infty]$, the local tent space $t^p(\wedge T^*M\times(0,1])$ consists of all measurable families of differential forms $U$ on $M$ with
\[
\|U\|_{t^p}=
\begin{cases}
     \big(\int_M (\cA_{\text{loc}}U(x))^p\ \d\mu(x)\big)^{\frac{1}{p}} , &{\rm if}\quad p\in[1,\infty); \\
    \esss_{x\in M} \cC_{\text{loc}}U(x), &{\rm if}\quad p=\infty.
  \end{cases}
\]
Let $L^2_\bullet(\wedge T^*M\times(0,1])$ denote the space of all measurable families of differential forms $U$ on $M$ with $\|U\|_{L^2_\bullet}^2= \int_0^1 \|U_t\|_{L^2(\wedge T^*M)}^2\frac{\d t}{t}$. As before, this is an equivalent norm on $t^2(\wedge T^*M\times(0,1])$.

Next, fix a unit cube structure $\mathscr{Q}=(Q_j)_j$ on $M$. For each $p\in[1,\infty]$, the space $L^p_{\mathscr{Q}}(\wedge T^*M)$ consists of all measurable differential forms $u$ on $M$ with
\[
\|u\|_{L^p_{\mathscr{Q}}}=
\begin{cases}
    \big(\sum_{Q_j\in{Q}}(\mu(Q_j)^{\frac{1}{p}-\frac{1}{2}}\|\ca_{Q_j}u\|_{L^2(\wedge T^*M)})^p\big)^{\frac{1}{p}}, &{\rm if}\quad p\in[1,\infty); \\
    \sup_{Q_j\in \mathscr Q} \mu(Q_j)^{-\frac{1}{2}}\|\ca_{Q_j}f\|_{L^2(\wedge T^*M)}, &{\rm if}\quad p=\infty.
\end{cases}
\]
As before, we have $L^2_{\mathscr{Q}}(\wedge T^*M)=L^2(\wedge T^*M)$.

A $t^1(\wedge T^*M)$-atom is a measurable family of differential forms $A=(A_t)_{t\in(0,1]}$ on $M$ supported in the truncated tent $T^{1}(B)$ over a ball $B$ in $M$ of radius $r(B)\leq 2$ with $\|A\|_{L^2_\bullet} \leq \mu(B)^{-1/2}$. The atomic characterisation in Theorem~\ref{maintentatomic} is proved in this context by defining the local maximal operator $\mathcal{M}_{\text{loc}}$ for all measurable differential forms $u$ on $M$ by
\[\mathcal{M}_\text{loc}u(x)=\sup_{r\in(0,1]} \frac{1}{V(x,r)} \|\ca_{B(x,r)}u\|_{L^1(\wedge T^*M)}\]
for all $x\in M$

An $L^1_\mathscr{Q}(\wedge T^*M)$-atom is a measurable differential form $a$ on $M$ supported on a ball $B$ in $M$ of radius $r(B)\geq1$ with $\|a\|_{2}\leq\mu(B)^{-1/2}$. The proof of the atomic characterisation in Theorem~\ref{mainLQatomic} goes over directly.

The duality and interpolation results from Sections \ref{SectionLocalTent} and \ref{SectionNewFunctionSpaces} extend to this setting as well. In what follows, we only consider spaces of differential forms and usually omit writing $\wedge T^*M$ and $\wedge T^*M\times(0,1]$.

\begin{Def}\label{Def: QandS}
Let $M$ be a complete Riemannian manifold. Let ${\omega\in[0,\pi/2)}$ and $R\geq0$ and suppose that $\cD$ is a closed densely-defined operator on $L^2(\wedge T^*M)$ of $\text{type }S_{\omega,R}$ satisfying (H1-2). Given $\theta \in (\omega,\pi/2)$, $r>R$, $\psi\in\Psi(S_{\theta,r}^o)$ and $\phi\in\Theta(S_{\theta,r}^o)$, define the bounded operators $\cQ^{\cD}_{\psi,\phi}:L^2 \rightarrow L^2_\bullet\oplus L^2$ and $\cS^{\cD}_{\psi,\phi}:L^2_\bullet\oplus L^2\rightarrow L^2$ by
\[
\cQ^{\cD}_{\psi,\phi}u = (\psi_t({\cD})u, \phi({\cD})u)
\]
for all $u\in L^2$ and $t\in(0,1]$, and
\[
\cS^{\cD}_{\psi,\phi}(U,u) = \int^1_0\psi_s({\cD}) U_s\frac{\d s}{s} + \phi({\cD})u = \lim_{a\rightarrow 0} \int_a^1 \psi_s({\cD}) U_s\frac{\d s}{s} + \phi({\cD})u
\]
for all $(U,u)\in L^2_{\bullet}\oplus L^2$.
\end{Def}

The operator $\cQ^{\cD}_{\psi,\phi}$ is bounded because $\cD$ satisfies (H1) and (H2). This is a consequence of the equivalence of (H2) with the requirement that $\cD$ satisfies local quadratic estimates, which we will introduce after Proposition~\ref{sq=I}. Further details are in \cite{Morris2010(1)}. It is also well known that the adjoint operator $\cD^*$ satisfies (H1-2) if and only if $\cD$ satisfies (H1-2). Therefore, we have $\cS^{\cD}_{\psi,\phi}=(\cQ^{\cD^*}_{\psi^*,\phi^*})^*$, where $\psi^*$ and $\phi^*$ are given by Definition~\ref{Def: sectors}, and this is a bounded operator.

The remainder of this section is dedicated to the proof of the following theorem, which is fundamental to the definition of our local Hardy spaces. It is a local analogue of Theorem~4.9 in \cite{AMcR}. The proof below simplifies some aspects of the original proof.

\begin{Thm} \label{locMainEst}
Let $\kappa,\lambda\geq0$ and suppose that $M$ is a complete Riemannian manifold satisfying \eqref{ELD}. Let $\omega\in[0,\pi/2)$ and $R\geq0$ and suppose that $\cD$ is a closed densely-defined operator on $L^2(\wedge T^*M)$ of $\text{type }S_{\omega,R}$ satisfying (H1-3). Let $\theta\in(\omega,\tfrac{\pi}{2})$, $r>R$ and $\beta>{\kappa}/{2}$ such that $r/C_\cD C_{\theta,r}>{\lambda}/{2}$, where $C_{\theta,r}$ is from (H1) and $C_\cD$ is from (H3).

For each $\psi\in\Psi^{\beta}(S_{\theta,r}^o),$ $\tilde{\psi}\in\Psi_{\beta}(S_{\theta,r}^o),$ $\phi\in\Theta^{\beta}(S_{\theta,r}^o)$ and $\tilde{\phi}\in\Theta(S_{\theta,r}^o)$, there exists $c>0$ such that the following hold for all $f\in H^\infty(S_{\theta,r}^o)$:
\begin{enumerate}\setlength{\itemsep}{3pt}
\item The operator $\cQ^\cD_\svt{\psi,\phi}f({\mathcal{D}})\cS^\cD_{\tilde\psi,\tilde\phi}$ has a bounded extension $\Pf$ satisfying
\begin{center}$\displaystyle
\|\Pf(U,u)\|_{t^p\oplus L^p_{\mathscr{Q}}}
\leq c \|f\|_\infty \|(U,u)\|_{t^p\oplus L^p_{\mathscr{Q}}}
$\end{center}
for all $(U,u)\in t^p\oplus L^p_{\mathscr{Q}}$ and $p\in[1,2]$;

\item The operator $\cQ^\cD_{\tilde\psi,\tilde\phi}f({\mathcal{D}})\cS^\cD_\svt{\psi,\phi}$ has a bounded extension $\tilde\Pf$ satisfying
\begin{center}$\displaystyle
\|\tilde\Pf(U,u)\|_{t^p\oplus L^p_{\mathscr{Q}}}
\leq c \|f\|_\infty \|(U,u)\|_{t^p\oplus L^p_{\mathscr{Q}}}
$\end{center}
for all $(U,u)\in t^p\oplus L^p_{\mathscr{Q}}$ and $p\in[2,\infty]$.
\end{enumerate}
\end{Thm}
\begin{proof}
Hypothesis (H2) and the comments in the paragraph after Definition~\ref{Def: QandS} guarantee that both $\cQ^\cD_\svt{\psi,\phi}f({\mathcal{D}})\cS^\cD_{\tilde\psi,\tilde\phi}$ and $\cQ^\cD_{\tilde\psi,\tilde\phi}f({\mathcal{D}})\cS^\cD_\svt{\psi,\phi}$ satisfy the estimates in (1) and (2) on $t^2\oplus L^2_{\mathscr{Q}}$.

To prove (1), define the following operators:
\[
\begin{array}{ll}
\displaystyle\Pf^{1,1} U=\int^1_0\psi_t({\mathcal{D}})f({\mathcal{D}})\tilde{\psi}_s({\mathcal{D}})U_s\frac{\d s}{s};& \Pf^{1,2} u= \psi_t({\mathcal{D}})f({\mathcal{D}})\tilde{\phi}({\mathcal{D}})u;\\
\displaystyle\Pf^{2,1}U= \int^1_0\phi({\mathcal{D}})f({\mathcal{D}})\tilde{\psi}_s({\mathcal{D}})U_s\frac{\d s}{s};& \Pf^{2,2}u= \phi({\mathcal{D}})f({\mathcal{D}})\tilde{\phi}({\mathcal{D}})u,
\end{array}
\]
for all $U\in L^2_\bullet$, $u\in L^2$ and $t\in(0,1]$, so we have the system
\[
\cQ^\cD_\svt{\psi,\phi}f({\mathcal{D}})\cS^\cD_{\tilde\psi,\tilde\phi}(U,u)
=\bigg(\begin{matrix}
\Pf^{1,1}& \Pf^{1,2} \\
\Pf^{2,1}& \Pf^{2,2}
\end{matrix}\bigg)
\bigg(\begin{matrix}
U \\
u
\end{matrix}\bigg)
\]
for all $(U,u)\in L^2_\bullet\oplus L^2$.

We claim that there exists $c>0$ such that
\begin{equation}\label{eq: unifatom}
\|\cQ^\cD_\svt{\psi,\phi}f({\mathcal{D}})\cS^\cD_{\tilde\psi,\tilde\phi}(A,a)\|_{t^1\oplus L^1_{\mathscr{Q}}} \leq c\|f\|_\infty
\end{equation}
for all $A$ that are $t^1$-atoms and $a$ that are $L^1_\mathscr{Q}$-atoms. The proof of \eqref{eq: unifatom} is quite technical, so we postpone it to Lemmas~\ref{LemQ11}, \ref{LemQ12}, \ref{LemQ21} and \ref{LemQ22}.

The set $t^1 \cap t^2$ is dense in $t^1$ by Proposition~\ref{Prop: tpt2denseintp}. Therefore, to prove that there exist bounded extensions $\Pf^{1,1}:t^1 \rightarrow t^1$ and $\Pf^{2,1}:t^1 \rightarrow L^1_\mathscr{Q}$, it suffices to show that
\begin{equation}\label{eq: unifatom2}
    \|\Pf^{1,1}U\|_{t^1} \lesssim \|f\|_\infty\|U\|_{t^1}
    \quad\text{and}\quad
    \|\Pf^{2,1}U\|_{L^1_\mathscr{Q}} \lesssim \|f\|_\infty\|U\|_{t^1}
\end{equation}
for all $U\in t^1 \cap t^2$.

If $U\in t^1 \cap t^2$, then by Theorem~\ref{maintentatomic} there exist a sequence $(\lambda_j)_j$ in $\ell^1$ and a sequence $(A_j)_j$ of $t^1$-atoms such that $\sum_j\lambda_jA_j$ converges to $U$ in $t^2$ with $\|(\lambda_j)_j\|_{\ell^1}\lesssim\|U\|_{t^1}$. Then, since $\cQ^\cD_\svt{\psi,\phi}f({\mathcal{D}})\cS^\cD_{\tilde\psi,\tilde\phi}$ is bounded on $t^2\oplus L^2_\mathscr{Q}$, we have
\[
\textstyle \Pf^{2,1}U=\sum_j\lambda_j \Pf^{2,1}(A_j),
\]
where the sum converges in $L^2_\mathscr{Q}$. Also, the partial sums $\sum_{j=1}^n\Pf^{2,1}(\lambda_j A_j)$ form a Cauchy sequence in $L^1_\mathscr{Q}$ by \eqref{eq: unifatom}. Therefore, there exists $v\in L^1_\mathscr{Q}$ such that
\[
\textstyle v=\sum_j\lambda_j\Pf^{2,1}( A_j),
\]
where the sum converges in $L^1_\mathscr{Q}$, and
$\|v\|_{L^1_\mathscr{Q}} \lesssim \|f\|_\infty \|U\|_{t^1}.$
Given that both $L^1_\mathscr{Q}$ and $L^2_\mathscr{Q}$ are continuously embedded in $L^1_\mathscr{Q}+L^2_\mathscr{Q}$, as in the proof of Theorem~\ref{Thm: LpQpropertiesInterp}, we must have $v=\Pf^{2,1}U$. A similar argument holds for $\Pf^{1,1}U$ to give \eqref{eq: unifatom2}.

The set $L^1_\mathscr{Q}\cap L^2_\mathscr{Q}$ is dense in $L^1_\mathscr{Q}$ by Proposition~\ref{Prop: LpQL2denseinLp}. Therefore, to prove that there exist bounded extensions $\Pf^{1,2}:L^1_\mathscr{Q} \rightarrow t^1$ and $\Pf^{2,2}:L^1_\mathscr{Q} \rightarrow L^1_\mathscr{Q}$, it suffices to show that
\begin{equation}\label{eq: unifatom3}
    \|\Pf^{1,2}u\|_{t^1} \lesssim \|f\|_\infty\|u\|_{L^1_\mathscr{Q}}
    \quad\text{and}\quad
    \|\Pf^{2,2}u\|_{L^{1}_\mathscr{Q}} \lesssim \|f\|_\infty\|u\|_{L^1_\mathscr{Q}}
\end{equation}
for all $u\in L^1_\mathscr{Q}\cap L^2_\mathscr{Q}$.

If $u\in L^1_\mathscr{Q}\cap L^2_\mathscr{Q}$, then by Theorem~\ref{mainLQatomic} there exist a sequence $(\lambda_j)_j$ in $\ell^1$ and a sequence $(a_j)_j$ of $L^1_\mathscr{Q}$-atoms such that $\sum_j\lambda_ja_j$ converges to $u$ in $L^2_\mathscr{Q}$ with $\|(\lambda_j)_j\|_{\ell^1}\lesssim\|u\|_{L^1_\mathscr{Q}}$. Then, since $\cQ^\cD_\svt{\psi,\phi}f({\mathcal{D}})\cS^\cD_{\tilde\psi,\tilde\phi}$ is bounded on $t^2\oplus L^2_\mathscr{Q}$, we have
\[
\textstyle \Pf^{1,2}u=\sum_j\lambda_j\Pf^{1,2}(a_j),
\]
where the sum converges in $t^2$. Also, the partial sums $\sum_{j=1}^n\Pf^{1,2}(\lambda_j a_j)$ form a Cauchy sequence in $t^1$ by \eqref{eq: unifatom}. Therefore, there exists $V\in t^1$ such that
\[
\textstyle V=\sum_j\lambda_j\Pf^{1,2}( a_j),
\]
where the sum converges in $t^1$, and $\|V\|_{t^1} \lesssim \|f\|_\infty \|u\|_{L^1_\mathscr Q}.$ Given that both $t^1$ and $t^2$ are continuously embedded in $t^1+t^2$, as in the proof of Theorem \ref{Thm: tp2propertiesInterp}, we must have $V=\Pf^{1,2}u$. A similar argument holds for $\Pf^{2,2}U$ to give \eqref{eq: unifatom3}.

The bounds in \eqref{eq: unifatom2} and \eqref{eq: unifatom3} prove that $\cQ^\cD_\svt{\psi,\phi}f({\mathcal{D}})\cS^\cD_{\tilde\psi,\tilde\phi}$ has a bounded extension satisfying the estimate in (1) on $t^1\oplus L^1_\mathscr{Q}$. Therefore, result (1) follows by the interpolation in Theorems \ref{Thm: tp2propertiesInterp} and \ref{Thm: LpQpropertiesInterp}.

To prove (2), note that replacing $\cD$ with $\cD^*$ in the proof of (1) shows that $\cQ_\svtt{\psi^*,\phi^*}^{\cD^*}f^*({\mathcal{D^*}}) \cS_{\tilde\psi^*,\tilde\phi^*}^{\cD^*}$ has a bounded extension $\cP_{f^*}$ satisfying the estimate in (1) on $t^1\oplus L^1_\mathscr{Q}$. The duality in Theorems \ref{Thm: tp2propertiesDuality} and \ref{Thm: LpQpropertiesDuality} then allows us to define the dual operator $\cP_{f^*}'$ satisfying the estimate in (2) on $t^\infty\oplus L^\infty_\mathscr{Q}$. We also have ${\cP_{f^*}'=\cQ^\cD_{\tilde\psi,\tilde\phi}f({\mathcal{D}})\cS^\cD_\svt{\psi,\phi}}$ on $(t^\infty \cap t^2)\oplus (L^\infty_\mathscr{Q}\cap L^2_\mathscr{Q})$, as $\cS^\cD_{\psi,\phi}=(\cQ^{\cD^*}_{\psi^*,\phi^*})^*$
on $t^2\oplus L^2_\mathscr{Q}$. Therefore, result (2) follows by the interpolation in Theorems \ref{Thm: tp2propertiesInterp} and \ref{Thm: LpQpropertiesInterp}.
\end{proof}

The remainder of this section is devoted to proving (\ref{eq: unifatom}). The proof is divided into four lemmas. We adopt the notation
\[
\cQ^\cD_\svt{\psi,\phi}f({\mathcal{D}})\cS^\cD_{\tilde\psi,\tilde\phi}
=\bigg(\begin{matrix}
\Pf^{1,1}& \Pf^{1,2} \\
\Pf^{2,1}& \Pf^{2,2}
\end{matrix}\bigg)
\]
as in the proof of Theorem~\ref{locMainEst}.

\begin{Lem}\label{LemQ11}
Under the assumptions of Theorem~\ref{locMainEst}, there exists $c>0$ such that $\|\Pf^{1,1}A\|_{t^1}\leq c\|f\|_\infty$ for all $A$ that are $t^1$-atoms.
\end{Lem}
\begin{proof}
Let $A$ be a $t^1$-atom. There exists a ball $B$ in $M$ with radius ${r(B)}\leq 2$ such that $A$ is supported in $T^1(B)$ and $\|A\|_{L^2_{\bullet}}\leq \mu(B)^{-1/2}$. If $r(B)> 1/2$, let $K=0$. If $r(B)\leq 1/2$, let $K$ be the positive integer such that $2^K\leq 1/r(B) < 2^{K+1}$. Next, associate $B$ with the characteristic functions $\ca_k$ defined by
\[
\ca_k=
\begin{cases}
\ca_{T^1(4B)}&{\rm if}\quad k=0;\\
\ca_{T^1(2^{k+2}B)\setminus T^1(2^{k+1}B)}& {\rm if}\quad K\geq1 \quad\textrm{and} \quad k\in\{1,\dots, K\}.
\end{cases}
\]
Also, define the ball $B^*$ with radius $r(B^*)\in[4,8]$ by $B^*=2^{K+2} B$ and associate it with the characteristic functions $\ca_k^*$ defined by
\[\ca_k^*=\ca_{T^1((k+1)B^*)\setminus T^1(k B^*)}\]
for all $k\in\N=\{1,2,...\}$.

Let $\tilde A_k=\ca_k \Pf^{1,1}A$ and $\tilde A_k^*=\ca_k^* \Pf^{1,1}A$, so we have $\supp \tilde A_k \subseteq T^{1}(2^{k+2}B)$, $\supp \tilde A^*_k\subseteq T^{1}\left((k+1)B^*\right)$ and
\[
\Pf^{1,1}A=\sum_{k=0}^{K} \tilde A_k+\sum_{k=1}^\infty \tilde A_k^*.
\]
We prove below that there exist $c>0$ and two sequences $(\lambda_k)_{k\in\{0,...,K\}}$ and $(\lambda^*_k)_{k\in\N}$ in $\ell^1$, all of which do not depend on $A$, such that the following hold:
\begin{align}
&\label{eq: normbound1} \|\tilde A_k\|_{L^2_{\bullet}}\leq c\|f\|_\infty{\lambda_k}\, {\mu\left(2^{k+2}B\right)^{-\frac{1}{2}}} \hspace{1cm}\textrm{for all}\quad k\in\{0,...,K\} ;\\
&\label{eq: normbound2} \|\tilde A^*_k\|_{L^2_{\bullet}}\leq c\|f\|_\infty{\lambda^*_k}\, {\mu\left((k+1)B^*\right)^{-\frac{1}{2}}}\quad\textrm{for all}\quad k\in\N.
\end{align}
The result then follows by Remark~\ref{RemLargeTentAtoms}.

To prove \eqref{eq: normbound1} and \eqref{eq: normbound2}, choose $\delta$ in $(0,\frac{2\beta-\kappa}{3})$ so that $\psi\in\Psi^{\beta}_{2\delta}(S_{\theta,r}^o)$ and that $\tilde{\psi}\in\Psi^{2\delta}_{\beta}(S_{\theta,r}^o)$, which is possible because $\beta>\kappa/2$. Also, choose $a$ in $(\frac{\lambda}{2} \frac{C_\cD C_{\theta,r}}{r},1)$, which is possible because $r/C_\cD C_{\theta,r} >\lambda/2$. Proposition~\ref{ODexp5} applied with $\sigma=\delta$ and $\tau=\beta-\delta$ then shows that
\begin{equation}\label{eq: off1}
\|\ca_E(\psi_tf\tilde\psi_s)(\cD)\ca_F\|\lesssim \|f\|_\infty e^{-a (r/C_\cD C_{\theta,r})\rho(E,F)}
\begin{cases}
    (\frac{s}{t})^{\beta-\delta}\langle\frac{t}{\rho(E,F)}\rangle^{\beta+\delta}
    & \!\!{\rm if}\ \ s\leq t; \\
    (\frac{t}{s})^{\delta}\langle\frac{s}{\rho(E,F)}\rangle^{\beta+\delta}
    & \!\!{\rm if}\ \ t\leq s
\end{cases}
\end{equation}
for all $s,t\in(0,1]$ and closed subsets $E$ and $F$ of $M$. Applying the Cauchy--Schwarz inequality and considering the support of $A$, we also obtain
\begin{align}\begin{split}\label{eq: off2}
|(\cP^{1,1}_fA)_t|^2 &= \left|\int^{\langle {r(B)}\rangle}_0 \min\{\tfrac{t}{s},\tfrac{s}{t}\}^{\frac{\delta}{2}} \left(\min\{\tfrac{t}{s},\tfrac{s}{t}\}^{-\frac{\delta}{2}} (\psi_tf\tilde\psi_s)(\cD)A_s\right) \frac{\d s}{s}\right|^2\\
&\lesssim \int^{\langle {r(B)}\rangle}_0\min\{\tfrac{t}{s},\tfrac{s}{t}\}^{-\delta}|(\psi_tf\tilde\psi_s)(\cD)A_s|^2\frac{\d s}{s}
\end{split}\end{align}
for all $t\in(0,1]$. We now use \eqref{eq: off1} and \eqref{eq: off2} to prove \eqref{eq: normbound1} and \eqref{eq: normbound2}:
\\

\noindent\textit{Proof of \eqref{eq: normbound1}}. The operator $\cQ^\cD_\svt{\psi,\phi}f({\mathcal{D}})\cS^\cD_{\tilde\psi,\tilde\phi}$ is bounded on $L^2_{\bullet}\oplus L^2$, so we have
\begin{equation*}
    \|\tilde A_0\|_{L^2_{\bullet}}\leq \|\Pf(A,0)\|_{L^2_{\bullet}}\lesssim \|f\|_\infty\|A\|_{L^2_{\bullet}}\lesssim \|f\|_\infty\mu(4B)^{-\frac{1}{2}}.
\end{equation*}
Suppose that $K\geq 1$ and that $k\in\{1,\dots,K\}$, which implies that ${2^kr(B)\leq1}$. Note that the support of $\tilde A_k$ is contained in $T^1(2^{k+2}B)\setminus T^1(2^{k+1}B)$. Also, if $(x,t)$ belongs to $T^1(2^{k+2}B)\setminus T^1(2^{k+1}B)$ and $t\leq2^kr(B)$, then $x$ belongs to $2^{k+2}B\setminus 2^kB$. Using \eqref{eq: off2}, we then obtain
\begin{align*}
\|\tilde A_k\|_{L^2_{\bullet}}^2&\lesssim \int^{2^k r(B)}_0 \int^{r(B)}_0\min\{\tfrac{t}{s},\tfrac{s}{t}\}^{-\delta}\|\ca_{2^{k+2}B\setminus 2^kB}(\psi_tf\tilde\psi_s)(\cD)A_s\|_2^2\frac{\d s}{s}\frac{\d t}{t}\\
&\quad+\int^{\langle2^{k+2} {r(B)}\rangle}_{2^{k}{r(B)}}\int^{r(B)}_0\min\{\tfrac{t}{s},\tfrac{s}{t}\}^{-\delta}\|(\psi_tf\tilde\psi_s)(\cD)A_s\|_2^2\frac{\d s}{s}\frac{\d t}{t}\\
&= I_1+I_2.
\end{align*}

To estimate $I_1$, note that $\rho(2^{k+2}B\setminus 2^kB,B)=(2^k-1){r(B)}\leq 1$, since we are assuming that $2^kr(B)\leq1$. Using \eqref{eq: off1} and \eqref{ELD}, we then obtain
\begin{align*}
I_1 &\lesssim \|f\|_\infty^2\int_0^{r(B)}\int^s_0\Big(\frac{t}{s}\Big)^\delta \left(\frac{s}{2^kr(B)}\right)^{2\beta+2\delta} \frac{\d t}{t}\|A_s\|_2^2\frac{\d s}{s}\\
&\quad+\|f\|_\infty^2\int^{r(B)}_0 \int^{2^kr(B)}_s\Big(\frac{s}{t}\Big)^{2\beta-3\delta} \left(\frac{t}{2^kr(B)}\right)^{2\beta+2\delta}\frac{\d t}{t}\|A_s\|_2^2\frac{\d s}{s}\\
&\lesssim \|f\|_\infty^2 (2^{-(2\beta+2\delta)k}+2^{-(2\beta-3\delta)k})\|A\|_{L^2_{\bullet}}^2\\
&\lesssim \|f\|_\infty^2 2^{-(2\beta-3\delta)k}\mu(B)^{-1}\\
&\lesssim \|f\|_\infty^2 2^{-(2\beta-\kappa-3\delta)k} e^{\lambda 2^{k+2}r(B)}\mu(2^{k+2}B)^{-1}\\
&\lesssim \|f\|_\infty^2 2^{-(2\beta-\kappa-3\delta)k} \mu(2^{k+2}B)^{-1},
\end{align*}
where $2^kr(B)\leq 1$ was used in the final inequality. We also obtain
\begin{align*}
I_2 &\leq \|f\|_\infty^2 \int^{\infty}_{2^{k} {r(B)}}\int^{r(B)}_0\left(\frac{s}{t}\right)^{2\beta-3\delta}\|A_s\|_2^2\frac{\d s}{s}\frac{\d t}{t}\\
&\lesssim \|f\|_\infty^2 \int^{r(B)}_0\left(\frac{s}{2^kr(B)}\right)^{2\beta-3\delta}\|A_s\|_2^2\frac{\d s}{s}\\
&\lesssim \|f\|_\infty^2 2^{-(2\beta-3\delta)k}\|A\|_{L^2_{\bullet}}^2\\
&\lesssim \|f\|_\infty^2 2^{-(2\beta-\kappa-3\delta)k}\mu(2^{k+2}B)^{-1}.
\end{align*}
The bounds for $I_1$ and $I_2$ show that
\[\|\tilde A_k\|_{L^2_{\bullet}} \lesssim \|f\|_\infty 2^{-(2\beta-\kappa-3\delta)k/2} \mu(2^{k+2}B)^{-\frac{1}{2}},\]
which proves (\ref{eq: normbound1}) with $\lambda_k=2^{-(2\beta-\kappa-3\delta)k/2}$, since $2\beta-\kappa-3\delta >0$.
\\

\noindent\textit{Proof of \eqref{eq: normbound2}}. Suppose that $k\in\N$. If $(x,t)$ belongs to $T^1((k+1)B^*)\setminus T^1(kB^*)$, then $x$ belongs to $(k+1)B^*\setminus (k-1/4)B^*$, since the radius $r(B^*)\in[4,8]$. Also, since $r(B)\leq r(B^*)/4$, we have
\[
\rho((k+1)B^*\setminus (k-\tfrac{1}{4})B^*, B) \geq (k-\tfrac{1}{4})r(B^*)-{r(B)} \geq \max \{1,kr(B^*)\}.
\]
Using \eqref{eq: off1}, \eqref{eq: off2} and \eqref{ELD}, we then obtain
\begin{align*}
\|\tilde A^*_k\|_{L^2_{\bullet}} &\lesssim \|f\|_\infty^2 \int^{1}_0\int^{\langle {r(B)}\rangle}_0\min\{\tfrac{t}{s},\tfrac{s}{t}\}^{-\delta}\|\ca_{(k+1)B^*\setminus (k-1/4)B^*}(\psi_tf\tilde\psi_s)(\cD)A_s\|_2^2\frac{\d s}{s}\frac{\d t}{t}\\
&\lesssim \|f\|_\infty^2 e^{-2a(r/C_\cD C_{\theta,r}) kr(B^*)}\int_0^{\langle {r(B)} \rangle}\int^s_0\Big(\frac{t}{s}\Big)^\delta s^{2\beta+2\delta} \frac{\d t}{t}\|A_s\|_2^2\frac{\d s}{s}\\
&\quad+ \|f\|_\infty^2 e^{-2a(r/C_\cD C_{\theta,r}) kr(B^*)}\int^{\langle {r(B)}\rangle}_0\int^{1}_s\left(\frac{s}{t}\right)^{2\beta-3\delta} t^{2\beta+2\delta}\frac{\d t}{t}\|A_s\|_2^2\frac{\d s}{s}\\
&\lesssim \|f\|_\infty^2 e^{-2a(r/C_\cD C_{\theta,r}) kr(B^*)} \langle {r(B)}\rangle^{2\beta-3\delta} \|A\|_{L^2_{\bullet}}^2\\
&\lesssim \|f\|_\infty^2 e^{-2a(r/C_\cD C_{\theta,r}) kr(B^*)}{r(B)}^{\kappa}\mu(B)^{-1}\\
&\lesssim \|f\|_\infty^2 e^{-(2a(r/C_\cD C_{\theta,r})-\lambda)kr(B^*)}k^\kappa\mu((k+1)B^*)^{-1}.
\end{align*}
This shows that
\[
\|\tilde A_k^*\|_{L^2_{\bullet}} \lesssim \|f\|_\infty e^{-(2a(r/C_\cD C_{\theta,r})-\lambda)k} \mu((k+1)B^*)^{-\frac{1}{2}},
\]
which proves \eqref{eq: normbound2} with $\lambda^*_k=e^{-(2a(r/C_\cD C_{\theta,r})-\lambda)k}$, since $2a(r/C_\cD C_{\theta,r})-\lambda>0$.
\end{proof}

\begin{Lem}\label{LemQ21}
Under the assumptions of Theorem~\ref{locMainEst}, there exists $c>0$ such that $\|\Pf^{2,1}A\|_{L^1_\mathscr{Q}} \leq c\|f\|_\infty$ for all $A$ that are $t^1$-atoms.
\end{Lem}
\begin{proof}
Let $A$ be a $t^1$-atom. There exists a ball $B$ in $M$ with radius $r(B)\leq 2$ such that $A$ is supported in $T^1(B)$ and $\|A\|_{L^2_{\bullet}}\leq \mu(B)^{-1/2}$. Define the ball $B^*$ with radius $r(B^*)\in[2,4]$ by
\[
B^*=\begin{cases}
2B &{\rm if}\quad  1< r(B)\leq 2;\\
(2/r(B))B&{\rm if}\quad r(B)\leq 1
\end{cases}
\]
and associate $B^*$ with the characteristic functions $\ca_k^*$ defined by
\[
\ca_k^*=
\begin{cases}
\ca_{2B^*}&{\rm if}\quad k=0;\\
\ca_{(k+2)B^*\setminus (k+1) B^*}& {\rm if}\quad k=1,2,\dots\,.
\end{cases}
\]

Let $\tilde A_k^*=\ca_k^* \Pf^{2,1}A$, so we have $\supp \tilde A_k^* \subseteq (k+2)B^*$ and $\Pf^{2,1}A=\sum_{k=0}^\infty \tilde A_k^*$. We prove below that there exist $c>0$ and a sequence $(\lambda_k^*)_k$ in $\ell^1$, both of which do not depend on $A$, such that
\begin{align}
&\label{eq: normbound4} \|\tilde A_k^*\|_2\leq c\|f\|_\infty{\lambda_k^*}\, {\mu\left((k+2)B^*\right)^{-\frac{1}{2}}}.
\end{align}
The result then follows from Theorem~\ref{mainLQatomic}.

To prove \eqref{eq: normbound4}, choose $a$ as in the proof of Lemma~\ref{LemQ11}. Proposition~\ref{ODexp5} applied with ${\tau=\kappa/2}$ then shows that
\[
\|\ca_E(\phi f\tilde\psi_s)(\cD)\ca_F\| \lesssim \|f\|_\infty s^{\frac{\kappa}{2}} e^{-a(r/C_\cD C_{\theta,r})\rho(E,F)}
\]
for all $s\in(0,1]$ and closed subsets $E$ and $F$ of $M$. Now note that if $k\geq0$, then
\[
\rho((k+2)B^*\backslash (k+1)B^*,B) = (k+1)r(B^*)-r(B) \geq kr(B^*).
\]
Using \eqref{ELD}, we then obtain
\begin{align*}
\|\tilde A_k^*\|_2^2 &= \int_M \ca_k^* \bigg|\int_0^{\langle r(B)\rangle} s^{\frac{\kappa}{2}} s^{-\frac{\kappa}{2}}(\phi f\tilde\psi_s)(\cD) A_s \frac{\d s}{s}\bigg|^2 \d\mu \\
&\lesssim r(B)^{\kappa} \int_0^{r(B)}s^{-\kappa} \|\ca_k^*(\phi f\tilde\psi_s)(\cD) A_s\|_2^2 \frac{\d s}{s}\\
&\leq \|f\|_\infty^2 e^{-2a(r/C_\cD C_{\theta,r}) kr(B^*)}r(B)^{\kappa} \|A\|_{L^2_\bullet}^2\\
&\leq \|f\|_\infty^2 e^{-(2a(r/C_\cD C_{\theta,r})-\lambda) kr(B^*)}k^\kappa \mu((k+2)B^*)^{-1},
\end{align*}
which proves \eqref{eq: normbound4} with $\lambda^*_k=e^{-(2a(r/C_\cD C_{\theta,r})-\lambda)k}$.
\end{proof}

\begin{Lem}\label{LemQ12}
Under the assumptions of Theorem~\ref{locMainEst}, there exists $c>0$ such that $\|\Pf^{1,2}A\|_{t^1} \leq c\|f\|_\infty$ for all $A$ that are $L^1_\mathscr{Q}$-atoms.
\end{Lem}
\begin{proof}
Let $A$ be an $L^1_\mathscr{Q}$-atom. There exists a ball $B$ in $M$ with radius $r(B)\geq1$ such that $A$ is supported in $B$ and $\|A\|_2\leq \mu(B)^{-1/2}$. In view of Remark~\ref{RemSmallL1QAtoms} and Theorem~\ref{mainLQatomic}, however, it suffices to assume that $r(B)=1$. In that case, associate $B$ with the characteristic functions $\ca_k$ defined by
\[
\ca_k=
\begin{cases}
\ca_{T^1(B)}&{\rm if}\quad k=0;\\
\ca_{T^1((k+1)B)\setminus T^1(k B)}& {\rm if}\quad k=1,2,\dots\,.
\end{cases}
\]

Let $\tilde A_k=\ca_k\Pf^{1,2}A$, so we have $\supp \tilde A_k\subseteq T^1((k+1)B)$ and
$\Pf^{1,2}A = \sum_{k=0}^\infty{\tilde A_k}$.
We prove below that there exist $c>0$ and a sequence $(\lambda_k)_k$ in $\ell^1$, both of which do not depend on $A$, such that
\begin{align}
&\label{eq: normbound3} \|\tilde A_k\|_{L^2_{\bullet}}\leq c\|f\|_\infty{\lambda_k}\, {\mu\left((k+1)B\right)^{-\frac{1}{2}}}.
\end{align}
The result then follows by Remark~\ref{RemLargeTentAtoms}.

To prove \eqref{eq: normbound3}, choose $\delta$ and $a$ as in the proof of Lemma~\ref{LemQ11}. Proposition~\ref{ODexp5} applied with $\sigma=\delta$ then shows that
\[
\|\ca_E(\psi_t f\tilde\phi)(\cD)\ca_F\| \lesssim \|f\|_\infty t^\delta e^{-a(r/C_\cD C_{\theta,r})\rho(E,F)}
\]
for all $t\in(0,1]$ and closed subsets $E$ and $F$ of $M$. Now note that if $k\geq1$ and $(x,t)$ belongs to $T^1((k+1)B)\setminus T^1(kB)$, then $x$ belongs to $(k+1)B\setminus (k-1)B$, since $t\leq1$ and $r(B)=1$. Using \eqref{ELD}, we then obtain
\begin{align*}
\|{\tilde A_k}\|_{L^2_{\bullet}}^2&=\int^1_0 \|\ca_k(\psi_t f\tilde\phi)(\cD)A\|_2^2\frac{\d t}{t}\\
&\lesssim \|f\|_{\infty}^2e^{-2a(r/C_\cD C_{\theta,r}) k}\int^1_0t^{2\delta}\frac{\d t}{t}\|A\|_2^2 \\
&\lesssim\|f\|_{\infty}^2e^{-2a(r/C_\cD C_{\theta,r}) k}\mu(B)^{-1}\\
&\lesssim\|f\|_{\infty}^2e^{-(2a(r/C_\cD C_{\theta,r})-\lambda)k}k^{\kappa}\mu((k+1)B)^{-1} \\
&\lesssim \|f\|_\infty^2 e^{-(2a(r/C_\cD C_{\theta,r})-\lambda)k/2} \mu((k+1)B)^{-1},
\end{align*}
which proves \eqref{eq: normbound3} with $\lambda_k=e^{-(2a(r/C_\cD C_{\theta,r})-\lambda)k/4}$.
\end{proof}

\begin{Lem}\label{LemQ22}
Under the assumptions of Theorem~\ref{locMainEst}, there exists $c>0$ such that $\|\Pf^{2,2}A\|_{L^1_\mathscr{Q}} \leq c\|f\|_\infty$ for all $A$ that are $L^1_\mathscr{Q}$-atoms.
\end{Lem}
\begin{proof}
Let $A$ be an $L^1_\mathscr{Q}$-atom. As in the proof of Lemma~\ref{LemQ12}, it suffices to assume that there exists a ball $B$ in $M$ with radius $r(B)=1$ such that $A$ is supported in $B$ and $\|A\|_2\leq \mu(B)^{-1/2}$. Associate $B$ with the characteristic functions $\ca_k$ defined by
\[
\ca_k=
\begin{cases}
\ca_{2B}&{\rm if}\quad k=0;\\
\ca_{(k+2)B\setminus (k+1) B}& {\rm if}\quad k=1,2,\dots\,.
\end{cases}
\]

Let $\tilde A_k=\ca_k \Pf^{2,2}A$, so we have $\supp \tilde A_k \subseteq (k+2)B$ and
$\Pf^{2,2}A=\sum_{k=0}^\infty \tilde A_k$. As in the proof of Lemma~\ref{LemQ21}, it is enough to find $c>0$ and a sequence $(\lambda_k)_k$ in~$\ell^1$, both of which do not depend on $A$, such that
\begin{align}
&\label{eq: normbound5} \|\tilde A_k\|_2\leq c\|f\|_\infty{\lambda_k}\, {\mu\left((k+2)B\right)^{-\frac{1}{2}}}.
\end{align}

Choose $a$ as in the proof of Lemma~\ref{LemQ11}. Using Proposition~\ref{ODexp5} and \eqref{ELD}, we then obtain
\[
\|\tilde A_k\|_2 \lesssim \|f\|_\infty e^{-a(r/C_\cD C_{\theta,r}) k}\|A\|_2 \lesssim \|f\|_\infty e^{-(2a(r/C_\cD C_{\theta,r})-\lambda)k/2} k^{\kappa/2} \mu((k+2)B)^{-\frac{1}{2}},
\]
which proves \eqref{eq: normbound5} with $\lambda_k=e^{-(2a(r/C_\cD C_{\theta,r})-\lambda)k/4}$.
\end{proof}

\section{Local Hardy Spaces $h^p_\cD(\wedge T^*M)$}\label{Sectionh^1defintion}
Throughout this section, let $\kappa,\lambda\geq0$ and suppose that $M$ is a complete Riemannian manifold satisfying \eqref{ELD}. Also, let $\omega\in[0,\pi/2)$ and $R\geq0$ and suppose that $\cD$ is a closed densely-defined operator on $L^2(\wedge T^*M)$ of $\text{type }S_{\omega,R}$ satisfying hypotheses (H1-3) from Section~\ref{SectionODE} with constants $C_{\theta,r}>0$ and $C_\cD>0$, where $C_{\theta,r}$ is defined for each $\theta\in(\omega,\pi/2)$ and $r>R$.

The $\Phi$-class of holomorphic functions is introduced to prove a variant of the Calder\'{o}n reproducing formula. This allows us to characterise $L^2(\wedge T^*M)$ in terms of square functions involving the operators $\cQ^\cD_{\psi,\phi}$ and $\cS^\cD_{\psi,\phi}$ from the previous section, where $\phi$ is restricted to the $\Phi$-class. We combine this with Theorem~\ref{locMainEst} to define local Hardy spaces of differential forms $h^p_\cD(\wedge T^*M)$ for all $p\in[1,\infty]$ in terms of square functions and a retraction on the space $t^{p}(\wedge T^*M\times(0,1])\oplus L^p_{\mathscr{Q}}(\wedge T^*M)$. In what follows, we only consider spaces of differential forms and usually omit writing $\wedge T^*M$ and $\wedge T^*M\times(0,1]$.

\begin{Def}\label{DefPhiClass}
Given $\theta\in(0,\pi/2)$, $r>0$ and $\beta>0$, define $\Phi^\beta(S_{\theta,r}^o)$ to be the set of all $\phi\in\Theta^\beta(S_{\theta,r}^o)$ with the following properties:
\begin{enumerate}
\item\label{Phi1} For all $z$ in $S_{\theta,r}^o$, $\phi(z)\neq0$;
\item\label{Phi2} $\inf_{z\in D_r^o}|\phi(z)|\neq0$;
\item\label{Phi3} $\sup_{t\geq1}|\phi_t(z)| \lesssim |\phi(z)|$ for all $z$ in $S_{\theta,r}^o\setminus D_r$.
\end{enumerate}
Also, let $\Phi(S_{\theta,r}^o)=\bigcup_{\beta>0} \Phi^\beta(S_{\theta,r}^o)$.
\end{Def}

For example, if $\theta\in(0,\pi/2)$, $0<r<\sqrt{a}$ and $\beta>0$, then the functions $e^{-\sqrt{z^2+a}}$, $e^{-z^2}$ and $(z^2+a)^{-\beta}$ are in $\Phi^\beta(S_{\theta,r}^o)$. We now require the following version of the Calder\'{o}n reproducing formula.

\begin{Prop}[Calder\'{o}n reproducing formula] \label{sq=I}
Let $\theta\in(\omega,\pi/2)$ and $r>R$. Given $\alpha,\beta,\gamma,\sigma,\tau,\upsilon>0$ and nondegenerate $\psi\in\Psi_\alpha^\beta(S_{\theta,r}^o)$ and $\phi\in\Phi^\gamma(S_{\theta,r}^o)$, there exist $\tilde\psi\in\Psi_\sigma^\tau(S_{\theta,r}^o)$ and $\tilde\phi\in\Theta^\upsilon(S_{\theta,r}^o)$ such that
\begin{equation}\label{sectoridentity}
\int_0^1\tilde\psi_{t}(z)\psi_{t}(z)\frac{\d t}{t} + \tilde\phi(z)\phi(z)=1
\end{equation}
for all $z\in S_{\theta,r}^o$. Moreover, we have $\cS^\cD_\svt{{\psi,\phi}}\cQ^\cD_{\tilde\psi,\tilde\phi} =\cS^\cD_{\tilde\psi,\tilde\phi}\cQ^\cD_\svt{\psi,\phi}=I$ on $L^2$.
\end{Prop}
\begin{proof}
Given $f\in H^\infty(S_{\theta,r}^o)$, let $f_{-}(z)=f(-z)$ and $f^*(z)=\overline{f(\bar z)}$ for all $z\in S_{\theta,r}^o$.
Choose integers $M$ and $N$ so that $4M\geq\max(\frac{\sigma}{\alpha},\frac{\tau}{\beta})+1$ and $4M\beta+(4N-1)\gamma\geq \upsilon$. Let $c=\int^\infty_0|\psi(t)\psi(-t)|^{2M}|\phi(t)\phi(-t)|^{2N}\ \frac{\d t}{t}$ and define the functions
\[
\tilde{\psi}= c^{-1}\psi^{M-1}(\psi^*\psi_{-}\psi_{-}^*)^{M}(\phi\phi^*\phi_{-}\phi_{-}^*)^N
\quad\text{and}\quad
\tilde\phi= \frac{1}{\phi}\left(1-\int_0^1 \tilde\psi_{t}\psi_{t}\frac{\d t}{t}\right),
\]
in which case \eqref{sectoridentity} is immediate and $\tilde{\psi}\in\Psi_{\alpha(4M-1)}^{\beta(4M-1)}(S_{\theta,r}^o) \subseteq\Psi_\sigma^\tau(S_{\theta,r}^o)$. The function $\tilde\phi$ is holomorphic on $S_{\theta,r}^o$ by Morera's Theorem, since $\phi(z)\neq 0$ for all $z\in S_{\theta,r}^o$, and bounded on $D^o_r$, since $\inf_{z\in D_r^o}|\phi(z)|\neq 0$. A change of variable shows that
$
\int^\infty_0\tilde\psi_t(x){\psi}_t(x)\frac{\d t}{t}=1
$
for all $x\in\R\setminus\{0\}$, and since $z\mapsto \int^\infty_0\tilde\psi_t(z){\psi}_t(z)\frac{\d t}{t}$ is holomorphic on $S_\theta^o$, we must have
$
\int^\infty_0\tilde\psi_{t}(z)\psi_{t}(z)\frac{\d t}{t}=1
$
for all $z\in S_{\theta}^o$. It then follows from property \eqref{Phi3} in Definition~\ref{DefPhiClass} that
\begin{align*}
|\tilde\phi(z)| \lesssim \sup_{t\geq1}\frac{|\phi_t(z)| }{|\phi(z)|} \int_1^\infty (t|z|)^{-4M\beta-(4N-1)\gamma} \frac{\d t}{t}
\lesssim |z|^{-\upsilon}
\end{align*}
for all $z\in S^o_\theta$, so $\tilde\phi\in \Theta^{\upsilon}(S_{\theta,r}^o)$.

The last part of the proposition follows from holomorphic functional calculus, since $\cD$ satisfies (H1) and (H2). Further details are in Lemma~2.9 of \cite{Morris2010(1)}.
\end{proof}

Given $\psi\in\Psi(S_{\theta,r}^o)$ and $\phi\in\Phi(S_{\theta,r}^o)$, since $\cD$ satisfies (H1) and (H2), the main result of \cite{Morris2010(1)} shows that the local quadratic estimate
\begin{equation}\label{eq: h2equiv1}
\|u\|_2\eqsim \|\cQ^\cD_{\psi,\phi}u\|_{L^2_\bullet\oplus L^2}
\end{equation}
holds for all $u\in L^2$. There also exists $\tilde\psi\in\Psi(S_{\theta,r}^o)$ and $\tilde\phi\in\Theta(S_{\theta,r}^o)$ such that $\cS^\cD_\svt{{\psi,\phi}}\cQ^\cD_{\tilde\psi,\tilde\phi}=I$ on $L^2$ by Proposition~\ref{sq=I}. This shows that $L^2=\cS_{\psi,\phi}^\cD(L^2_\bullet\oplus L^2)$ with
\begin{equation}\label{eq: h2equiv2}
\|u\|_2\eqsim\inf\{\|U\|_{L^2_\bullet\oplus L^2} : U\in L^2_\bullet\oplus L^2 \text{ and } u=\cS^\cD_{\psi,\phi}U\}
\end{equation}
for all $u\in L^2$, since both $\cS^\cD_\svt{{\psi,\phi}}$ and $\cQ^\cD_{\tilde\psi,\tilde\phi}$ are bounded operators. These characterisations of $L^2$ help to motivate our definition of the local Hardy spaces. In particular, we define $h^p_\cD$ by replacing $L^2_\bullet\oplus L^2$ with $t^p\oplus L^p_\mathscr{Q}$ in \eqref{eq: h2equiv1} and \eqref{eq: h2equiv2}, and suitably extending the operators $\cQ^\cD_{\psi,\phi}$ and $\cS^\cD_{\psi,\phi}$.

There is a fundamental difference here from the Hardy spaces $H^p_D$ in \cite{AMcR}. The reproducing formula used to define $H^p_D$ is based on selecting $\psi$ and $\tilde\psi$ in $\Psi(S_\theta^o)$ such that $\int_0^\infty\tilde\psi_{t}(z)\psi_{t}(z)\frac{\d t}{t}=1$ for all $z\in S_\theta^o$. The decay of the $\Psi(S_\theta^o)\text{-class}$ functions near the origin implies that $\int_0^\infty\tilde\psi_{t}(D)\psi_{t}(D)\frac{\d t}{t}=\mathsf{P}_{\overline{\Ran(D)}}$, where $\mathsf{P}_{\overline{\Ran(D)}}$ denotes the projection onto the closure of $\Ran(D)$, as given by the Hodge decomposition ${L^2=\overline{\Ran(D)}\oplus\Nul(D)}$. This leads the authors of \cite{AMcR} to define $H^2_D$ to be $\overline{\Ran(D)}$. Identity \eqref{sectoridentity}, by contrast, holds on a neighbourhood $D_r^o$ of the origin as well as on the bisector $S_\theta^o$, and since the $\Phi$-class functions are nonzero at the origin, we get $\cS^\cD_{\tilde\psi,\tilde\phi}\cQ^\cD_\svt{\psi,\phi}=I$ on all of $L^2$. The local Hardy spaces are therefore not subject to the null space considerations that one encounters with the Hardy spaces. In fact, we show that $h^2_\cD$ can be identified with $L^2$.

We now define an ambient space $h^0_\cD$ in order to have $h^p_\cD\subseteq h^0_\cD$ for all $p\in[1,\infty]$. This requires that we recall the results concerning the spaces $t^1+\tilde t^\infty$ and $L^1_{\mathscr Q}+\tilde L^\infty_{\mathscr Q}$ in Corollaries~\ref{Cor: tp2Interp} and \ref{Cor: LpQInterp}.

\begin{Def}\label{Def: AmbientSpace}
Let $\theta\in(\omega,\pi/2)$, $r>R$ and $\beta>{\kappa}/{2}$ such that $r/C_\cD C_{\theta,r}>{\lambda}/{2}$. Fix $\eta\in\Psi_\beta^\beta(S^o_{\theta,r})$ and $\varphi\in\Phi^\beta(S^o_{\theta,r})$ satisfying
\[
\int_0^1\eta_{t}^2(z)\frac{\d t}{t} + \varphi^2(z)=1
\]
for all $z\in S_{\theta,r}^o$. The ambient space $h_\cD^0$ is defined to be the abstract completion of $L^2$ under the norm defined by
\[
\|u\|_{h^0_\cD} = \|\cQ^\cD_{\eta,\varphi} u\|_{(t^1+\tilde t^\infty) \oplus (L^1_{\mathscr Q}+\tilde L^\infty_{\mathscr Q})}
\]
for all $u\in L^2$. This provides an identification of $L^2$ with a dense subspace of $h_\cD^0$. The functions $\eta$ and $\varphi$ remain fixed for the remainder of this section.
\end{Def}

To check that $\|\cdot\|_{h^0_\cD}$ is a norm on $L^2$, suppose that $\|u\|_{h^0_\cD}=0$ for some $u\in L^2$. It follows that $\cQ^\cD_{\eta,\varphi} u=0$, and since $\cQ^\cD_{\eta,\varphi} u\in L^2_\bullet\oplus L^2$, the equivalence in \eqref{eq: h2equiv1} guarantees that $u=0$, as required.

The following result allows us to define the local Hardy spaces.
\begin{Prop}\label{Prop: QandSext}
The operators $\cQ^\cD_{\eta,\varphi}$ and $\cS^\cD_{\eta,\varphi}$ have bounded extensions
\[
\tilde\cQ^\cD_{\eta,\varphi}:h^0_\cD\rightarrow(t^1+\tilde t^\infty) \oplus (L^1_{\mathscr Q}+\tilde L^\infty_{\mathscr Q})
\]
and
\[
\tilde\cS^\cD_{\eta,\varphi}:(t^1+\tilde t^\infty) \oplus (L^1_{\mathscr Q}+\tilde L^\infty_{\mathscr Q})\rightarrow h^0_\cD
\]
such that $\tilde\cS^\cD_{\eta,\varphi}\tilde\cQ^\cD_{\eta,\varphi}=I$ on $h^0_\cD$, the restriction $\tilde\cQ^\cD_{\eta,\varphi}\tilde\cS^\cD_{\eta,\varphi}:t^p\oplus L^p_{\mathscr Q}\rightarrow t^p\oplus L^p_{\mathscr Q}$ is bounded for each $p\in[1,\infty)$, and the restriction $\tilde\cQ^\cD_{\eta,\varphi}\tilde\cS^\cD_{\eta,\varphi}: \tilde t^\infty\oplus \tilde L^\infty_{\mathscr Q}\rightarrow \tilde t^\infty\oplus \tilde L^\infty_{\mathscr Q}$ is bounded.
\end{Prop}
\begin{proof}
We immediately have
\[
\|\cQ^\cD_{\eta,\varphi} u\|_{(t^1+\tilde t^\infty) \oplus (L^1_{\mathscr Q}+\tilde L^\infty_{\mathscr Q})}
= \|u\|_{h^0_\cD}
\]
for all $u\in L^2$, and since $L^2$ is identified with a dense subspace of $h^0_\cD$, the bounded extension $\tilde\cQ^\cD_{\eta,\varphi}$ exists.

It follows from Theorem~\ref{locMainEst} that  $\cQ^\cD_{\eta,\varphi}\cS^\cD_{\eta,\varphi}$ has a bounded extension from $t^p\oplus L_{\mathscr Q}^p$ to $t^p\oplus L_{\mathscr Q}^p$ for each $p\in[1,\infty]$, and hence from $\tilde t^\infty\oplus \tilde L_{\mathscr Q}^\infty$ to $\tilde t^\infty\oplus \tilde L_{\mathscr Q}^\infty$ as well. Moreover, the extensions coincide with a single bounded operator
\[
\cP:(t^1+\tilde t^\infty) \oplus (L^1_{\mathscr Q}+\tilde L^\infty_{\mathscr Q}) \rightarrow (t^1+\tilde t^\infty) \oplus (L^1_{\mathscr Q}+\tilde L^\infty_{\mathscr Q})
\]
such that the restriction of $\cP$ to $t^p\oplus L_{\mathscr Q}^p$ coincides with the extension of $\cQ^\cD_{\eta,\varphi}\cS^\cD_{\eta,\varphi}$ to  $t^p\oplus L_{\mathscr Q}^p$ for each $p\in[1,\infty)$, and the restriction of $\cP$ to $\tilde t^\infty\oplus \tilde L_{\mathscr Q}^\infty$ coincides with the extension of $\cQ^\cD_{\eta,\varphi}\cS^\cD_{\eta,\varphi}$ to $\tilde t^\infty\oplus \tilde L_{\mathscr Q}^\infty$. Therefore, we have
\[
\|\cS^\cD_{\eta,\varphi} U\|_{h^0_\cD} =
\|\cP U\|_{(t^1+\tilde t^\infty) \oplus (L^1_{\mathscr Q}+\tilde L^\infty_{\mathscr Q})}
\lesssim \|U\|_{(t^1+\tilde t^\infty) \oplus (L^1_{\mathscr Q}+\tilde L^\infty_{\mathscr Q})}
\]
for all $U\in t^2\oplus L^2_\mathscr{Q}$, and since $t^2\oplus L^2_\mathscr{Q}$ is dense in $(t^1+\tilde t^\infty) \oplus (L^1_{\mathscr Q}+\tilde L^\infty_{\mathscr Q})$ by Corollaries~\ref{Cor: tp2Interp} and \ref{Cor: LpQInterp}, the bounded extension $\tilde\cS^\cD_{\eta,\varphi}$ exists.

It follows that $\tilde\cS^\cD_{\eta,\varphi}\tilde\cQ^\cD_{\eta,\varphi}$ is bounded on $h^0_\cD$. The formula $\cS^\cD_{\eta,\varphi}\cQ^\cD_{\eta,\varphi}=I$ holds on~$L^2$ by Proposition~\ref{sq=I}, so by density $\tilde\cS^\cD_{\eta,\varphi}\tilde\cQ^\cD_{\eta,\varphi}=I$ on $h^0_\cD$.

It also follows that $\tilde\cQ^\cD_{\eta,\varphi}\tilde\cS^\cD_{\eta,\varphi}$ is bounded on $(t^1+\tilde t^\infty) \oplus (L^1_{\mathscr Q}+\tilde L^\infty_{\mathscr Q})$, and that $\tilde\cQ^\cD_{\eta,\varphi}\tilde\cS^\cD_{\eta,\varphi}=\cP$ on $(t^p\cap t^2)\oplus (L^p_{\mathscr Q} \cap L^2_{\mathscr Q})$ for  $p\in[1,\infty)$, and on $(\tilde t^\infty\cap t^2)\oplus(\tilde L^\infty_{\mathscr Q} \cap L^2_{\mathscr Q})$. Now suppose that $p\in[1,\infty)$ and that $u\in t^p\oplus L^p_{\mathscr Q}$. There exists a sequence $(u_n)_n$ in $(t^p\cap t^2)\oplus (L^p_{\mathscr Q} \cap L^2_{\mathscr Q})$ that converges to $u$ in $t^p\oplus L^p_{\mathscr Q}$ by Propositions~\ref{Prop: tpt2denseintp} and~\ref{Prop: LpQL2denseinLp}. The continuity of the embedding
$t^p\oplus L^p_{\mathscr Q}\subseteq (t^1+\tilde t^\infty) \oplus (L^1_{\mathscr Q}+\tilde L^\infty_{\mathscr Q})$,
which is a consequence of the interpolation in Corollaries~\ref{Cor: tp2Interp} and \ref{Cor: LpQInterp}, then implies that
\begin{align*}
\|\cP u-\tilde\cQ^\cD_{\eta,\varphi}\tilde\cS^\cD_{\eta,\varphi}u&\|_{(t^1+\tilde t^\infty) \oplus (L^1_{\mathscr Q}+\tilde L^\infty_{\mathscr Q})}\\
&\quad\leq \|\cP(u-u_n)\|_{t^p\oplus L^p_{\mathscr Q}} + \|\tilde\cQ^\cD_{\eta,\varphi}\tilde\cS^\cD_{\eta,\varphi}(u_n-u)\|_{(t^1+\tilde t^\infty) \oplus (L^1_{\mathscr Q}+\tilde L^\infty_{\mathscr Q})}
\end{align*}
for all $n\in\N$. Therefore, we have $\tilde\cQ^\cD_{\eta,\varphi}\tilde\cS^\cD_{\eta,\varphi}=\cP $ on $t^p\oplus L^p_{\mathscr Q}$ for all $p\in[1,\infty)$. We also have $\tilde\cQ^\cD_{\eta,\varphi}\tilde\cS^\cD_{\eta,\varphi}=\cP$ on $\tilde t^\infty\oplus\tilde L^\infty_{\mathscr Q}$ by the density properties in Corollaries~\ref{Cor: tp2Interp} and \ref{Cor: LpQInterp}, so the result follows.
\end{proof}

We now define the local Hardy spaces.
\begin{Def}\label{Def: local.Hardy}
For each $p\in[1,\infty)$, the \textit{local Hardy space} $h^p_\cD$ consists of all $u\in h^0_\cD$ with
\[
\|u\|_{h^p_\cD} = \|\tilde\cQ^\cD_{\eta,\varphi}u\|_{t^p\oplus L^p_{\mathscr Q}} < \infty.
\]
For $p=\infty$, the \textit{local Hardy space} space $h^\infty_\cD$ consists of all $u\in h^0_\cD$ such that $\tilde\cQ^\cD_{\eta,\varphi}u \in \tilde t^\infty \oplus \tilde L^\infty_{\mathscr Q}$ with
\[
\|u\|_{h^\infty_\cD} = \|\tilde\cQ^\cD_{\eta,\varphi}u\|_{\tilde t^\infty \oplus \tilde L^\infty_{\mathscr Q}} = \|\tilde\cQ^\cD_{\eta,\varphi}u\|_{t^\infty \oplus L^\infty_{\mathscr Q}}.
\]
\end{Def}

The dual of $h^1_\cD$ should be identified with a bmo type space, as in the classical case in \cite{Goldberg01}. To construct the ambient space $h^0_\cD$, however, we used the closed subspace $\tilde t^\infty \oplus \tilde L^\infty_\mathscr{Q}$ of $t^\infty \oplus L^\infty_\mathscr{Q}$. This suggests that $h^\infty_\cD$ can only be identified with a closed subspace of the dual of $h^1_\cD$. Therefore, we do not denote $h^\infty_\cD$ by $\text{bmo}_\cD$ and we postpone the construction of an appropriate $\text{bmo}_\cD$ space to the sequel. Note that we do identify the dual of $h^p_\cD$ for all $p\in(1,\infty)$ in Theorem \ref{Thm: hp.duality} below.

The local Hardy spaces are Banach spaces for all $p\in[1,\infty]$. To see this, suppose that $p\in[1,\infty)$ and that $(u_n)_n$ is a Cauchy sequence in $h^p_\cD$. Then there exists $v$ in $t^p\oplus L^p_{\mathscr Q}$ such that $\lim_n\|\tilde\cQ^\cD_{\eta,\varphi}u_n-v\|_{t^p\oplus L^p_{\mathscr Q}}=0$. Moreover, the embedding $t^p\oplus L^p_{\mathscr Q}\subseteq (t^1+\tilde t^\infty) \oplus (L^1_{\mathscr Q}+\tilde L^\infty_{\mathscr Q})$ implies that $\lim_n\|\tilde\cQ^\cD_{\eta,\varphi}u_n-v\|_{(t^1+\tilde t^\infty) \oplus (L^1_{\mathscr Q}+\tilde L^\infty_{\mathscr Q})}=0$, and hence that there exists $u$ in $h^0_\cD$ such that $\lim_n\|u_n-u\|_{h^0_\cD}=0$. Therefore, we have $\tilde\cQ^\cD_{\eta,\varphi}u=v\in t^p\oplus L^p_{\mathscr Q}$, which implies that $u\in h^p_\cD$ and that $\lim_n\|u_n-u\|_{h^p_\cD}=0$. The proof for $p=\infty$ is the same but with $\tilde t^\infty\oplus \tilde L^\infty_\mathscr{Q}$ instead of $t^p\oplus L^p_\mathscr{Q}$.

The definition of the ambient space allowed us to identify $L^2$ with a dense subspace of $h^0_\cD$. It now follows from \eqref{eq: h2equiv1} that $L^2\subseteq h^2_\cD$ under this identification. In fact, we have $L^2=h^2_\cD$ under this identification by \eqref{eq: h2equiv2} and the following proposition, which gives an equivalent definition for $h^p_{\cD}$.

\begin{Prop}\label{prop: normequiv}
If $p\in[1,\infty)$, then $h^p_\cD=\tilde\cS^\cD_{\eta,\varphi}(t^p\oplus L^p_\mathscr{Q})$ and
\[
\|u\|_{h^p_\cD} \eqsim
\inf \{\|U\|_{t^p\oplus L^p_\mathscr Q} : U\in t^p\oplus L^p_\mathscr Q \text{ and } u=\tilde\cS^\cD_{\eta,\varphi}U \}.
\]
If $p=\infty$, then the above holds with $\tilde t^\infty\oplus \tilde L^\infty_\mathscr{Q}$ instead of $t^p\oplus L^p_\mathscr{Q}$.
\end{Prop}
\begin{proof}
Suppose that $p\in[1,\infty)$. Proposition~\ref{Prop: QandSext} shows that $\tilde\cS^\cD_{\eta,\varphi}\tilde\cQ^\cD_{\eta,\varphi}=I$
on $h^0_\cD$, and that the restricted operators
\[
\tilde\cQ^\cD_{\eta,\varphi}:h^p_\cD\rightarrow t^p\oplus L^p_\mathscr{Q}
\quad\text{and}\quad
\tilde\cS^\cD_{\eta,\varphi}:t^p\oplus L^p_\mathscr{Q} \rightarrow h^p_\cD
\]
are bounded. Therefore, we have $h^p_\cD=\tilde\cS^\cD_{\eta,\varphi}(t^p\oplus L^p_\mathscr{Q})$ with
\[
\inf_{\substack{U\in t^p\oplus L^p_\mathscr Q;\\ u=\tilde\cS^\cD_{\eta,\varphi}U}} \|U\|_{t^p\oplus L^p_\mathscr Q} \leq \|\tilde\cQ^\cD_{\eta,\varphi}u\|_{t^p\oplus L^p_\mathscr Q} = \|u\|_{h^p_\cD} = \|\tilde\cS^\cD_{\eta,\varphi}V\|_{h^p_\cD} \lesssim \|V\|_{t^p\oplus L^p_\mathscr Q}
\]
for all $V\in t^p\oplus L^p_\mathscr Q$ satisfying $u=\tilde\cS^\cD_{\eta,\varphi}V.$

The proof for $p=\infty$ is the same but with $\tilde t^\infty\oplus \tilde L^\infty_\mathscr{Q}$ instead of $t^p\oplus L^p_\mathscr{Q}$.
\end{proof}

This leads us to the following density properties of the local Hardy spaces.
\begin{Cor}\label{Cor: hphqDensehp}
For all $p\in[1,\infty]$ and $q\in[1,\infty)$, the set $h^p_\cD\cap h^q_\cD$ is dense in $h^p_\cD$. Moreover, for all $p, q\in[1,\infty)$, we have
$
h^p_\cD\cap h^q_\cD = \tilde\cS^\cD_{\eta,\varphi}((t^p\cap t^q)\oplus(L^p_{\mathscr Q}\cap L^q_{\mathscr Q})).
$
This also holds for $p=\infty$ but with $\tilde t^\infty$ and $\tilde L^\infty_{\mathscr Q}$ instead of $t^p$ and $L^p_{\mathscr Q}$.
\end{Cor}
\begin{proof}
If $p,q\in[1,\infty)$, then $h^p_\cD=\tilde\cS^\cD_{\eta,\varphi}(t^p\oplus L^p_\mathscr{Q})$ by Proposition~\ref{prop: normequiv}, so the density of $h^p_\cD\cap h^q_\cD$ in $h^p_\cD$ follows from the density properties in Propositions~\ref{Prop: tpt2denseintp} and \ref{Prop: LpQL2denseinLp}. If $p=\infty$, then the result follows from the density properties in Corollaries~\ref{Cor: tp2Interp} and~\ref{Cor: LpQInterp}.

If $p,q\in[1,\infty)$ and $u\in h^p_\cD\cap h^q_\cD$, then by the reproducing formula in Proposition~\ref{Prop: QandSext}, we have
\[
u=\tilde\cS^\cD_{\eta,\varphi}\tilde\cQ^\cD_{\eta,\varphi}u
\in \tilde\cS^\cD_{\eta,\varphi}((t^p\cap t^q)\oplus(L^p_{\mathscr Q}\cap L^q_{\mathscr Q})),
\]
since $\tilde\cQ^\cD_{\eta,\varphi}u\in (t^p\cap t^q)\oplus (L^p_{\mathscr Q}\cap L^q_{\mathscr Q})$. If $p=\infty$, then this holds with $\tilde t^\infty$ and $\tilde L^\infty_{\mathscr Q}$ instead of $t^p$ and $L^p_{\mathscr Q}$, which completes the proof.
\end{proof}

The interpolation results for the local tent spaces $t^p$ and the spaces $L^p_{\mathscr Q}$ allow us to interpolate the local Hardy spaces.

\begin{Thm}\label{Thm: Interpolation}
If $\theta\in(0,1)$ and $1\leq p_0 < p_1\leq \infty$, then
\[
[h^{p_0}_\cD,h^{p_1}_\cD]_{\theta}=h^{p_\theta}_\cD,
\]
where $1/p_\theta=(1-\theta)/p_0+\theta/p_1$ and $[\cdot,\cdot]_\theta$ denotes complex interpolation.
\end{Thm}
\begin{proof}
The interpolation space $[h^{p_0}_\cD,h^{p_1}_\cD]_{\theta}$ is well-defined because it is an immediate consequence of Definition~\ref{Def: local.Hardy} that
$h^p_\cD \subseteq h^0_\cD$ for all $p\in[1,\infty]$.

Suppose that $p_1\in(1,\infty)$. Theorems~\ref{Thm: tp2propertiesInterp} and \ref{Thm: LpQpropertiesInterp} show that
\[
[t^{p_0}\oplus L^{p_0}_{\mathscr Q}, t^{p_1}\oplus L^{p_1}_{\mathscr Q}]_{\theta}
= t^{p_\theta}\oplus L^{p_\theta}_{\mathscr Q}.
\]
Proposition~\ref{Prop: QandSext} shows that the reproducing formula $\tilde\cS^\cD_{\eta,\varphi}\tilde\cQ^\cD_{\eta,\varphi}=I$ holds on $h^0_\cD$, and that the restricted operators
\[
\tilde\cQ^\cD_{\eta,\varphi}:h^p_\cD\rightarrow t^p\oplus L^p_\mathscr{Q}
\quad\text{and}\quad
\tilde\cS^\cD_{\eta,\varphi}:t^p\oplus L^p_\mathscr{Q} \rightarrow h^p_\cD
\]
are bounded for all $p\in[1,\infty)$. It follows by Theorem~1.2.4 of \cite{T}, which concerns the interpolation of spaces related by a retraction, that $\tilde\cQ^\cD_{\eta,\varphi}: [h^{p_0}_\cD,h^{p_1}_\cD]_\theta \rightarrow t^{p_\theta}\oplus L^{p_\theta}_{\mathscr Q}$ is an isomorphism onto the subspace
\[
\tilde\cQ^\cD_{\eta,\varphi}\tilde\cS^\cD_{\eta,\varphi}(t^{p_\theta}\oplus L^{p_\theta}_{\mathscr Q}) = \tilde\cQ^\cD_{\eta,\varphi}(h^{p_\theta}_\cD)
\]
for all $p_0\in[1,p_1)$, where the equality is given by Proposition~\ref{prop: normequiv}. The reproducing formula then implies that $[h^{p_0}_\cD,h^{p_1}_\cD]_\theta = h^{p_\theta}_\cD$.

The proof for $p_1=\infty$ is the same but with $\tilde t^\infty\oplus \tilde L^\infty_\mathscr{Q}$ instead of $t^{p_1}\oplus L^{p_1}_\mathscr{Q}$, and it relies on Corollaries~\ref{Cor: tp2Interp} and \ref{Cor: LpQInterp}.
\end{proof}

The next result is an application of the interpolation of the local Hardy spaces.
\begin{Lem}\label{Lem: ThetaBds}
Let $\theta\in(\omega,\tfrac{\pi}{2})$, $r>R$ and $\beta>\kappa/2$ such that $r/C_\cD C_{\theta,r}>~\lambda/2$. For each $\psi\in\Psi^{\beta}(S_{\theta,r}^o),$ $\tilde{\psi}\in\Psi_{\beta}(S_{\theta,r}^o),$ $\phi\in\Theta^{\beta}(S_{\theta,r}^o)$ and $\tilde{\phi}\in\Theta(S_{\theta,r}^o)$, the following hold:
\begin{enumerate}
\item The operators $\cQ^\cD_{\psi,\phi}$ and $\cS^\cD_{\tilde\psi,\tilde\phi}$ have bounded extensions $\tilde\cQ^\cD_{\psi,\phi}:h^p_\cD\rightarrow t^p\oplus L^p_{\mathscr Q}$ and $\tilde\cS^\cD_{\tilde\psi,\tilde\phi}:t^p\oplus L^p_{\mathscr Q}\rightarrow h^p_\cD$
    for all $p\in[1,2]$.
\item The operators $\cQ^\cD_{\tilde\psi,\tilde\phi}$ and $\cS^\cD_{\psi,\phi}$ have bounded extensions $\tilde\cQ^\cD_{\tilde\psi,\tilde\phi}:h^p_\cD\rightarrow t^p\oplus L^p_{\mathscr Q}$ and $\tilde\cS^\cD_{\psi,\phi}:t^p\oplus L^p_{\mathscr Q}\rightarrow h^p_\cD$ for all $p\in[2,\infty)$. This also holds for $p=\infty$ but with $\tilde t^\infty\oplus \tilde L^\infty_\mathscr{Q}$ instead of $t^\infty\oplus L^\infty_\mathscr{Q}$.
\end{enumerate}
\end{Lem}
\begin{proof}
If $u\in h^1_\cD \cap L^2$, then $\cQ^\cD_{\eta,\varphi} u\in (t^1\cap t^2) \oplus (L^1_{\mathscr Q}\cap L^2_{\mathscr Q})$ and ${u=\cS^\cD_{\eta,\varphi}\cQ^\cD_{\eta,\varphi} u}$, so by Theorem \ref{locMainEst} we have
\[
\|\cQ^\cD_{\psi,\phi}u\|_{t^1\oplus L^1_\mathscr Q}
= \|\cQ^\cD_{\psi,\phi}\cS^\cD_{\eta,\varphi}\cQ^\cD_{\eta,\varphi} u\|_{t^1\oplus L^1_\mathscr Q}
\lesssim  \|u\|_{h^1_\cD}.
\]
The set $h^1_\cD \cap L^2$ is dense in $h^1_\cD$ by Corollary \ref{Cor: hphqDensehp}, so the bounded extension $\tilde\cQ^\cD_{\psi,\phi}$ exists for $p=1$, and hence for all $p\in[1,2]$ by interpolation.

If $U\in(t^1\cap t^2)\oplus(L^1_{\mathscr Q}\cap L^2_{\mathscr Q})$, then by Theorem~\ref{locMainEst} we have
\[
\|\cS^\cD_{\tilde\psi,\tilde\phi}U\|_{h^1_\cD}
= \|\cQ^\cD_\svt{\eta,\varphi}\cS^\cD_{\tilde\psi,\tilde\phi} U\|_{t^1\oplus L^1_\mathscr Q}
\lesssim  \|U\|_{t^1\oplus L^1_\mathscr Q}.
\]
The density properties in Propositions~\ref{Prop: tpt2denseintp} and \ref{Prop: LpQL2denseinLp} then imply that the bounded extension $\tilde\cS^\cD_{\tilde\psi,\tilde\phi}$ exists for $p=1$, and hence for all $p\in[1,2]$ by interpolation, which proves (1). The proof of (2) is similar.
\end{proof}

This allows us to construct a family of equivalent norms on the local Hardy spaces.

\begin{Prop}\label{Prop: NormeEquivChangingQandS}
Let $\theta\in(\omega,\tfrac{\pi}{2})$, $r>R$ and $\beta>\kappa/2$ so that $r/C_\cD C_{\theta,r}>~\lambda/2$. For each $\psi\in\Psi^{\beta}(S_{\theta,r}^o),$ $\tilde{\psi}\in\Psi_{\beta}(S_{\theta,r}^o),$ $\phi\in\Phi^{\beta}(S_{\theta,r}^o)$ and $\tilde{\phi}\in\Phi(S_{\theta,r}^o)$, the following hold:
\begin{enumerate}
\item The extension operators from Lemma~\ref{Lem: ThetaBds} satisfy $h^p_\cD= \tilde\cS^\cD_{\tilde\psi,\tilde\phi}(t^p\oplus L^p_{\mathscr Q})$ and
\[
\|u\|_{h^p_\cD} \eqsim \|\tilde\cQ^\cD_{\psi,\phi}u\|_{t^p\oplus L^p_{\mathscr Q}}
\eqsim \inf_{u=\tilde\cS^\cD_{\tilde\psi,\tilde\phi}U} \|U\|_{t^p\oplus L^p_\mathscr Q}
\]
for all $u\in h^p_\cD$ and $p\in[1,2]$.
\item The extension operators from Lemma~\ref{Lem: ThetaBds} satisfy $h^p_\cD= \tilde\cS^\cD_{\psi,\phi}(t^p\oplus L^p_{\mathscr Q})$ and
\[
\|u\|_{h^p_\cD} \eqsim \|\tilde\cQ^\cD_{\tilde\psi,\tilde\phi}u\|_{t^p\oplus L^p_{\mathscr Q}}
\eqsim \inf_{u=\tilde\cS^\cD_{\psi,\phi}U} \|U\|_{t^p\oplus L^p_\mathscr Q}
\]
for all $u\in h^p_\cD$ and $p\in[2,\infty)$. This also holds for $p=\infty$ but with $\tilde t^\infty\oplus \tilde L^\infty_\mathscr{Q}$ instead of $t^\infty\oplus L^\infty_\mathscr{Q}$.
\end{enumerate}
\end{Prop}
\begin{proof}
Suppose that $p\in[1,2]$. Proposition~\ref{sq=I} shows that there exists $\psi'\in\Psi_\beta (S_{\theta,r}^o)$ and $\phi'\in\Theta(S_{\theta,r}^o)$ such that $\cS^\cD_{\psi',\phi'}\cQ^\cD_{\psi,\phi}=I$ on $L^2$. Lemma~\ref{Lem: ThetaBds} then shows that
\[
\|u\|_{h^p_\cD}
= \|\cQ^\cD_{\eta,\varphi}\cS^\cD_{\psi',\phi'}\cQ^\cD_{\psi,\phi} u\|_{t^p\oplus L^p_\mathscr Q}
\lesssim  \|\cQ^\cD_{\psi,\phi}u\|_{t^p\oplus L^p_\mathscr Q}
\lesssim  \|u\|_{h^p_\cD}
\]
for all $u\in h^p_\cD \cap L^2$, so by density we have $\|u\|_{h^p_\cD} \eqsim \|\tilde\cQ^\cD_{\psi,\phi}u\|_{t^p\oplus L^p_\mathscr Q}$ for all $u\in h^p_\cD$.

There also exists $\tilde\psi'\in\Psi^\beta(S_{\theta,r}^o)$ and $\tilde\phi'\in\Theta^{\beta}(S_{\theta,r}^o)$ such that $\cS^\cD_\svt{{\tilde\psi},{\tilde\phi}}\cQ^\cD_\svt{{\tilde\psi}',{\tilde\phi}'}=I$ on $h^p_\cD \cap L^2$, so by density we have $\tilde\cS^\cD_\svt{{\tilde\psi},{\tilde\phi}}\tilde\cQ^\cD_\svt{{\tilde\psi}',{\tilde\phi}'}=I$ on $h^p_\cD$. It then follows from Lemma~\ref{Lem: ThetaBds} that $h^p_\cD= \tilde\cS^\cD_{\tilde\psi,\tilde\phi}(t^p\oplus L^p_{\mathscr Q})$.

Now suppose that $u\in h^p_\cD$, in which case $u=\tilde\cS^\cD_\svt{\tilde\psi,\tilde\phi}\tilde\cQ^\cD_\svt{{\tilde\psi}',{\tilde\phi}'}u$ and there exists $V$ in $t^p \oplus L^p_{Q}$ such that $u=\tilde\cS^\cD_{\tilde\psi,\tilde\phi}V$ and $\|V\|_{t^p\oplus L^p_{Q}} \leq 2\inf_{u=\tilde\cS^\cD_{\psi,\phi}U} \|U\|_{t^p\oplus L^p_\mathscr Q}$. Lemma~\ref{Lem: ThetaBds} then shows that
\[
\inf_{u=\tilde\cS^\cD_{\psi,\phi}U} \|U\|_{t^p\oplus L^p_\mathscr Q} \leq \|\tilde\cQ^\cD_\svt{{\tilde\psi}',{\tilde\phi}'} u\|_{t^p\oplus L^p_{\mathscr Q}}
\lesssim \|u\|_{h^p_\cD}
= \|\tilde\cS^\cD_{\tilde\psi,\tilde\phi}V\|_{h^p_\cD}
\leq 2\inf_{u=\tilde\cS^\cD_{\psi,\phi}U} \|U\|_{t^p\oplus L^p_\mathscr Q},
\]
which completes the proof of (1). The proof of (2) is similar.
\end{proof}

All of the equivalent norms on $h^p_\cD$ are denoted by $\|\cdot\|_{h^p_\cD}$. As an example, recall the Hodge--Dirac operator $D$ and the Hodge--Laplacian $\Delta=D^2$ from Example~\ref{Eg: HodgeDirac}. If $\beta>\kappa/2$ and $a>\lambda^2/4$, then by recalling the $\Phi$-class functions listed after Definition~\ref{DefPhiClass}, we have
\begin{align*}
    \|u\|_{h^p_D}&\eqsim \|{tDe^{-t\sqrt{\Delta+aI}}u}\|_{t^p}+\|{e^{-\sqrt{\Delta+aI}}u}\|_{L^p_{\cQ}}\\
    &\eqsim \|t^2\Delta e^{-t^2\Delta}u\|_{t^p}+\|e^{-\Delta}u\|_{L^p_{\cQ}}\\
    &\eqsim \|tD(t^2\Delta+aI)^{-\beta}u\|_{t^p}+\|{(\Delta+aI)^{-\beta}u}\|_{L^p_{\cQ}}
\end{align*}
for all $u\in h^p_\cD$ and $p\in[1,\infty]$, where the operators are initially defined on $L^2$ and extended to $h^p_\cD$.

Finally, the duality results for the local tent spaces $t^p$ and the spaces $L^p_{\mathscr Q}$ allow us to derive a duality result for the local Hardy spaces.

\begin{Thm}\label{Thm: hp.duality}
If $p\in(1,\infty)$ and $1/p+1/p'=1$, then the mapping
\[
v\mapsto \langle u,v\rangle_{h^2_\cD} = \langle\tilde\cQ^\cD_{\eta,\varphi} u,\tilde\cQ^{\cD^*}_{\eta^*,\varphi^*} v\rangle_{L^2_\bullet\oplus L^2}
\]
for all $u\in h^p_\cD$ and $v\in h^{p'}_{\cD^*}$, is an isomorphism from $h^{p'}_{\cD^*}$ onto the dual $(h^p_\cD)^*$.
\end{Thm}
\begin{proof}
Using Theorems~\ref{Thm: tp2propertiesDuality} and \ref{Thm: LpQpropertiesDuality}, we obtain
\[
|\langle\tilde\cQ^\cD_{\eta,\varphi}u,\tilde\cQ^{\cD^*}_{\eta^*,\varphi^*}v\rangle_{L^2_\bullet\oplus L^2}|
\leq \|u\|_{h^p_\cD} \|v\|_{h^{p'}_{\cD^*}}
\]
for all $u\in h^p_\cD$ and $v\in h^{p'}_{\cD^*}$, since $\tilde\cQ^\cD_{\eta,\varphi}u\in t^p\oplus L^p_{\mathscr Q}$ by Definition~\ref{Def: local.Hardy}, and $\tilde\cQ^{\cD^*}_{\eta^*,\varphi^*}v$ is in $t^{p'}\oplus L^{p'}_{\mathscr Q}$ by Proposition~\ref{Prop: NormeEquivChangingQandS}.

Now suppose that $T\in (h^p_\cD)^*$ and define $\tilde T\in(t^p\oplus L^p_{\mathscr Q})^*$ by
\[
\tilde T(V) = T(\tilde\cS^\cD_{\eta,\varphi}V)
\]
for all $V\in t^p\oplus L^p_{\mathscr Q}$. It follows from Theorems~\ref{Thm: tp2propertiesDuality} and \ref{Thm: LpQpropertiesDuality} that there exists $U_T$ in $t^{p'}\oplus L^{p'}_{\mathscr Q}$ such that $\|U_T\|_{t^{p'}\oplus L^{p'}_{\mathscr{Q}}} \eqsim \|\tilde{T}\|$ and $\tilde T(V) = \langle V,U_T \rangle_{L^2_\bullet\oplus L^2}$ for all $V\in t^p\oplus L^p_{\mathscr Q}$. The reproducing formula $\tilde\cS^\cD_{\eta,\varphi}\tilde\cQ^\cD_{\eta,\varphi}=I$, which is valid on $h^p_\cD$  by Proposition~\ref{Prop: QandSext}, then implies that
\[
Tu= T(\tilde\cS^\cD_{\eta,\varphi}\tilde\cQ^\cD_{\eta,\varphi}u)
=\tilde T (\tilde\cQ^\cD_{\eta,\varphi}\tilde\cS^\cD_{\eta,\varphi}\tilde\cQ^\cD_{\eta,\varphi}u) = \langle \tilde\cQ^\cD_{\eta,\varphi}\tilde\cS^\cD_{\eta,\varphi}\tilde\cQ^\cD_{\eta,\varphi}u,U_T \rangle_{L^2_\bullet\oplus L^2}
\]
for all $u\in h^p_\cD$. If $U_T\in (t^{p'}\cap t^2) \oplus (L^{p'}_{\mathscr Q}\cap L^2_{\mathscr Q})$, then since $(\cQ^\cD_{\eta,\varphi})^*=\cS^{\cD^*}_{\eta^*,\varphi^*}$ on $t^2\oplus L^2_{\mathscr Q}$ and $(\cS^\cD_{\eta,\varphi})^*=\cQ^{\cD^*}_{\eta^*,\varphi^*}$ on $L^2$, we obtain
\[
Tu=\langle \cQ^\cD_{\eta,\varphi} u,\cQ^{\cD^*}_{\eta^*,\varphi^*}(\cS^{\cD^*}_{\eta^*,\varphi^*} U_T) \rangle_{L^2_\bullet\oplus L^2}
\]
for all $u \in h^p_\cD\cap L^2$. If $U_T\in t^{p'}\oplus L^{p'}_{\mathscr Q}$, then the density properties in Propositions~\ref{Prop: tpt2denseintp} and \ref{Prop: LpQL2denseinLp} imply that the above result extends to
\[
Tu=\langle \tilde\cQ^\cD_{\eta,\varphi} u, \tilde\cQ^{\cD^*}_{\eta^*,\varphi^*}(\tilde\cS^{\cD^*}_{\eta^*,\varphi^*} U_T) \rangle_{L^2_\bullet\oplus L^2}
\]
for all $u\in h^p_\cD$, and by Proposition~\ref{Prop: NormeEquivChangingQandS}, we have $\tilde\cS^{\cD^*}_{\eta^*,\varphi^*} U_T \in h^{p'}_{\cD^*}$ with 
\[
\|\tilde\cS^{\cD^*}_{\eta^*,\varphi^*} U_T\|_{h^{p'}_{\cD^*}} \lesssim \|U_T\|_{t^{p'}\oplus L^{p'}_{\mathscr{Q}}} \lesssim \|\tilde{T}\| \lesssim \|T\|,
\]
where the last inequality follows from Proposition~\ref{prop: normequiv}.

Finally, to prove injectivity, let $v\in h^{p'}_{\cD^*}$ and suppose that $\langle u,v\rangle_{h^2_{\cD}}=0$ for all $u\in h^p_{\cD}$. It suffices to show that $v=0$. Define $\ell(V)=\langle V, \tilde\cQ^{\cD^*}_{\eta^*,\varphi^*} v\rangle_{L^2_\bullet \oplus L^2}$ for all $V\in t^p\oplus L^p_{\mathscr{Q}}$, in which case $\ell \in (t^p\oplus L^p_{\mathscr{Q}})^*$ with $\|\ell\| \eqsim \|\tilde{\cQ}^{\cD^*}_{\eta^*,\varphi^*} v\|_{t^{p'}\oplus L^{p'}_{\mathscr{Q}}} \eqsim \|v\|_{h^{p'}_{\cD^*}}$, since $\tilde{\cQ}^{\cD^*}_{\eta^*,\varphi^*} v \in t^{p'}\oplus L^{p'}_{\mathscr{Q}}$. Using the reproducing formula and duality, we obtain
\[
\ell(V) =\langle V, \tilde{\cQ}^{\cD^*}_{\eta^*,\varphi^*}\tilde{\cS}^{\cD^*}_{\eta^*,\varphi^*}\tilde{\cQ}^{\cD^*}_{\eta^*,\varphi^*}v\rangle_{L^2_\bullet \oplus L^2}
=\langle \tilde{\cS}^{\cD}_{\eta,\varphi} V, v\rangle_{h^2_\cD}
=0
\]
for all $V\in t^{p}\oplus L^p_{\mathscr{Q}}$, since Proposition~\ref{prop: normequiv} implies that $\tilde{\cS}^{\cD}_{\eta,\varphi} V \in h^p_{\cD}$. Altogether, we have $\|v\|_{h^{p'}_{\cD^*}} \eqsim \|\ell\| = 0$, hence $v=0$ as required.
\end{proof}

\subsection{Molecular Characterisation} \label{SubSection2ndmainresult}
We prove a molecular characterisation of $h^1_\cD$. The Hardy space $H^1_D$ from \cite{AMcR} is characterised in terms of $H^1_D$-molecules, which are differential forms $a$ that satisfy $a=D^Nb$ for some differential form $b$ and $N\in\N$. In contrast to atoms, molecules are not assumed to be compactly-supported. Instead, the $L^2$-norms of $a$ and $b$ are concentrated on some ball. The condition $a=D^Nb$ is the substitute for the moment condition required of classical atoms. The molecular characterisation of $h^1_\cD$ proved here involves two different types of molecules, reflecting the atomic characterisation of $h^1(\R^n)$ mentioned in the introduction. The first kind are concentrated on balls of radius less than 1 and are of the type used to characterise $H^1_D$, whilst the second kind are concentrated on balls of radius larger than 1 and are not required to satisfy a moment condition.

We use the following notation to specify the $L^2$-norm distribution of molecules.

\begin{Not1}
Given a ball $B$ in $M$ of radius $r(B)>0$, let $\ca_k(B)$ denote the characteristic function defined by
\[
\ca_k(B)=
\begin{cases}
\ca_{B}&{\rm if}\quad k=0;\\
\ca_{2^{k}B\setminus 2^{k-1}B}& {\rm if}\quad k=1,2,\dots.\\
\end{cases}
\]
\end{Not1}

\begin{Def}\label{localmoleculedef}
Given $N\in\N$ and $q\geq0$, an $h_\cD^1$-$\textit{molecule of type }(N,q)$ is a measurable differential form $a$ associated with a ball $B$ in $M$ of radius $r(B)>0$ such that the following hold:
\begin{enumerate}\setlength\itemsep{3pt}
\item The bound $\|\ca_k(B)a\|_2 \leq \exp({-q 2^{k-1}r(B)}) 2^{-k} \mu(2^kB)^{-{1}/{2}}$ for all $k\geq0$;
\item If $r(B)<1$, then there exists a differential form $b$ with ${a=\cD^N b}$ and the bound $\|\ca_k(B)b\|_2 \leq r(B)^N \exp({-q 2^{k-1}r(B)}) 2^{-k}\mu(2^kB)^{-{1}/{2}}$ for all $k\geq~0$.
\end{enumerate}
\end{Def}

If $a$ and $b$ are as in Definition~\ref{localmoleculedef}, then $a$ and $b$ are in $L^2=h^2_\cD\subseteq h^0_\cD$ with
\begin{equation}\label{N-molFullL2BddRemLocal(a)}
\|a\|_2 \leq \sum_{k=0}^\infty \|\ca_k(B) a\|_2
\leq 2 e^{-qr(B)/2}\mu(B)^{-\frac{1}{2}}
\end{equation}
and
\begin{equation}\label{N-molFullL2BddRemLocal(b)}
\|b\|_2 \leq 2r(B)^N e^{-qr(B)/2} \mu(B)^{-\frac{1}{2}}.
\end{equation}
Condition (2) is obviated in Definition~\ref{localmoleculedef} when $r(B)\geq1$, so we set $N=0$ in that case. We will see that $q$ is related to the exponential growth parameter $\lambda$ in \eqref{ELD}, and that we can set $q=0$ when $M$ is doubling, since then $\lambda=0$. Given~$\delta>1$, note that the results in this section also hold for $h_\cD^1$-molecules defined by replacing $2^k$ and $2^{-k}$ with $\delta^k$ and $\delta^{-k}$ in Definition~\ref{localmoleculedef}.

\begin{Def}
Given $N\in\N$ and $q\geq0$, define $h^1_{\cD,\text{mol}(N,q)}$ to be the space of all $u$ in $h^0_\cD$ for which there exist a sequence $(\lambda_j)_j$ in $\ell^1$ and a sequence $(a_j)_j$ of $h^1_\cD$-molecules of $\text{type }(N,q)$ such that $\sum_j \lambda_j a_j$ converges to $u$ in $h^0_\cD$. Moreover, define
\[
\|u\|_{h^1_{\cD,\text{mol}(N,q)}} = \inf \{\|(\lambda_j)_j\|_{\ell^1} : f=\textstyle{\sum_j} \lambda_j a_j\}
\]
for all $u\in h^1_{\cD,\text{mol}(N,q)}$.
\end{Def}

The following is the molecular characterisation of $h^1_{\cD}$. Theorem~\ref{Thm: Intro.Molecules} follows from this result in the case of the Hodge--Dirac operator by Example~\ref{Eg: HodgeDirac}.

\begin{Thm}\label{mainmoleculeresult}
Let $\kappa,\lambda\geq0$ and suppose that $M$ is a complete Riemannian manifold satisfying \eqref{ELD}. Suppose that $\cD$ is a closed densely-defined operator on $L^2(\wedge T^*M)$ satisfying (H1-3) from Section~\ref{SectionODE}. If $N\in\N$, $N>\kappa/2$ and $q\geq\lambda$, then
$h^1_\cD=h^1_{\cD,\text{mol}(N,q)}$.
\end{Thm}
\begin{proof}
Fix $N\in\N$ and $q\geq0$. Let $\tilde\psi$ and $\tilde\phi$ be the functions from Lemmas~\ref{localmolecularlemma1} and \ref{localmolecularlemma1.5} below. Suppose that $u\in h^1_\cD \subseteq h^0_\cD$. Proposition \ref{Prop: NormeEquivChangingQandS} then implies that there exists $(V,v)\in t^1\oplus L^1_{\mathscr Q}$ such that $u=\tilde\cS^\cD_{\tilde\psi,\tilde\phi}(V,v)$ and $\|(V,v)\|_{t^1\oplus L^1_{\mathscr Q}} \lesssim \|u\|_{h^1_\cD}$. Also, by Theorems~\ref{maintentatomic} and \ref{mainLQatomic}, there exist a sequence $(A_j)_j$ of $t^1$-atoms, a sequence $(a_j)_j$ of $L^1_{\mathscr Q}$-atoms and two sequences $(\lambda_j)_j$ and $(\tilde\lambda_j)_j$ in $\ell^1$ such that
\[
V=\sum_j\lambda_j A_j \quad\text{and}\quad v=\sum_j \tilde\lambda_j a_j,
\]
where these sums converge in $t^1$ and $L^1_{\mathscr Q}$, respectively. Moreover, we can assume that $\|(\lambda_j)_j\|_{\ell^1}\lesssim \|V\|_{t^1}$, $\|(\tilde\lambda_j)_j\|_{\ell^1} \lesssim \|v\|_{L^1_{\mathscr Q}}$ and, by Remark~ \ref{RemSmallL1QAtoms}, that each ${L^1_{\mathscr Q}\textrm{-atom}}$ $a_j$ is associated with a ball of radius equal to 1. Therefore, we have
\[
u=\sum_j \left(\lambda_j\int_0^1 \tilde\psi_t(\cD) A_j\frac{\d t}{t} + \tilde\lambda_j \tilde\phi(\cD)a_j\right),
\]
where the sum converges in $h^1_\cD$, and hence also in $h^0_\cD$, because Proposition~\ref{Prop: NormeEquivChangingQandS} implies that
\[
\textstyle\|u-\sum_{j=1}^n\cS^\cD_{\tilde\psi,\tilde\phi}(\lambda_j A_j,\tilde\lambda_j a_j) \|_{h^1_\cD}
\lesssim \|V-\sum_{j=1}^n\lambda_j A_j\|_{t^1} + \|v-\sum_{j=1}^n\tilde\lambda_j a_j\|_{L^1_{\mathscr Q}}
\]
for all $n\in\N$. It follows from Lemmas~\ref{localmolecularlemma1} and \ref{localmolecularlemma1.5} that $u\in h^1_{\cD,\text{mol}(N,q)}$, and since $\|(\lambda_j)_j\|_{\ell^1}+\|(\tilde\lambda_j)_j\|_{\ell^1} \lesssim \|u\|_{h^1_\cD}$, we have shown that $h^1_\cD\subseteq h^1_{\cD,\text{mol}(N,q)}$.

We prove the converse in the case $N\in\N$, $N>\kappa/2$ and $q\geq\lambda$. Let $\psi$ and $\phi$ be the functions from Lemmas~\ref{localmolecularlemma2} and \ref{localmolecularlemma2.5} below. Suppose that $u\in h^1_{\cD,\text{mol}(N,q)}\subseteq h^0_\cD$. There exist a sequence $(a_j)_j$ of $h^1_{\cD}$-molecules of $\text{type }(N,q)$ and a sequence $(\lambda_j)_j$ in $\ell^1$ such that
$\sum_j\lambda_ja_j$ converges to $u$ in $h^0_\cD$. It follows from Proposition~\ref{Prop: NormeEquivChangingQandS} and Lemmas~\ref{localmolecularlemma2} and \ref{localmolecularlemma2.5} that $\sum_{j=1}^n \lambda_j a_j$ is in $h^1_\cD$ with $\|\sum_{j=1}^n \lambda_j a_j\|_{h^1_\cD}\lesssim \sum_{j=1}^n |\lambda_j|$ for all $n\in\N$. Therefore, there exists $v$ in $h^1_\cD$ such that $\sum_j\lambda_ja_j$ converges to $v$ in $h^1_\cD$, and hence also in $h^0_\cD$. This implies that $u=v\in h^1_\cD$, so by Proposition~\ref{Prop: NormeEquivChangingQandS} we have
\[
\textstyle \|\tilde\cQ^\cD_{\psi,\phi} u -\sum_{j=1}^n \lambda_j \cQ^\cD_{\psi,\phi} a_j \|_{t^1\oplus L^1_{\mathscr Q}}
\lesssim \|u-\sum_{j=1}^n\lambda_j a_j\|_{h^1_\cD}
\]
for all $n\in\N$. It follows from Lemmas \ref{localmolecularlemma2} and \ref{localmolecularlemma2.5} that
\[
\textstyle
\|u\|_{h^1_\cD} \eqsim \|\tilde\cQ^\cD_{\psi,\phi}u\|_{t^1\oplus L^1_{\mathscr Q}}
\leq \sum_j\lambda_j \|(\psi_t(\cD)a_j,\phi(\cD)a_j)\|_{t^1\oplus L^1_{\mathscr Q}}
\lesssim \|(\lambda_j)_j\|_{\ell^1},
\]
which shows that $h^1_{\cD,\text{mol}(N,q)} \subseteq h^1_\cD$.
\end{proof}

We now prove four lemmas to construct the functions $\tilde\psi,\tilde\phi,\psi$ and $\phi$ that were used to prove Theorem~\ref{mainmoleculeresult}.

\begin{Lem} \label{localmolecularlemma1}
Let $\theta\in(\omega,\tfrac{\pi}{2})$, $r>R$ and $\beta>\kappa/2$ such that $r/C_\cD C_{\theta,r}>\lambda/2$. For each $N\in\N$ and $q\geq0$, there exist $c>0$ and $\tilde\psi\in\Psi_\beta(S_{\theta,r}^o)$ such that
\[c\int_0^1\tilde\psi_t(\cD)A_t\frac{\d t}{t}\]
is an $h^1_\cD$-molecule of $\text{type }(N,q)$ for all $A$ that are $t^1$-atoms.
\end{Lem}
\begin{proof} Let $A$ be a $t^1$-atom. There exists a ball $B$ in $M$ with radius $r(B)\leq 2$ such that $A$ is supported in $T^1(B)$ and $\|A\|_{L^2_{\bullet}}\leq \mu(B)^{-1/2}$. Choose $\tilde r$ so that $\tilde r\geq r$ and $\tilde r/C_\cD C_{\theta,\tilde r}>\lambda+q$. Also, choose $\tilde\psi$ in $\Psi_{\beta+N+1}(S_{\theta,\tilde r}^o)$, in which case $\psi\in\Psi_\beta(S_{\theta,r}^o)$. Next, define $\tilde{\tilde\psi}(z)= z^{-N}\tilde\psi(z)$, in which case $\tilde{\tilde\psi}\in\Psi_{\beta+1}(S_{\theta,r}^o)$ and
\[\int_0^1\tilde\psi_t(\cD)A_t\frac{\d t}{t} = \cD^N\left(\int_0^1t^N\tilde{\tilde\psi}_t(\cD)A_t\frac{\d t}{t}\right).\]
It remains to prove that there exists $c>0$, which does not depend on $A$, such~that
\begin{equation}\label{eq: a1}
\|\ca_k(B)\left(\int_0^1\tilde\psi_t(\cD)A_t\frac{\d t}{t}\right)\|_2
\leq ce^{-q2^{k-1}r(B)}2^{-k}\mu(2^kB)^{-\frac{1}{2}}
\end{equation}
for all $k\geq0$, and that if $r(B)<1$, then
\begin{equation}\label{eq: a2}
\|\ca_k(B)\left(\int_0^1t^N\tilde{\tilde\psi}_t(\cD)A_t\frac{\d t}{t}\right)\|_2
\leq cr(B)^N e^{-q2^{k-1}r(B)} 2^{-k}\mu(2^kB)^{-\frac{1}{2}}
\end{equation}
for all $k\geq0$.

Now, since $\beta>\kappa/2$ and $\tilde r/C_\cD C_{\theta,\tilde r}>\lambda+q$, Lemma~\ref{ODexpHvNP} implies the following estimates:
\begin{align}
\label{psiest} \|\ca_E\tilde\psi_t(\cD)\ca_F\|
&\lesssim \langle t/\rho(E,F)\rangle^{\frac{\kappa}{2}+1} e^{-(\lambda+q)\rho(E,F)};\\
\label{tpsiest} \|\ca_E\tilde{\tilde\psi}_t(\cD)\ca_F\|
&\lesssim \langle t/\rho(E,F)\rangle^{\frac{\kappa}{2}+1} e^{-(\lambda+q)\rho(E,F)}
\end{align}
for all $t\in(0,1]$ and closed subsets $E$ and $F$ of $M$.

We now prove \eqref{eq: a1}. If $k=0$ or $k=1$, then by \eqref{eq: h2equiv2} and  $r(B)\leq2$, we have
\[\|\ca_k(B)\left(\int_0^1\tilde\psi_t(\cD)A_t\frac{\d t}{t}\right)\|_2 \lesssim \|A\|_{L^2_\bullet} \lesssim
\begin{cases}
e^{-q2^{-1}r(B)}\mu(B)^{-\frac{1}{2}}  &{\rm if}\quad k=0; \\
e^{-qr(B)}2^{-1}\mu(2B)^{-\frac{1}{2}} &{\rm if}\quad k=1.
\end{cases}
\]
If $k\geq 2$, then \[\rho(2^kB\backslash2^{k-1}B,B)=(2^{k-1}-1)r(B)\gtrsim 2^kr(B)\]
and $\mu(2^kB) \leq 2^{k\kappa}e^{\lambda(2^k-1)r(B)}\mu(B)$, so by \eqref{psiest}, and since $r(B)\leq2$, we have
\begin{align*}
\|\ca_k(B)\bigg(\int_0^1\tilde\psi_t(\cD)&A_t\frac{\d t}{t}\bigg)\|_2
\leq \int_0^{r(B)} \|\ca_k(B)\tilde\psi_t(\cD)\ca_B\| \|A_t\|_2 \frac{\d t}{t} \\
&\lesssim \left(\int_0^{r(B)} \left(\frac{t}{2^k {r(B)}}\right)^{2(\frac{\kappa}{2}+1)} \frac{\d t}{t}\right)^{\frac{1}{2}} e^{-(\lambda+q)(2^{k-1}-1){r(B)}} \|A\|_{L^2_\bullet} \\
&\leq 2^{-k(\frac{\kappa}{2}+1-\frac{\kappa}{2})} e^{-q(2^{k-1}-1){r(B)}} e^{\lambda(-2^{k-1}+1+2^{k-1}-\frac{1}{2}){r(B)}} \mu(2^kB)^{-\frac{1}{2}} \\
&\lesssim e^{-q2^{k-1}{r(B)}} 2^{-k} \mu(2^kB)^{-\frac{1}{2}}.
\end{align*}

We prove \eqref{eq: a2} similarly. If $k=0$ or $k=1$, then we have
\begin{align*}
\|\ca_k(B)\left(\int_0^1t^N\tilde{\tilde\psi}_t(\cD)A_t\frac{\d t}{t}\right)\|_2
&\lesssim {r(B)}^N \|A\|_{L^2_\bullet} \\
&\lesssim {r(B)}^N
\begin{cases}
e^{-q2^{-1}{r(B)}}\mu(B)^{-\frac{1}{2}} &{\rm if}\quad k=0; \\
e^{-qr(B)}2^{-1}\mu(2B)^{-\frac{1}{2}}  &{\rm if}\quad k=1.
\end{cases}
\end{align*}

If $k\geq 2$, then by \eqref{tpsiest} we have
\begin{align*}
\|\ca_k(B)\left(\int_0^1t^N\tilde{\tilde\psi}_t(\cD)A_t\frac{\d t}{t}\right)\|_2
&\leq {r(B)}^N \int_0^{r(B)} \|\ca_k(B)\tilde{\tilde\psi}_t(\cD)\ca_B\| \|A_t\|_2 \frac{\d t}{t} \\
&\lesssim {r(B)}^N e^{-q2^{k-1}{r(B)}} 2^{-k} \mu(2^kB)^{-\frac{1}{2}},
\end{align*}
which completes the proof.
\end{proof}

\begin{Lem} \label{localmolecularlemma1.5}
Let $\theta\in(\omega,\tfrac{\pi}{2})$, $r>R$ and $\beta>\kappa/2$ such that $r/C_\cD C_{\theta,r}>\lambda/2$. For each $N\in\N$ and $q\geq0$, there exist $c>0$ and $\tilde\phi\in\Phi(S_{\theta,r}^o)$ such that $c\tilde\phi(\cD)a$ is an $h^1_\cD$-molecule of $\text{type }(N,q)$ for all $a$ that are $L^1_{\mathscr{Q}}$-atoms supported on balls $B$ of radius ${r(B)}=1$ with $\|a\|_2\leq \mu(B)^{-1/2}$.
\end{Lem}
\begin{proof} Let $a$ and $B$ be as stated in the lemma. Choose $\tilde r$ so that $\tilde r\geq r$ and $\tilde r/C_\cD C_{\theta,\tilde r}>\lambda+q$. Also, choose $\tilde\phi$ in $\Phi(S_{\theta,\tilde r}^o)$, in which case $\tilde \phi \in \Phi(S_{\theta,r}^o)$. Now, since ${r(B)}=1$, it only remains to prove that there exists $c>0$, which does not depend on $a$, such that
\[\|\ca_k(B)\tilde\phi(\cD)a\|_2 \leq c e^{-q2^{k-1}{r(B)}} 2^{-k}\mu(2^kB)^{-\frac{1}{2}}\]
for all $k\geq0$. To do this, choose $\delta$ in $(0,\tilde r/C_\cD C_{\theta,\tilde r}-(\lambda+q))$. Lemma~\ref{ODexpHvNP} then implies that
\begin{equation}\label{phiest}
\|\ca_E\tilde\phi(\cD)\ca_F\| \lesssim e^{-(\lambda+q+\delta)\rho(E,F)} \lesssim \langle 1/\rho(E,F)\rangle^{\frac{\kappa}{2}+1} e^{-(\lambda+q)\rho(E,F)}
\end{equation}
for all closed subsets $E$ and $F$ of $M$.

If $k=0$ or $k=1$, then by \eqref{eq: h2equiv2}, and since $r(B)=1$, we have
\[\|\ca_k(B)\tilde\phi(\cD)a\|_2
\lesssim \|a\|_2
\leq \mu(B)^{-\frac{1}{2}} \lesssim
\begin{cases}
e^{-q2^{-1}{r(B)}}\mu(B)^{-\frac{1}{2}} &{\rm if}\quad k=0; \\
e^{-qr(B)}2^{-1}\mu(2B)^{-\frac{1}{2}}  &{\rm if}\quad k=1.
\end{cases}
\]
If $k\geq 2$, then using \eqref{phiest} and proceeding as in Lemma~\ref{localmolecularlemma1}, we obtain
\begin{align*}
\|\ca_k(B)\tilde\phi(\cD)a\|_2 &\leq \|\ca_k(B)\tilde\phi(\cD)\ca_B\|\|a\|_2
\lesssim e^{-q2^{k-1}{r(B)}} 2^{-k} \mu(2^kB)^{-\frac{1}{2}},
\end{align*}
which completes the proof.
\end{proof}

\begin{Lem} \label{localmolecularlemma2}
Let $\theta\in(\omega,\tfrac{\pi}{2})$, $r>R$ and $\beta>\kappa/2$ such that $r/C_\cD C_{\theta,r}>\lambda/2$. For each $N\in\N$, $N>\kappa/2$ and $q\geq\lambda$, there exist $c>0$ and $\psi\in\Psi^\beta(S_{\theta,r}^o)$ such that $\|\psi_t(\cD)a\|_{t^1} \leq c$
for all $a$ that are $h^1_\cD$-molecules of $\text{type }(N,q)$.
\end{Lem}
\begin{proof}
Let $a$ be an $h^1_\cD$-molecule of $\text{type }(N,q)$. There exists a ball $B$ in $M$ of radius ${r(B)}>0$ such that the requirements of Definition~\ref{localmoleculedef} are satisfied. Let $C^1_0(B)=C^1(B)$ be the truncated Carleson box over $B$ introduced in Section~\ref{SectionLocalTent}, and let $C^1_k(B)=C^1(2^kB)\setminus C^1(2^{k-1}B)$ for each $k\geq1$. As depicted in Figure~\ref{fig1}, divide each $C^1_k(B)$ with the following characteristic functions:
\begin{align*}
&\eta_k  = \ca_{C^1_k(B)} \ca_{M\times(0,{r(B)}]};\\
&\eta_k' = \ca_{C^1_k(B)} \ca_{M\times({r(B)},2^{k-1}{r(B)}]};\\
&\eta_k''= \ca_{C^1_k(B)} \ca_{M\times(2^{k-1}{r(B)},2^kr(B)]},
\end{align*}
so we have $\ca_{C_k^1(B)}=\eta_k + \eta'_k + \eta''_k$ and
$\sum_{k} \ca_{C_k^1(B)} =~ \ca_{M\times(0,1]}$.

\begin{figure}[h]
\centering

\begin{pspicture}(-6.5,-0.5)(6,5)
\psset{unit=1cm}

\newgray{gray2}{0.8}
\pspolygon[linestyle=none,fillstyle=solid,fillcolor=gray2](-4,0)(-4,4)(4,4)(4,0)(2,0)(2,2)(-2,2)(-2,0)

\psline(-5.5,0)(5.25,0)
\psline(-5.25,-0.25)(-5.25,4.6)
\psline(-5.25,4.7)(-5.25,5)
\psline(-5.35,1)(-5.15,1)
\psline(-5.35,2)(-5.15,2)
\psline(-5.35,4)(-5.15,4)
\psline(-5.35,5)(-5.15,5)
\psline(-5.35,4.7)(-5.15,4.7)
\psline(-5.35,4.6)(-5.15,4.6)
\psline(0,-0.1)(0,0.1)

\psline(-1,0)(-1,1)\psline(-1,1)(1,1)\psline(1,0)(1,1)
\psline(-2,0)(-2,2)\psline(-2,2)(2,2)\psline(2,0)(2,2)
\psline(-4,0)(-4,4)\psline(-4,4)(4,4)\psline(4,0)(4,4)

\psline(-4.75,1)(-1,1)
\psline(-4,2)(-2,2)
\psline(-4.75,4)(-4,4)
\psline(4.75,1)(1,1)
\psline(4,2)(2,2)
\psline(4,4)(4.75,4)

\uput[u](0,0.1){$\eta_0$}
\uput[u](-1.5,0.1){$\eta_1$}
\uput[u](1.5,0.1){$\eta_1$}
\uput[u](3,0.1){$\eta_2$}
\uput[u](-3,0.1){$\eta_2$}
\uput[u](1.5,0.1){$\eta_1$}
\uput[u](4.5,0.1){$\eta_3$}
\uput[u](-4.5,0.1){$\eta_3$}
\uput[u](3,1.1){$\eta_2'$}
\uput[u](-3,1.1){$\eta_2'$}
\uput[u](-4.5,2.1){$\eta_3'$}
\uput[u](4.5,2.1){$\eta_3'$}
\uput[u](0,1.1){$\eta_1''$}
\uput[u](0,2.5){$\eta_2''$}
\uput[u](0,4){$\eta_3''$}
\uput[d](0,0){0}
\uput[d](1,0){$r(B)$}
\uput[d](2,0){$2r(B)$}
\uput[d](4,0){$4r(B)$}
\uput[r](5.25,0){$M$}
\uput[l](-5.25,1){$r(B)$}
\uput[l](-5.25,2){$2r(B)$}
\uput[l](-5.25,4){$4r(B)$}
\uput[l](-5.25,5){1}
\end{pspicture}
\caption{The division of $C_2^1(B)$ used in Lemma~\ref{localmolecularlemma2} for a ball $B$ in $M$ of radius $r(B)<1/4$.}
\label{fig1}
\end{figure}
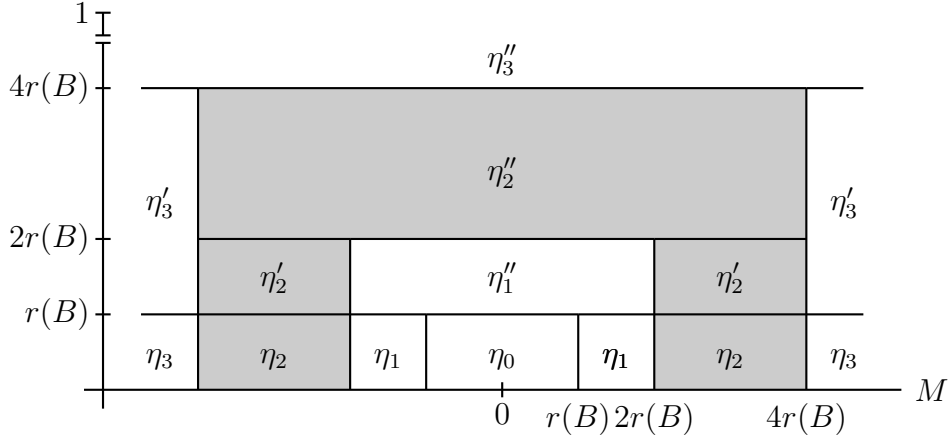

Suppose that there exist $\psi\in\Psi^\beta (S_{\theta,r}^o)$ and $c,\delta>0$, all of which do not depend on $a$, such that the following hold for all $k\geq0$:
\begin{subequations}\label{A-C}
\begin{align}
\label{A}
\|\eta_k\psi_t(\cD)a\|_{L^2_\bullet}  &\leq c2^{-\delta k} \mu(2^kB)^{-\frac{1}{2}}; \\
\label{B}
\|\eta_k'\psi_t(\cD)a\|_{L^2_\bullet} &\leq c2^{-\delta k} \mu(2^{k} B)^{-\frac{1}{2}}; \\
\label{C}
\|\eta_k''\psi_t(\cD)a\|_{L^2_\bullet}&\leq c2^{-\delta k} \mu(2^kB)^{-\frac{1}{2}}.
\end{align}
\end{subequations}
In that case, each $(2^{\delta k}/c)\ca_{C^1_k(B)}\psi_t(\cD)a$ is a $t^1$-Carleson atom, and since
\[\psi_t(\cD)a=\sum_{k=0}^\infty \ca_{C^1_k(B)}\psi_t(\cD)a\]
almost everywhere in $M\times (0,1]$, Proposition~\ref{Prop: CarlesonBoxTentAtoms} implies that $\psi_t(\cD)a$ is in $t^1$ with $\|\psi_t(\cD)a\|_{t^1} \lesssim c\sum_{k=0}^\infty 2^{-\delta k} \lesssim 1$. Therefore, it suffices to prove \eqref{A-C}.

To prove \eqref{A-C}, choose $\tilde r$ so that $\tilde r\geq r$ and $\tilde r/C_\cD C_{\theta,\tilde r}>\lambda$. Also, choose $\delta$ in $(0,\beta-\kappa/2)$ and choose $\psi$ in $\Psi_{\beta}^{\beta+N}(S_{\theta,\tilde r}^o)$, in which case $\psi\in \Psi^{\beta}(S_{\theta,r}^o)$. Then, since $\beta>\kappa/2$, Lemma~\ref{ODexpHvNP} implies that
\begin{equation}\label{psiODexp}
\|\ca_E\psi_t(\cD)\ca_F\| \lesssim \langle{t}/{\rho(E,F)}\rangle^{\frac{\kappa}{2}+\delta}e^{-\lambda\rho(E,F)} \leq \langle{t}/{\rho(E,F)}\rangle^\delta
\end{equation}
for all closed subsets $E$ and $F$ of $M$.

We now prove \eqref{A}. If $k=0$, then by \eqref{eq: h2equiv1} and \eqref{N-molFullL2BddRemLocal(a)} we have
\[\|\eta_0\psi_t(\cD)a\|_{L^2_\bullet}\leq \|\psi_t(\cD)a\|_{L_\bullet^2} \lesssim \|a\|_2 \lesssim \mu(B)^{-\frac{1}{2}}.\]
Now consider $k\geq 1$. For each $l\in\N$, define $I_l$ by
\[\|\eta_k\psi_t(\cD)a\|_{L^2_\bullet}^2 \leq\sum_{l=0}^\infty \int_{0}^{\langle {r(B)}\rangle} \|\ca_k(B)\psi_t(\cD)\ca_l(B)a\|_2^2\ \frac{\d t}{t} = \sum_{l=0}^\infty I_l.\]
If $0\leq l\leq k-2$, then
\[\rho(2^kB\backslash 2^{k-1}B,2^lB\backslash 2^{l-1}B) =(2^{k-1}-2^l){r(B)}\gtrsim 2^kr(B)\]
and
$\mu(2^kB) \leq 2^{(k-l)\kappa}e^{\lambda(2^{k-l}-1)2^lr(B)}\mu(2^lB),$
so by \eqref{psiODexp} we have
\begin{align*}
I_l &\lesssim \int_0^{r(B)} \left(\frac{t}{2^kr(B)}\right)^{2(\frac{\kappa}{2}+\delta)} \frac{\d t}{t} e^{-2\lambda(2^{k-1}-2^l){r(B)}} e^{-q 2^lr(B)}2^{-2l}\mu(2^lB)^{-1}\\
&\lesssim 2^{-2l(\frac{\kappa}{2}+1)}2^{-2k(\frac{\kappa}{2}+\delta-\frac{\kappa}{2})} e^{\lambda(-2^{k}+2^{l+1}+2^k-2^l-2^l){r(B)}} \mu(2^kB)^{-1}\\
&\lesssim 2^{-2 l}2^{-2\delta k}\mu(2^kB)^{-1}.
\end{align*}
If $k-1\leq l\leq k+1$, then $\mu(2^kB)\lesssim e^{\lambda 2^lr}\mu(2^lB)$, so we have
\[
I_l
\leq \|\psi_t(\cD)\ca_l(B)a\|_{L^2_\bullet}^2
\lesssim e^{-q 2^lr(B)}2^{-2l}\mu(2^lB)^{-1}
\lesssim 2^{-2l}\mu(2^kB)^{-1}.
\]
If $l\geq k+2$, then
\[\rho(2^kB\backslash 2^{k-1}B,2^lB\backslash 2^{l-1}B) =(2^{l-1}-2^k){r(B)}\gtrsim 2^lr(B)\]
and $\mu(2^kB) \leq \mu(2^lB)$, so by \eqref{psiODexp} we have
\begin{align*}
I_l \lesssim \int_0^{{r(B)}} \left(\frac{t}{2^lr(B)}\right)^{2\delta} \frac{\d t}{t} 2^{-2l}\mu(2^lB)^{-1}
\lesssim 2^{-2l}2^{-2\delta k} \mu(2^kB)^{-1}.
\end{align*}
Note that we needed $q\geq\lambda$ when $0\leq l\leq k+1$. This proves \eqref{A}, since now
\[
\|\eta_k\psi_t(\cD)a\|_{L^2_\bullet}^2
\leq \sum_{l=0}^{\infty} I_l
\lesssim \sum_{l=0}^\infty 2^{-2l}2^{-2\delta k} \mu(2^kB)^{-1}
\lesssim 2^{-2\delta k} \mu(2^kB)^{-1}.
\]

To prove \eqref{B} and \eqref{C} we only need to consider when ${r(B)}<1$, otherwise $\eta_k'=\eta_k''=0$. In that case, there exists a differential form $b$ such that $a=\cD^N b$, as in Definition~\ref{localmoleculedef}. Define $\tilde\psi(z)=z^N\psi(z)$, in which case $\tilde\psi\in\Psi_N(S_{\theta,\tilde r}^o)$, where $\tilde r\geq r$ was fixed previously so that $\tilde r/C_\cD C_{\theta,\tilde r}>\lambda$. Now choose $\ep$ in $(0,N-\kappa/2)$. Then, since $N>\kappa/2$, Lemma~\ref{ODexpHvNP} implies that
\begin{equation}\label{tpsiODexp}
\|\ca_E\tilde\psi_t(\cD)\ca_F\|
\lesssim \langle{t}/{\rho(E,F)}\rangle^{\frac{\kappa}{2}+\ep} e^{-\lambda\rho(E,F)}
\lesssim \langle{t}/{\rho(E,F)}\rangle^{\ep}
\end{equation}
for all closed subsets $E$ and $F$ of $M$.

To prove \eqref{B}, we only consider $k\geq 2$, since otherwise $\eta_k'=0$. For each $l\in\N$, define $J_l$ by
\[\|\eta_k'\psi_t(\cD)a\|_{L^2_\bullet}^2 \leq\sum_{l=0}^\infty \int_{{r(B)}}^{\langle2^{k-1}{r(B)}\rangle} \|\ca_k(B)\tilde\psi_t(\cD)\ca_l(B)b\|_2^2\ \frac{\d t}{t^{2N+1}} = \sum_{l=0}^\infty J_l.\]
The proof proceeds as for $I_l$ by using \eqref{tpsiODexp} instead of \eqref{psiODexp}. If $0\leq l\leq k-2$, then since $N-\kappa/2-\ep>0$ and ${r(B)}<1$, we have
\begin{align*}
J_l &\lesssim \int_{r(B)}^{1} \left(\frac{t}{2^kr(B)}\right)^{2(\frac{\kappa}{2}+\ep)} \frac{\d t}{t^{2N+1}} e^{-2\lambda(2^{k-1}-2^l){r(B)}} r(B)^{2N}e^{-q 2^lr(B)}2^{-2l}\mu(2^lB)^{-1}\\
&\lesssim {r(B)}^{2(N-\frac{\kappa}{2}-\ep)}\int_{r(B)}^{1} t^{-2(N-\frac{\kappa}{2}-\ep)} \frac{\d t}{t} 2^{-2l(\frac{\kappa}{2}+1)}2^{-2k(\frac{\kappa}{2}+\ep-\frac{\kappa}{2})}\mu(2^kB)^{-1}\\
&\lesssim 2^{-2l}2^{-2\ep k}\mu(2^kB)^{-1}.
\end{align*}
If $k-1\leq l\leq k+1$, then since ${r(B)}<1$, we have
\begin{align*}
J_l \leq {r(B)}^{-2N}\|\tilde\psi_t(\cD)\ca_l(B)b\|_{L^2_\bullet}^2
\lesssim e^{-q2^lr(B)}2^{-2l}\mu(2^lB)^{-1}
\lesssim 2^{-2l}\mu(2^kB)^{-1}.
\end{align*}
If $l\geq k+2$, then since $N-\ep>0$ and ${r(B)}<1$, we have
\begin{align*}
J_l &\lesssim \int_{r(B)}^{1} \left(\frac{t}{2^lr(B)}\right)^{2\ep} \frac{\d t}{t^{2N+1}} {r(B)}^{2N}2^{-2l}\mu(2^lB)^{-1}\\
&\leq {r(B)}^{2(N-\ep)}\int_{r(B)}^{1} t^{-2(N-\ep)} \frac{\d t}{t} 2^{-2l}2^{-2\ep k} \mu(2^kB)^{-1}\\
&\lesssim 2^{-2l}2^{-2\ep k} \mu(2^kB)^{-1}.
\end{align*}
Note that we needed $q\geq\lambda$ when $0\leq l\leq k+1$. This proves \eqref{B}, since now $\|\eta_k'\psi_t(\cD)a\|_{L^2_\bullet}^2 \leq \sum_{l=0}^\infty J_l \lesssim 2^{-2\ep k} \mu(2^kB)^{-1}$.

To prove \eqref{C}, we only consider $k\geq 1$ for which $2^{k-1}{r(B)}<1$, since otherwise $\eta_k''=0$. For each $l\in\N$, define $K_l$ by
\[\|\eta_k''\psi_t(\cD)a\|_{L^2_\bullet}^2 \leq\sum_{l=0}^\infty \int_{2^{k-1}{r(B)}}^{\langle2^kr(B)\rangle} \|\ca_{2^kB}\tilde\psi_t(\cD)\ca_l(B)b\|_2^2\ \frac{\d t}{t^{2N+1}} = \sum_{l=0}^\infty K_l.\]
The proof proceeds as for $J_l$. In fact, we only require the weaker estimate obtained by setting $\ep=0$ in \eqref{tpsiODexp}. If $0\leq l\leq k+2$, then $\mu(2^kB) \lesssim 2^{(k-l)\kappa} \mu(2^lB)$, since $2^{k-1}{r(B)}<1$, so we have
\[
K_l \lesssim (2^kr(B))^{-2N} \|\tilde\psi_t(\cD)\ca_l(B)b\|_{L^2_\bullet}^2
\leq 2^{-2l(\frac{\kappa}{2}+1)}2^{-2k(N-\frac{\kappa}{2})}\mu(2^kB)^{-1}.
\]
If $l\geq k+2$, then we have
\begin{align*}
K_l \lesssim 2^{-2l(\frac{\kappa}{2}+1)} \int_{2^{k-1}{r(B)}}^{2^kr(B)} \left(\frac{{r(B)}}{t}\right)^{2(N-\frac{\kappa}{2})} \frac{\d t}{t} \mu(2^lB)^{-1}
\leq 2^{-2l}2^{-2k(N-\frac{\kappa}{2})}\mu(2^kB)^{-1}.
\end{align*}
Note that we did not require $q\geq\lambda$ here. This proves \eqref{C}, since $N>\kappa/2$ and now $\|\eta_k''\psi_t(\cD)a\|_{L^2_\bullet}^2 \leq \sum_{l=0}^\infty K_l \lesssim 2^{-2(N-\frac{\kappa}{2})k} \mu(2^kB)^{-1}$.
\end{proof}

\begin{Lem} \label{localmolecularlemma2.5}
Let $\theta\in(\omega,\tfrac{\pi}{2})$, $r>R$ and $\beta>\kappa/2$ such that $r/C_\cD C_{\theta,r}>\lambda/2$. For each $N\in\N$, $N>\kappa/2$ and $q\geq\lambda$, there exist $c>0$ and $\phi\in\Phi^\beta(S_{\theta,r}^o)$ such that $\|\phi(\cD)a\|_{L^1_\mathscr{Q}} \leq c$ for all $a$ that are $h^1_\cD$-molecules of $\text{type }(N,q)$.
\end{Lem}
\begin{proof}
Let $a$ be an $h^1_\cD$-molecule of $\text{type }(N,q)$. There exists a ball $B$ in $M$ of radius $r(B)>0$ such that the requirements of Definition~\ref{localmoleculedef} are satisfied. Let $B^*=(1/\langle {r(B)}\rangle)B$, so the radius $r(B^*)\geq1$.

Suppose that there exist $\phi\in\Phi^\beta(S_{\theta,r}^o)$ and $c,\delta>0$, all of which do not depend on $a$, such that
\begin{equation}\label{D-E}
\|\ca_k(B^*)\phi(\cD)a\|_2 \leq c2^{-\delta k} \mu(2^kB^*)^{-\frac{1}{2}}
\end{equation}
for all $k\geq0$. In that case, each $(2^{\delta k}/c)\ca_k(B^*)\phi(\cD)a$ is an $L^1_{\mathscr{Q}}$-atom, and since
\[
\phi(\cD)a=\sum_{k=0}^\infty \ca_k(B^*)\phi(\cD)a
\]
almost everywhere on $M$, Theorem~\ref{mainLQatomic} implies that $\|\phi(\cD)a\|_{L^1_{\mathscr{Q}}} \lesssim c\sum_{k=0}^\infty 2^{-\delta k}$. Therefore, it suffices to prove \eqref{D-E}.

To prove \eqref{D-E}, choose $\tilde r$ so that  $\tilde r\geq r$ and $\tilde r/C_\cD C_{\theta,\tilde r}>\lambda$. Also, choose $\delta$ in $(0,\tilde r/C_\cD C_{\theta,\tilde r}-\lambda)$ and choose $\phi$ in $\Phi^{\beta+N}(S_{\theta,\tilde r}^o)$, in which case $\phi\in\Phi^{\beta}(S_{\theta,r}^o)$. Lemma~\ref{ODexpHvNP} then implies that
\begin{equation}\label{phiODexp}
\|\ca_E\phi(\cD)\ca_F\| \lesssim e^{-(\lambda+\delta)\rho(E,F)} \lesssim \langle 1/\rho(E,F)\rangle^{\frac{\kappa}{2}+1} e^{-\lambda\rho(E,F)}
\leq \langle1/\rho(E,F)\rangle
\end{equation}
for all closed subsets $E$ and $F$ of $M$.

We now prove \eqref{D-E} when ${r(B)}\geq1$, in which case ${B^*=B}$. If $k=0$, then by \eqref{eq: h2equiv1} and \eqref{N-molFullL2BddRemLocal(a)} we have
\[\|\ca_0(B)\phi(\cD)a\|_2 \leq \|\phi(\cD)a\|_2 \lesssim \|a\|_2 \lesssim \mu(B)^{-\frac{1}{2}}.\]
Now consider $k\geq 1$ and for each $l\in\N$, define $I_l'$ by
\[\|\ca_k(B)\phi(\cD)a\|_2^2 \leq \sum_{l=0}^\infty \|\ca_k(B)\phi(\cD)\ca_l(B)a\|_2^2 = \sum_{l=0}^\infty I_l'.\]
The proof proceeds as for $I_l$ in Lemma~\ref{localmolecularlemma2} by using \eqref{phiODexp} instead of \eqref{psiODexp}. If $0\leq l\leq k-2$, then since ${r(B)}\geq1$, we have
\begin{align*}
I_l' \lesssim \left(\frac{1}{2^kr(B)}\right)^{2(\frac{\kappa}{2}+1)} e^{-2\lambda(2^{k-1}-2^l){r(B)}} e^{-q 2^lr}2^{-2l}\mu(2^lB)^{-1}
\lesssim 2^{-2l}2^{-2k}\mu(2^kB)^{-1}.
\end{align*}
If $k-1\leq l\leq k+1$, then we have
\[
I_l' \leq \|\phi(\cD)\ca_l(B)a\|_2^2
\lesssim e^{-q2^l {r(B)}}2^{-2l}\mu(2^lB)^{-1}
\lesssim 2^{-2l}\mu(2^kB)^{-1}.
\]
If $l\geq k+2$, then since ${r(B)}\geq1$, we have
\begin{align*}
I_l' \lesssim \left(\frac{1}{2^lr(B)}\right)^{2} 2^{-2l}\mu(2^lB)^{-1}
\lesssim 2^{-2l}2^{-2k} \mu(2^kB)^{-1}.
\end{align*}
Note that we needed $q\geq\lambda$ when $0\leq l\leq k+1$. This proves \eqref{D-E} when $r(B)\geq1$, since now $\|\ca_k(B)\phi(\cD)a\|_2^2 \leq \sum_{l=0}^\infty I_l' \lesssim 2^{-2k} \mu(2^kB)^{-1}$.

If ${r(B)}<1$, then $r(B^*)=1$ and there exists a differential form $b$ such that $a=\cD^N b$, as in Definition~\ref{localmoleculedef}. Define $\psi(z)=z^N\phi(z)$, in which case $\psi \in \Psi_N(S_{\theta,\tilde r}^o)$, where $\tilde r\geq r$ was fixed previously so that $\tilde r/C_\cD C_{\theta,\tilde r}>\lambda$. Now choose $\ep$ in ${(0,N-\kappa/2)}$. Then, since $N>\kappa/2$, Lemma~\ref{ODexpHvNP} implies that
\begin{equation}\label{tphiODexp}
\|\ca_E\psi(\cD)\ca_F\| \lesssim \langle1/\rho(E,F)\rangle^{\frac{\kappa}{2}+\ep}e^{-\lambda\rho(E,F)} \leq \langle1/\rho(E,F)\rangle^{\ep}
\end{equation}
for all closed subsets $E$ and $F$ of $M$.

We now prove \eqref{D-E} when ${r(B)}<1$. If $k=0$, then by \eqref{eq: h2equiv1} and \eqref{N-molFullL2BddRemLocal(b)}, and since ${r(B)}<1$ and $N>\kappa/2$, we have
\[
\|\ca_0(B^*)\phi(\cD)a\|_2 \leq \|\psi(\cD)b\|_2 \lesssim {r(B)}^N\mu(B)^{-\frac{1}{2}} \lesssim {r(B)}^{N-\frac{\kappa}{2}}\mu(B^*)^{-\frac{1}{2}}\leq\mu(B^*)^{-\frac{1}{2}}.
\]
Now consider $k\geq 1$. For each $l\in\N$, define $I_l''$ by
\[
\|\ca_k(B^*)\phi(\cD)a\|_2^2 \leq\sum_{l=0}^\infty \|\ca_k(B^*)\psi(\cD)\ca_l(B)b\|_2^2 = \sum_{l=0}^\infty I_l''.
\]
If $1\leq 2^l< 2^{k-1}/r(B)$, then
\[
\rho(2^kB^*\backslash 2^{k-1}B^*,2^lB\backslash 2^{l-1}B) =2^{k-1}-2^lr(B)\gtrsim 2^k
\]
and $\mu(2^kB^*) \leq (2^{k-l}/{r(B)})^{\kappa}e^{\lambda(2^{k-l}/{r(B)}-1)2^lr(B)}\mu(2^lB)$,
so that by \eqref{tphiODexp}, and since $r(B)<1$ and $N>\kappa/2$, we have
\begin{align*}
I_l'' &\lesssim \left(\frac{1}{2^k}\right)^{2(\frac{\kappa}{2}+\ep)} e^{-2\lambda(2^{k-1}-2^lr(B))} {r(B)}^{2N} e^{-q2^lr(B)}2^{-2l}\mu(2^lB)^{-1}\\
&\lesssim 2^{-2l(\frac{\kappa}{2}+1)}2^{-2k(\frac{\kappa}{2}+\ep-\frac{\kappa}{2})} {r(B)}^{2(N-\frac{\kappa}{2})}e^{\lambda(-2^{k}+2^{l+1}{r(B)}+2^k-2^lr(B)-2^lr(B))} \mu(2^kB^*)^{-1}\\
&\lesssim 2^{-2l}2^{-2\ep k} \mu(2^kB^*)^{-1}.
\end{align*}
If $2^{k-1}/r(B) \leq 2^l \leq 2^{k+1}/r(B)$, then $\mu(2^kB^*)\lesssim e^{\lambda2^lr}\mu(2^lB)$, and since ${r(B)}<1$, we have
\[
I_l'' \leq \|\psi(\cD)\ca_l(B)b\|_2^2
\lesssim {r(B)}^{2N}e^{-q2^lr(B)}2^{-2l}\mu(2^lB)^{-1}
\lesssim 2^{-2l}\mu(2^kB^*)^{-1}.
\]
If $2^l>2^{k+1}/r(B)$, then
\[
\rho(2^kB^*\backslash 2^{k-1}B^*,2^lB\backslash 2^{l-1}B) =2^{l-1}{r(B)}-2^k\gtrsim 2^l
\]
and $\mu(2^kB^*) \leq \mu(2^lB)$, so that by \eqref{tphiODexp}, and since ${r(B)}<1$, we have
\begin{align*}
I_l'' \lesssim \left(\frac{1}{2^l}\right)^{2\ep} {r(B)}^{2N} 2^{-2l}\mu(2^lB)^{-1}
\lesssim 2^{-2l}2^{-2\ep k} \mu(2^kB^*)^{-1}.
\end{align*}
Note that we needed $q\geq\lambda$ when $1\leq 2^l<2^{k+1}/r(B)$. This proves \eqref{D-E} when ${r(B)}<1$, since now $\|\ca_k(B^*)\phi(\cD)a\|_2^2 \leq \sum_{l=0}^\infty I_l'' \lesssim 2^{-2\ep k} \mu(2^kB^*)^{-1}$.
\end{proof}

\subsection{Local Riesz Transforms and Holomorphic Functional Calculi}\label{SubSection1stmainresult}
We now prove the principal result of the paper, which is the local analogue of Theorem~5.11 in \cite{AMcR}.

\begin{Thm}\label{mainresult}
Let $\kappa,\lambda\geq0$ and suppose that $M$ is a complete Riemannian manifold satisfying \eqref{ELD}. Let $\omega\in[0,\pi/2)$ and $R\geq0$ and suppose that $\cD$ is a closed densely-defined operator on $L^2(\wedge T^*M)$ of $\text{type }S_{\omega,R}$ satisfying hypotheses (H1-3) from Section~\ref{SectionODE}. Let $\theta\in(\omega,\pi/2)$ and $r>R$ such that $r/C_\cD C_{\theta,r}>\lambda/2$, where $C_{\theta,r}$ is from~(H1) and $C_\cD$ is from (H3). Then for all $f\in H^\infty(S_{\theta,r}^o)$, the operator $f(\cD)$ on $L^2(\wedge T^*M)$ has a bounded extension to $h^p_{\cD}(\wedge T^*M)$ such that
\[
\|f(\cD)u\|_{h^p_\cD}\lesssim \|f\|_\infty \|u\|_{h^p_\cD}
\]
for all $u\in h^p_\cD$ and $p\in[1,\infty]$.
\end{Thm}
\begin{proof}
If $u\in h^1_\cD\cap L^2$, then Proposition~\ref{prop: normequiv} gives $U \in t^1 \oplus L^1_\mathscr{Q}$ with $\cS^\cD_{\eta,\varphi}U=u$ and $\|U\|_{t^1\oplus L^1_\mathscr{Q}} \leq 2\|u\|_{h^1_\cD}$. Therefore, by Theorem~\ref{locMainEst} we have
\begin{align*}
    \|f(\cD)u\|_{h^1_{\cD}}
    = \|\cQ^\cD_{\eta,\varphi}f(\cD)\cS^\cD_{\eta,\varphi}U\|_{t^1\oplus L^1_\mathscr{Q}}
    \lesssim \|f\|_\infty\|U\|_{t^1\oplus L^1_{\mathscr Q}}
    \lesssim \|f\|_\infty\|u\|_{h^1_{\cD}}
\end{align*}
for all $u\in h^1_{\cD}\cap L^2$, and since $h^1_\cD\cap L^2$ is dense in $h^1_\cD$ by Corollary \ref{Cor: hphqDensehp}, $f(\cD)$ has a bounded extension to $h^1_\cD$. The same proof with $\tilde t^\infty \oplus \tilde L^\infty_\mathscr{Q}$ instead of $t^1\oplus L^1_\mathscr{Q}$ shows that $f(\cD)$ has a bounded extension to $h^\infty_\cD$. These extensions coincide on $h^1_\cD\cap h^\infty_\cD$, since $h^1_\cD\cap h^\infty_\cD\subseteq h^2_\cD=L^2$ is a consequence of the interpolation of the local Hardy spaces in Theorem~\ref{Thm: Interpolation}. Therefore, the required extension exists by interpolation.
\end{proof}

Theorem~\ref{Thm: Intro.FuncCalc} follows from this result in the case of the Hodge--Dirac operator by Example~\ref{Eg: HodgeDirac}, which allows us to prove Corollary~\ref{Cor: Intro.Riesz}.

\begin{proof}[Proof of Corollary~\ref{Cor: Intro.Riesz}]
It was shown in Example~\ref{Eg: HodgeDirac} that $D$ satisfies (H1-3) with $\omega=0$, $R=0$, $C_\cD =1$ and $C_{\theta,r}=1/\sin \theta$ for all $\theta\in(0,\pi/2)$ and $r>0$. Therefore, Corollary~\ref{Cor: Intro.Riesz} follows from Theorem~\ref{mainresult} by choosing $\theta$ in $(0,\pi/2)$ such that $\lambda/2\sin\theta<\sqrt a$, choosing $r$ in $(\lambda/2\sin \theta,\sqrt a)$ and defining the holomorphic function $f(z)= z(z^2+a)^{-1/2}$ for all $z\in S^o_{\theta,r}$.
\end{proof}

\appendix
\section{The Atomic Characterisation of $t^1(X\times(0,1])$}\label{AppendixA}
The proof of Theorem~\ref{maintentatomic} is an adaptation of \cite{Russ}, which in turn is based on the original proof in \cite{CMS}. For this, we introduce the notion of local $\gamma$-density.

\begin{Def} \label{*setsdefloc}
Let $X$ be a locally doubling metric measure space. Let $F$ be a closed subset of $X$ with $O=\comp{F}$ and $\mu(O)<~\infty$. For each $\gamma\in(0,1)$, the points of \textit{local $\gamma$-density} with respect to $F$ are the elements of the set
\[F^{\gamma}_{\text{loc}}=\left\{x\in X\ \left|\ \inf_{0<r\leq1}\frac{\mu(F\cap B(x,r))}{V(x,r)}\geq\gamma \right.\right\}.\]
The complement of this set is denoted by $O^{\gamma}_\text{loc}=\comp{(F^{\gamma}_\text{loc})}$.
\end{Def}

Local $\gamma$-density can be understood in terms of the local maximal operator $\mathcal M_{\text{loc}}$ from Section \ref{SectionLocalisation}. For each $\gamma\in(0,1)$, the following hold:
\begin{enumerate}
\item $F^{\gamma}_\text{loc}$ is closed;
\item $F^{\gamma}_\text{loc}\subseteq F$;
\item $O^{\gamma}_\text{loc}=\{x\in X\ |\ {\mathcal M}_{\text{loc}}\ca_O(x)>1-\gamma\}$;
\item $\mu(O^{\gamma}_\text{loc})\lesssim \mu(O)$.
\end{enumerate}
The proof of these properties relies on Proposition~\ref{CW2.1loc} and is left to the reader.

The proof of Theorem~\ref{maintentatomic} also requires the following lemma, which is adapted from Lemma~2.1 in \cite{Russ}.

\begin{Lem}\label{CMSlemma2+Russlemma2.1loc}
Let $X$ be a locally doubling metric measure space. Let $F$ be a closed subset of $X$ and let $\Phi$ be a nonnegative measurable function on $X\times(0,1]$. For each $\eta\in(0,1)$, there exists $\gamma\in(0,1)$ such that
\[\iint_{R_{1-\eta}^1(F^\gamma_{\text{loc}})} \Phi(y,t)V(y,t)\ \d\mu(y)\d t \lesssim \int_F\iint_{\Gamma^1(x)}\Phi(y,t)\ \d\mu(y)\d t \d\mu(x),\]
where $R_{1-\eta}^1$ and $\Gamma^1$ are defined in Section~\ref{SectionLocalTent}.
\begin{proof}
Fix $\eta\in(0,1)$ and let $\gamma\in(0,1)$ to be chosen later. For each $(y,t)$ in $R_{1-\eta}^1(F^\gamma_{\text{loc}})$, choose $\xi\in F^\gamma_{\text{loc}}$ such that $(y,t)\in\Gamma_{1-\eta}^1(\xi)$. We then have
\[
\mu(F\cap B(\xi,t))\geq \gamma V(\xi,t).
\]
Also, the condition $\rho(\xi,y)<(1-\eta)t$ implies that $B(\xi,\eta t)\subseteq B(y,t)$. Therefore, we have $B(\xi,\eta t)\subseteq B(\xi,t) \cap B(y,t)$ and by Proposition~\ref{PropkappaLD} there exists $c_\eta\in(0,1)$, depending on $\eta$, such that
\[
c_\eta V(\xi,t) \leq V(\xi,\eta t) \leq \mu(B(\xi,t) \cap B(y,t)).
\]
Now choose $\gamma\in(1-c_\eta,1)$. The above inequalities show that there exists $C_{\eta,\gamma}>0$, depending on $\eta$ and the choice of $\gamma$, such that \begin{align*}
\mu(F\cap B(y,t)) &\geq \mu(F\cap B(\xi, t))-\mu(B(\xi, t)\cap\comp{B}(y,t)) \\
&\geq (\gamma-(1-c_\eta)) V(\xi,t) \\
&\geq C_{\eta,\gamma} V(y,t),
\end{align*}
where the final inequality follows from \eqref{LD} and $B(y,t)\subseteq B(\xi,2t)$.

Using the above inequality and Fubini's theorem we obtain
\begin{align*}
\iint_{R_{1-\eta}^1(F^\gamma_{\text{loc}})}& \Phi(y,t)V(y,t)\ \d\mu(y)\d t \\
&\lesssim \iint_{R_{1-\eta}^1(F^\gamma_{\text{loc}})} \Phi(y,t) \mu(F\cap B(y,t))\ \d\mu(y)\d t \\
&\leq \iint_{R_1^1(F)}\int_{F\cap B(y,t)}\Phi(y,t)\ \d\mu(x)\d\mu(y)\d t \\
&\leq \int_F\iint_{\Gamma^1(x)}\Phi(y,t)\ \d\mu(y)\d t\d\mu(x).
\end{align*}
\end{proof}
\end{Lem}

We now complete the proof of the atomic characterisation of $t^1$.

\begin{proof}[Proof of Theorem~\ref{maintentatomic}]
Let $f\in t^{1}$ and for each $k\in\mathbb{Z}$, define
\[O_k = \{x\in X\ |\ \mathcal{A}_{\text{loc}}f(x)>2^k\}\]
and $F_k=\comp{O}_k$. The lower semicontinuity of $\mathcal{A}_{\text{loc}}f$ ensures that  $O_k$ is open. We also have $\mu(O_k) \leq 2^{-k}\|f\|_{t^{1}}<\infty$.

Let $\eta\in(0,1)$ to be chosen later and let $\gamma\in(0,1)$ be the constant, which depends on $\eta$, from Lemma~\ref{CMSlemma2+Russlemma2.1loc}. Let $F_k^*$ denote the set $(F_k)^\gamma_{\text{loc}}$ from Definition~\ref{*setsdefloc} and let $O_k^*=\comp{(F_k^*)}$. We then have $O_k\subseteq O_k^*$ and $\mu(O_k^*)\lesssim \mu(O_k)$.

First, we establish that $f$ is supported in $\bigcup_{k\in\mathbb{Z}} T_{1-\eta}^1(O_k^*)$. For each $k\in\mathbb{Z}$, we apply Lemma~\ref{CMSlemma2+Russlemma2.1loc} with  $\Phi(y,t)={|f(y,t)|^2}({V(y,t)t})^{-1}$ and $F=F_{k}$ to obtain
\begin{align*}
\iint_{\comp{\left(\bigcup_{j\in\mathbb{Z}} T_{1-\eta}^1(O_j^*)\right)}}& |f(y,t)|^2\ \d\mu(y)\frac{\d t}{t} =\iint_{\bigcap_{j\in\mathbb{Z}} R^1_{1-\eta}(F_j^*)} |f(y,t)|^2\ \d\mu(y)\frac{\d t}{t} \\
&\leq \iint_{R_{1-\eta}^1(F_{k}^*)} |f(y,t)|^2\ \d\mu(y)\frac{\d t}{t} \\
&\lesssim \int_{F_{k}}\iint_{\Gamma^1(x)} |f(y,t)|^2\ \frac{\d\mu(y)}{V(y,t)}\frac{\d t}{t}\d\mu(x) \\
&\lesssim \int \ca_{F_{k}}(x) (\mathcal{A}_{\text{loc}}f(x))^2\ \d\mu(x),
\end{align*}
where the final inequality follows from \eqref{LD}, since  if $(y,t)\in\Gamma^1(x)$, then ${t\leq 1}$ and $B(x,t)\subseteq B(y,2t)$. If $k$ is a negative integer, then pointwise on $X$ we have $\ca_{F_{k}} (\mathcal{A}_{\text{loc}}f)^2 \leq \mathcal{A}_{\text{loc}}f$ and $\lim_{k\rightarrow-\infty} \ca_{F_{k}}(\mathcal{A}_{\text{loc}}f)^2=0$, where $\mathcal{A}_{\text{loc}}f\in L^1(X)$. Therefore, by dominated convergence we have
\[\lim_{k\rightarrow-\infty} \int \ca_{F_{k}}(x) (\mathcal{A}_{\text{loc}}f(x))^2\ \d\mu(x) = 0,\]
which implies that $f=0$ almost everywhere on $\comp{\left(\bigcup_{j\in\mathbb{Z}} T_{1-\eta}^1(O_j^*)\right)}$, as required.

Now we decompose $f$ into $t^1$-atoms. For each $k\in\mathbb{Z}$, apply Proposition~\ref{Aimar2.13 + MS2.16} with $O=O_k^*$ and $h>0$ to be chosen later. This gives a sequence of pairwise disjoint balls $(B_j^k)_{j\in I_k}$, where each ball $B_j^k=B(x_j^k,r_j^k)$ has radius $r_j^k=\frac{1}{8}\min(\rho(x_j^k,\comp{O}_k^*),h)$ and $I_k$ is some indexing set. It also gives a sequence of nonnegative functions $(\phi_j^k)_{j\in I_k}$ supported in $\tilde B_j^k=4B_j^k$ such that $\sum_{j\in I_k}\phi_j^k=\ca_{O_k^*}$. For each $(y,t)$ in $X\times (0,1]$, we have
\[\ca_{T_{1-\eta}^1(O_k^*)\setminus T_{1-\eta}^1(O_{k+1}^*)}(y,t) = \sum_{j\in I_k} \phi_j^k(y) \ca_{T_{1-\eta}^1(O_k^*)\setminus T_{1-\eta}^1(O_{k+1}^*)}(y,t),\]
since either $(y,t)\in T_{1-\eta}^1(O_k^*)\setminus T_{1-\eta}^1(O_{k+1}^*)$, in which case $y\in O_k^*$ and we have $\sum_{j\in I_k} \phi_j^k(y)=1$, or both sides of the equation are zero.
Given that $f$ is supported in $\bigcup_{k\in\mathbb{Z}} T_{1-\eta}^1(O_k^*)$, the following holds for almost every $(y,t)\in X\times (0,1]$:
\begin{align}\label{eq: t1atomdecomp}\begin{split}
f(y,t) &= f(y,t) \sum_{k\in\mathbb{Z}} \ca_{T_{1-\eta}^1(O_k^*)\setminus T_{1-\eta}^1(O_{k+1}^*)}(y,t)\\
&= \sum_{k\in\mathbb{Z}}\sum_{j\in I_k} f(y,t) \phi_j^k(y) \ca_{T_{1-\eta}^1(O_k^*)\setminus T_{1-\eta}^1(O_{k+1}^*)}(y,t) \\
&= \sum_{k\in\mathbb{Z}}\sum_{j\in I_k} \lambda_j^k a_j^k(y,t),
\end{split}\end{align}
where
\begin{align*}
&a_j^k(y,t) = \frac{1}{\lambda_j^k} f(y,t)\phi_j^k(y) \ca_{T_{1-\eta}^1(O_k^*)\setminus T_{1-\eta}^1(O_{k+1}^*)}(y,t),\\
&\lambda_j^k =\left(\mu(\alpha B_j^k)\iint |f(y,t)|^2 \phi_j^k(y)^2 \ca_{T_{1-\eta}^1(O_{k}^*)\setminus T_{1-\eta}^1(O_{k+1}^*)}(y,t) \ \d\mu(y)\frac{\d t}{t}\right)^{\frac{1}{2}}
\end{align*}
and $\alpha>0$ will be chosen later.

Given that $f\in t^1$, the series in \eqref{eq: t1atomdecomp} also converges to $f$ in $t^1$ by dominated convergence. The same reasoning shows that if $f\in t^1\cap t^p$ for some $p\in(1,\infty)$, then the series also converges to $f$ in $t^p$.
It remains to choose the constants $\eta\in(0,1)$, $h>0$ and $\al>0$ so that \eqref{eq: t1atomdecomp} is the required atomic decomposition.

First, consider the support of $a_j^k$. If $(y,t)\in\supp a_j^k$, then $y\in\supp \phi_j^k\subseteq4B_j^k$ and we have
\begin{equation}\label{01}
\rho(y,z) \geq \rho(x_j^k,z) -\rho(x_j^k,y) \geq (\alpha-4)r_j^k
\end{equation}
for all $z\in\comp{(\alpha B_j^k)}$. We also have $\rho(y,\comp{O}_k^*)\geq (1-\eta)t$, since $(y,t)\in T^1_{1-\eta}(O_k^*)$. Now consider two cases: (1) If $8r_j^k=\min(\rho(x_j^k,\comp{O_k^*}),h)=\rho(x_j^k,\comp{O_k^*})$, then
\[
(1-\eta)t \leq \rho(y,\comp{O_k^*}) \leq \rho(y,x_j^k) + \rho(x_j^k,\comp{O_k^*}) \leq 12r_j^k,
\]
so by \eqref{01} we have
\begin{equation}\label{02}
\rho(y,z) \geq (\al-4) {(1-\eta)t}/{12}
\end{equation}
for all $z\in\comp{(\alpha B_j^k)}$; (2) If $8r_j^k=\min(\rho(x_j^k,\comp{O_k^*}),h)=h$, then
\begin{equation}\label{03}
\rho(y,z) \geq (\al-4){h}/{8}
\end{equation}
for all $z\in\comp{(\alpha B_j^k)}$.

Now choose $\eta\in(0,1)$, $h>0$ and $\al>0$ such that
\[(\al-4)(1-\eta)/12\geq1,\quad (\al-4)h/8\geq1\quad\text{ and }\quad\al h/8\leq2.\]
For example, set $\eta=1/4$, $h=1/2$ and $\al=20$. It then follows from \eqref{02} and \eqref{03} that $\rho(y,\comp{(\al B_j^k)}) \geq t$ and so $\supp a_j^k \subseteq T^1(\al B_j^k)$, where the radius of $\al B_j^k$ is $\al r_j^k \leq \al h/8\leq 2$. Also, it is immediate that $\|a_j^k\|_{L^2_\bullet} = \mu(\al B_j^k)^{-1/2}$ and thus $a_j^k$ is a $t^{1}$-atom.

It remains to prove the norm equivalence. Using the support condition just proved and applying Lemma~\ref{CMSlemma2+Russlemma2.1loc} with $F=F_k$ and \[\Phi(y,t)=\ca_{T^1(\al B_j^k)}(y,t){|f(y,t)|^2}{(V(y,t)t)^{-1}}\]
gives
\begin{align*}
(\lambda_j^{k})^2\mu(\al B_j^k)^{-1} &\leq \iint_{T^1(\al B_j^k) \cap \comp{[T_{1-\eta}^1(O_{k+1}^*)]}} |f(y,t)|^2\ \d\mu(y)\frac{\d t}{t} \\
&= \iint_{R^1_{1-\eta}(F_{k+1}^*)} \ca_{T^1(\al B_j^k)}(y,t) |f(y,t)|^2\ \d\mu(y)\frac{\d t}{t} \\
&\lesssim \int_{F_{k+1}}\iint_{\Gamma^1(x)}\ca_{T^1(\al B_j^k)}(y,t)|f(y,t)|^2\ \frac{\d\mu(y)}{V(y,t)}\frac{\d t}{t}\d\mu(x) \\
&\lesssim \int_{\comp{O_{k+1}}\cap \al B_j^k}(\mathcal{A}_{\text{loc}}f(x))^2\ \d\mu(x) \\
&\lesssim 2^{2k}\mu(\al B_j^k).
\end{align*}
Furthermore, by \eqref{LD} we have $\lambda_j^k \lesssim 2^k\mu(B_j^k)$, and since for each $k\in\Z$ the balls $(B_j^k)_j$ are pairwise disjoint and contained in $O_k^*$, we obtain
\begin{align*}
\sum_{k\in\mathbb{Z}} \sum_{j\in I_k} |\lambda_j^k|
&\leq \sum_{k\in\mathbb{Z}} 2^k \mu(O_k^*) \\
&\lesssim \sum_{k\in\mathbb{Z}} 2^k \mu(O_k) \\
&= \sum_{k\in\mathbb{Z}} 2\int_{2^{k-1}}^{2^k} \mu(\{x\in X\ |\ \mathcal{A}_{\text{loc}}f(x)>2^k\})\ \d t\\
&\lesssim \sum_{k\in\mathbb{Z}} \int_{2^{k-1}}^{2^k} \mu(\{x\in X\ |\ \mathcal{A}_{\text{loc}}f(x)>t\})\ \d t \\
&= \|f\|_{t^{1}},
\end{align*}
which completes the proof.
\end{proof}

\begin{Rem}\label{RemFinalTent}
If $b>1$, then a judicious choice of $\eta\in(0,1)$, $h>0$ and $\al>0$ in the proof of Theorem \ref{maintentatomic} allows us to characterise $f\in t^{1}$ in terms $t^{1}$-atoms supported on truncated tents $T^{1}(B)$ over balls $B$ with radius $r(B)\leq b$. The constants in the norm equivalence $\eqsim$ then depend on $b$ and, as we may expect, become unbounded as $b$ approaches 1.
\end{Rem}

\section*{Acknowledgements}
This work was conducted at the Centre for Mathematics and its Applications at the Australian National University, and at the Dipartimento di Matematica dell'Universit\`{a} degli Studi di Genova. The authors acknowledge support from these institutions. Carbonaro was also supported by PRIN 2007 ``Analisi Armonica'', McIntosh by the Australian Government through the Australian Research Council, and Morris by the Australian Government through an Australian Postgraduate Award.

A part of this work was conducted during the first author's visits to the Centre for Mathematics and its Applications at the Australian National University, and during the third author's visit to the Dipartimento di Matematica dell'Universit\`{a} degli Studi di Genova. The authors thank the faculty and staff of those institutions for their warm hospitality. The authors would also like to thank Pascal Auscher, Andreas Axelsson, Andrew Hassell, Brian Jefferies, Pierre Portal, Adam Rennie and Robert Taggart for helpful conversations and suggestions.

This work was first presented by the third author at the 8th International Conference on Harmonic Analysis and Partial Differential Equations, 16-20 June 2008, El Escorial, Spain.

\bibliographystyle{amsplain}
\providecommand{\bysame}{\leavevmode\hbox to3em{\hrulefill}\thinspace}

\end{document}